\documentclass[11pt]{amsart}
\usepackage[left=1.5in, right=1.5in]{geometry}
\usepackage[latin1]{inputenc}
\usepackage{amsfonts}
\usepackage[toc,page,title,titletoc,header]{appendix}
\usepackage{graphicx,psfrag,epsfig,multirow,caption,subcaption,color}
\usepackage{amssymb,amsmath,amscd,amsthm,amssymb,verbatim,hyperref,setspace}
\usepackage{bm}
\numberwithin{equation}{section}
\usepackage{mathrsfs,wrapfig}
\usepackage{indentfirst}
\usepackage{extarrows}
\newtheorem{theo}{Theorem}[section]
\newtheorem{lem}[theo]{Lemma}
\newtheorem{pro}[theo]{Proposition}
\newtheorem{cor}[theo]{Corollary}
\newtheorem{defi}[theo]{Definition}

\def\eps{\epsilon}

\def\bmt{\left[\begin{array}}
\def\emt{\end{array}\right]}

\title[Normally Hyperbolic Invariant Cylinder]{Normally Hyperbolic Invariant Cylinders Passing Through Multiple Resonance}
\author{Chong-Qing Cheng \& Min Zhou}
\address{Department of mathematics, Nanjing Univerisity, Nanjing 210093, China}
\email{chengcq@nju.edu.cn, minzhou@nju.edu.cn}
\begin{document}
\maketitle
\begin{abstract}
We study the continuation of periodic orbits from various compound of homoclinics in classical system. Together with the homoclinics, the periodic orbits make up a $C^1$-smooth, normally hyperbolic invariant cylinder with holes. It plays a key role to cross multiple resonant point.
\end{abstract}
\renewcommand\contentsname{Index}

\section{Introduction}
\setcounter{equation}{0}
Given an autonomous Hamiltonian, if a hyperbolic periodic orbit exists in an energy level set, the implicit function theorem implies a continuation of periodic orbits nearby, which make up a normally hyperbolic invariant cylinder (NHIC). So, it is natural to ask whether there exists a NHIC extending from the orbits homoclinic to a fixed point. In this paper, we study the problem for the classical system
\begin{equation}\label{eq1.1}
H(x,y)=\frac 12\langle Ay,y\rangle-V(x),\qquad z=(x,y)\in\mathbb{T}^n\times\mathbb{R}^n,
\end{equation}
where the matrix $A$ is positive definite, the smooth potential $V$ attains its minimum at a point $x_0$ only. In this case, $z_0=(x_0,0)$ is a fixed point of the Hamiltonian flow $\Phi_H^t$ and there exist some orbits homoclinic to the fixed point \cite{Bo}. Be aware that the system admits a symmetry $\mathbf{s}:(x,y)\to(x,-y)$, we see that if $z^+(t)=(x^+(t),y^+(t))$ is an orbit, $z^-(t)=\mathbf{s} z^+(t)=(x^+(-t),-y^+(-t))$ is also an orbit. Hence, non-shrinkable homoclinic orbits emerge paired.

To formulate our result, by a translation of variables $x\to x-x_0$ and $V\to V-V(x_0)$ we assume $x_0=0$, $V(0)=0$ and the following conditions:

({\bf H1}), {\it the Hessian matrix of $V$ at $x=0$ is positive definite. The $2n$ eigenvalues of $J\mathrm{diag}(-\partial ^2V(0),A)$ are all different, where $J$ denotes the standard symplectic matrix,
$$
-\lambda_n<\cdots<-\lambda_2<-\lambda_1<0<\lambda_1<\lambda_2<\cdots<\lambda_n,
$$
$(\lambda_1,\cdots,\lambda_n)$ is non-resonant and $V\in C^{2\kappa+1}$ with $(\kappa-1)\lambda_1>\lambda_n$. Let $\Xi_i^{+}=(\Xi_{i,x},\Xi_{i,y})$ be the eigenvector for $\lambda_i$, where $\Xi_{i,x},\Xi_{i,y}$ denote the component for the coordinates $x$ and $y$ respectively, then $\Xi_i^{-}=(\Xi_{i,x},-\Xi_{i,y})$ is the eigenvector for $-\lambda_i$.}

({\bf H2}), {\it for a pair of homoclinic orbits $\{z^+(t),z^-(t)\}$, the curve $x^+(t)$ approaches the origin in the direction of $\Xi_{1,x}$, $x^-(t)$ approaches in the direction of $-\Xi_{1,x}$
$$
\lim_{t\to\pm\infty}\frac{\dot x^+(t)}{\|\dot x^+(t)\|}=\Xi_{1,x}, \qquad \lim_{t\to\pm\infty}\frac{\dot x^-(t)}{\|\dot x^-(t)\|}=-\Xi_{1,x}.
$$
The stable manifold $W^s$ intersects the unstable manifold $W^u$ transversally along the orbit $z^{\pm}(t)$ in the following sense that
\begin{equation}\label{eq1.2}
T_zW^s\oplus T_zW^u=T_zH^{-1}(0)
\end{equation}
holds any for $z\ne 0$ on the homoclinic curve.}

We study $k$ pairs of homoclinic orbits $\{z^\pm_1(t),\cdots,z^\pm_k(t)\}$. A periodic orbit $z^+(t)$ is said to {\it shadow the orbits $\{z^+_1(t),\cdots,z^+_k(t)\}$} if the period admits a partition $[0,T]=[0,t_1]\cup[t_1,t_2]\cup\cdots\cup[t_{k-1},T]$ such that $z^{+}(t)|_{[t_{i-1},t_i]}$ falls into a small neighborhood of $z^+_i(t)$. In this case, its $\mathbf{s}$-symmetric counterpart $z^-(t)=\mathbf{s}z^+(t)$ shadows the orbits $\{z^-_k(t),\cdots,z^-_1(t)\}$.  

To study the case $k\ge 2$, we work in the covering spaces $\bar\pi_h$: $\mathbb{R}^n \times\mathbb{R}^n\to\mathbb{T}^n_h\times\mathbb{R}^n$ and $\pi_h:\mathbb{T}^n_h \times\mathbb{R}^n\to\mathbb{T}^n\times\mathbb{R}^n$, where $\mathbb{T}^n_h=\{(x_1,x_2, \cdots,x_n)\in\mathbb{R}^n:x_i\ \mathrm{mod}\ h_i\in\mathbb{N}\backslash 0\}$. To decide the class $h$, we let $\bar z_1(t)$ be the lift of $z^+_1(t)$ to $\mathbb{R}^{2n}$ such that $\lim_{t\to-\infty}\bar z_1(t)=0$, then choose a lift $\bar z_2(t)$ of $z^+_2(t)$ with $\lim_{t\to -\infty}\bar z_2(t)=\lim_{t\to\infty}\bar z_1(t)$. In the way, we get successively a lift $\bar z_i(t)$ of $z^+_i(t)$ for each $i$. Let $\bar\Gamma$ be the closure of $\cup_{t\in\mathbb{R}}(\cup_{i\le k}\bar z_i(t))$, we construct a shift $\sigma\bar\Gamma$. A curve $\bar z'_1(t)\subset \sigma\bar\Gamma$ is the lift of $z^+_1(t)$ such that $\lim_{t\to-\infty}\bar z'_1(t)=\lim_{t\to\infty}\bar z_k(t)$. Other $\bar z'_i(t)$ is successively fixed. Let $\sigma\bar\Gamma$ be the closure of $\cup_{t\in\mathbb{R}}(\cup_i\bar z'_i(t))$.

({\bf H3}), {\it for $k$ pairs of homoclinic orbits $\{z^\pm_1(t),\cdots,z^\pm_k(t)\}$, there exists a non negative integer $\ell$ and a covering space $\bar\pi_h$: $\mathbb{R}^n\times\mathbb{R}^n\to\mathbb{T}^n_h\times\mathbb{R}^n$ such that $\bar\pi_h(\bar\Gamma\cup\sigma\bar\Gamma \cup\cdots\cup\sigma^\ell\bar\Gamma)$ is a closed curve without self-intersection.} 

\begin{theo}\label{mainresult}
Assume $k$ pairs of homoclinic orbits $\{z^\pm_1(t),\cdots,z^\pm_k(t)\}$ satisfying the hypotheses $(${\bf H1,H2,H3}$)$. Then, there exists a continuation of periodic orbits from the homoclinic orbits $\{z^\pm_1(t),\cdots,z^\pm_k(t)\}$. More precisely, some $E_0>0$ exists such that

1, for any $E\in(0,E_0]$ there exist unique periodic orbit $z^+_E(t)$ and its $\mathbf{s}$-symmetric orbit $z^-_E(t)=\mathbf{s}z^+_E(t)$ shadowing the orbits $\{z^+_1(t),\cdots,z^+_k(t)\}$ and $\{z^-_k(t),\cdots,z^-_1(t)\}$ respectively. As a set depending on $E$, $\cup_tz^{\pm}_E(t)$ approaches $\cup_i\Gamma^\pm_i$ in Hausdorff metric as $E\downarrow0$;

2, for any $E\in[-E_0,0)$ there exists a unique periodic orbit $z_{E,i}$ shadowing the orbits $\{z^+_i(t),z^-_{i}(t)\}$ for $i=1,\cdots,k$. As a set depending on $E$, $\cup_tz_{E,i}(t)$ approaches $\Gamma^+_i\cup\Gamma^-_{i}$ in Hausdorff metric as $E\uparrow 0$; 

Let $\Pi=\Pi^+\cup_{1\le i\le k}(\Pi^-_i\cup\Gamma^+_i\cup\Gamma^-_i)$ where $\Pi^+=\cup_{E>0}(\cup_tz^{+}_E(t)\cup z^{-}_E(t))$ and $\Pi^-_i=\cup_{E<0}\cup_t z_{E,i}(t)$. For $k=1$, it makes up a $C^1$-NHIC with one hole. For $k\ge 2$, each connected component in the pull-back $\pi_h^{-1}\Pi$ of $\Pi$ to $\mathbb{T}^n_h\times\mathbb{R}^n$ is a $C^1$-NHIC with $(\ell+1)k$ holes. The homoclinic orbits are contained inside of the manifold.
\end{theo}

\begin{figure}[htp]
  \centering
  \begin{subfigure}[b]{0.45\textwidth}
                \includegraphics[width=5.8cm,height=2.5cm]{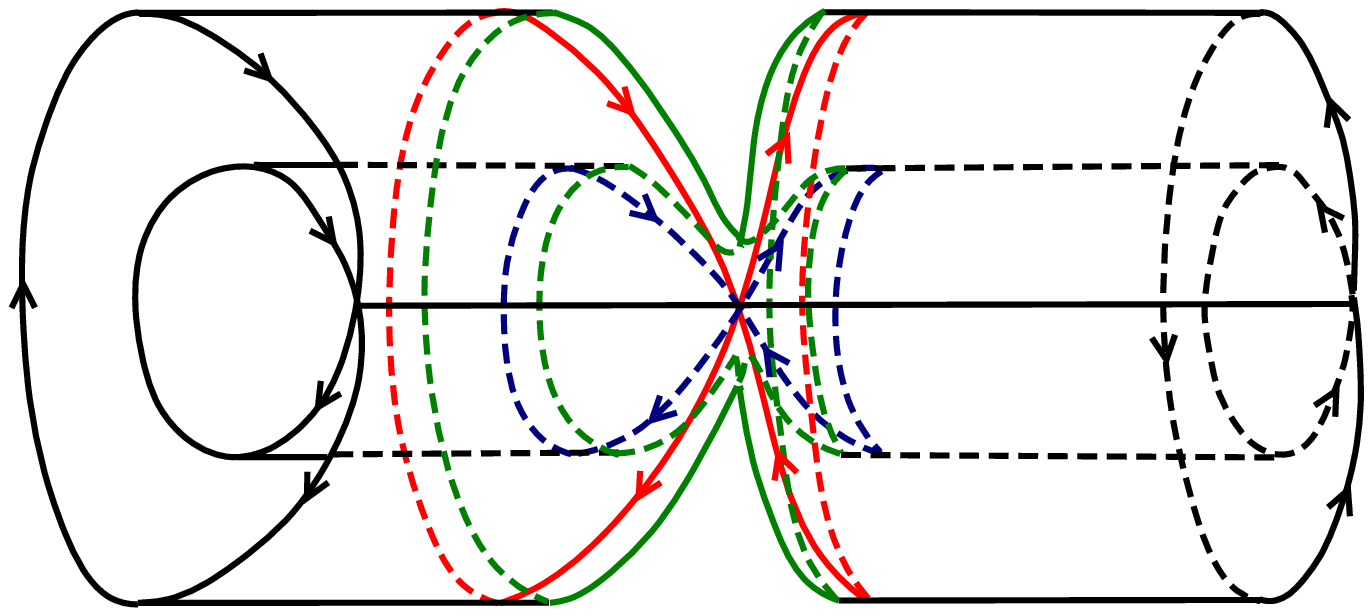}
      \end{subfigure}
  \begin{subfigure}[b]{0.35\textwidth}
               \includegraphics[width=3.3cm,height=3.5cm]{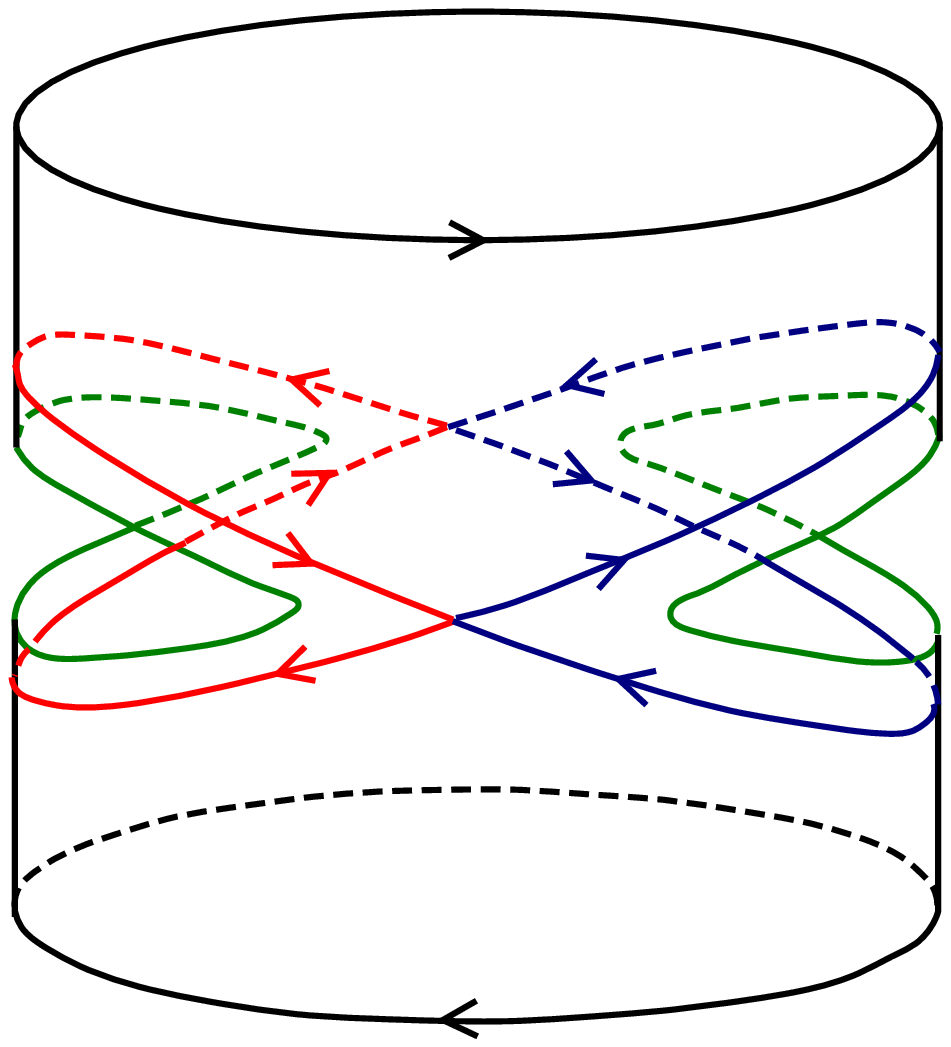}
      \end{subfigure}
  \caption{The left figure shows a singular cylinder in $\mathbb{T}^n\times\mathbb{R}^n$ for the case $k=2$, there are two pairs of homoclinic orbits, one is in red and another one is in dark blue. The right one is its lift to $\mathbb{T}^n_h\times\mathbb{R}^n$.}
  \label{fig1}
\end{figure}
\noindent{\bf Remark}. For $k\ge 2$, it is possible that $\ell\ge 1$. Here is an example that $n=2$, there are two pairs of homoclinic orbits $z^\pm_1(t)$ and $z^\pm_2(t)$ with $[z^+_1]=(1,0)$ and $[z^+_2]=(0,1)$. In this case,  $\mathbb{T}^2_h=\{(x_1,x_2)\in\mathbb{R}^2:x_1,x_2\mod 2\}$ and $\ell=1$. 

Without the condition ({\bf H3}), $\Pi$ can be still treated as a surface with self-intersection. Also, the $k$ pairs of homoclinic orbits are not required to be all different.

The hypotheses ({\bf H1, H2}) are open-dense condition in $C^{r}$-topology for any $r\ge 2$. The first one is obvious. To see the second, we notice that the local stable and unstable manifold $W^{s}_{loc}$ and $W^{u}_{loc}$ have their generating functions $S^s$ and $S^u$ respectively such that $W^{s,u}_{loc}=\mathrm{graph}(dS^{s,u})$.
A homoclinic orbit has to be in the local stable or unstable manifold when it approaches the fixed point as $t\to\infty$ or $t\to -\infty$. Therefore, in a ball $B_r\subset\mathbb{R}^n_x$ about the origin of suitably small radius $r$, each point uniquely determines an orbit lying in the stable (unstable) manifold. Since $\lambda_1<\lambda_2$, each $x\in B_r$ determines an orbit lying $W^{s,u}_{loc}$ that approaches the origin in the direction of $\Xi_1^{\pm}$ if and only if $x$ does not lie in a co-dimension one hypersurface $S$ which is diffeomorphic to a disc of $(n-1)$-dimension. To check the condition (\ref{eq1.2}) of transversal intersection, we notice that $dS^s=dS^u$ holds along the homoclinic orbits around the origin. There are plenty of small perturbations such that $\partial^2(S^u-S^s)$ is non-degenerate when it is restricted on a codimension-one section transversal to the curve $x^\pm(t)$.

The existence of the periodic orbits is reduced to the problem to find fixed point of the Poincar\'e return map. It will be down by applying Banach's fixed point theorem. We first study the periodic orbit of a single homology type in Section \ref{sec.2} for $E>0$ and in Section \ref{sec.3} for $E<0$. The periodic orbit of compound type homology class is studied in Section \ref{sec.5} and the uniqueness is proved in Section \ref{sec.4}. Because the return map is not defined on the level set $H^{-1}(0)$, the hard part is to prove the $C^1$-smoothness around the homoclinic orbits, which is fulfilled in Section \ref{sec.6} and \ref{sec.7}. The application of Theorem \ref{mainresult} to the problem of double resonance is discussed finally in Section \ref{sec.8}.

\section{Periodic orbit with single homology class}\label{sec.2}
\setcounter{equation}{0}
In this section we study the continuation of periodic orbits from a single homoclinic orbit $z^+(t)$ to positive energy region. Let $B_{r'}\subset\mathbb{R}^{2n}$ denote a ball about the origin of radius $r'$, where the coordinates $(x,y)$ are chosen such that 
\begin{equation*}
H(x,y)=G(x,y)+R(x), \qquad \mathrm{for}\ (x,y)\in B_{r'}
\end{equation*}
where $G=\sum_{i=1}^n\frac 12(y_i^2-\lambda_i^2x_i^2)$ and $R$ is the higher order term $R=O(\|x\|^3)$, namely, $|R(x)|/\|x\|^3$ is bounded as $\|x\|\to 0$.

For a vector $x=(x_1,x_2,\cdots,x_n)$, we use $\|x\|=(\sum_{i=1}^nx_i^2)^{\frac12}$ to denote its Euclidean norm and use $|x|=\max\{|x_1|,|x_2|,\cdots,|x_n|\}$ to denote its maximum norm.

Restricted in a neighborhood of the origin, we introduce a canonical transformation for convenience of notation
\begin{equation}\label{eq2.1}
\left[\begin{matrix}x_i \\ y_i\end{matrix}\right]=
\frac 1{\sqrt{2}}\left[\begin{matrix}
\frac 1{\sqrt{\lambda_i}} & -\frac 1{\sqrt{\lambda_i}}\\
\sqrt{\lambda_i} & \sqrt{\lambda_i}
\end{matrix}\right]\left[\begin{matrix}u_i \\ v_i\end{matrix}\right].
\end{equation}
In $(u,v)$-coordinates, the Hamiltonian $H$ takes the form $H(u,v)=\sum\lambda_iu_iv_i+R(u,v)$ with $R(u,v)=O(\|(u,v)\|^3)$, the Hamiltonian equation turns out to be
\begin{equation}\label{eq2.2}
\dot u_i=\lambda_iu_i+\frac{\partial R}{\partial v_i}, \qquad
\dot v_i=-\lambda_iv_i-\frac{\partial R}{\partial u_i}.
\end{equation}
If the Hamiltonian is $C^{2\kappa+1}$-smooth with $\kappa\in\mathbb{N}$, one has its Birkhoff normal form
\begin{equation}\label{eq2.3}
H(u,v)=\sum\lambda_iu_iv_i+N(I_1,\cdots,I_n)+O(\|(u,v)\|^{2\kappa+1})
\end{equation}
where $I_i=u_iv_i$ and $N$ is a polynomial of degree $\kappa$ without constant and linear part.

Since it is hyperbolic, the fixed point $z=0$ has its stable (unstable) manifold $W^s$ and $W^u$. Some $r'>0$ exists such that, restricted in $B_{r'}$, they are the graph of some maps. In the coordinates $(u,v)$, $\exists$ $U:\{v\in\mathbb{R}^n\}\cap B_{r'}\to\mathbb{R}^n$ and $V:\{u\in\mathbb{R}^n\}\cap B_{r'}\to\mathbb{R}^n$ such that
$$
W^s|_{B_{r'}}=\{U(v),v\},\qquad W^u|_{B_{r'}}=\{u,V(u)\}.
$$
\begin{lem}\label{flatpro}
Restricted in $B_{r'}$ with suitably small $r'>0$, there exists a canonical transformation $(p,q)=T(u,v)$ such that, for the Hamiltonian flow of $H^*=H\circ T^{-1}$, the local stable manifold $W^s$ and the unstable one $W^u$ satisfy the condition
\begin{equation}\label{eq2.4}
W^s|_{B_{r'}}=\{0,q\},\qquad W^u|_{B_{r'}}=\{p,0\}.
\end{equation}
If $H$ is a Birkhoff normal form, then $H^*=H\circ T^{-1}$ is also a Birkhoff normal form.
\end{lem}
\begin{proof}
Because the stable and unstable manifold are Lagrangian sub-manifold, both $U$ and $V$ are the differential of some functions
$$
U(v)=\partial F_u(v),\qquad V(u)=\partial F_v(u).
$$
By the generating function $S(u,q')=\langle u,q'\rangle-F_u(q')$, we get a canonical transformation $\Psi$: $v=q'$, $p'=u-\partial F_u(q')$. In the coordinates $(p',q')$, the local stable manifold lies in the subspace $\mathbb{R}^n_q=\{p'=0\}$, the unstable manifold is a graph $\{p',\partial F'_v(p')\}$ of some function $F'_v$. With the generating function $S'(p,q')=\langle p,q'\rangle-F'_v(p)$, we get another canonical transformation $\Psi'$: $p=p'$, $q=q'-\partial F'_v(p)$. Let $T=\Psi'\Psi$, the stable and unstable manifold of $\Phi^t_{H^*}$ satisfy the condition (\ref{eq2.4}).

For Birkhoff normal form, the generating function $F_s$ ($F_u$) of the stable (unstable) manifold is of order $2\kappa+2$, $F_{s,u}=O(|z|^{2\kappa+1})$. The generating function $S(u,q')=\langle u,q'\rangle-F_u(q')$ satisfies the condition $F_u(q')=O(|q'|^{2\kappa+1})$. Thus, the transformation does not change the normal form.
\end{proof}

Since $F_u(q)=O(\|q\|^2)$ and $F'_v(p)=O(\|p\|^2)$ when they are restricted in $B_{r'}$, we are able to extend $F_u$ and $F'_v$ $C^2$-smoothly to the whole space by setting $F_u=0$, $F'_v=0$ when they are valued at the place outside of $B_{\sqrt{r'}}$. With the lemma, we are able to assume that the Hamiltonian $H(u,v)$ satisfies one more condition:

({\bf H4}), {\it the local stable $($respectively unstable$)$ manifold of the hyperbolic fixed point $z=0$ is a neighborhood of the fixed point in the stable $($respectively unstable$)$ subspace of the linear flow $\Phi^t_G$.}

\begin{lem}
Under the assumption $(${\bf H4}$)$, the remainders in Equation $($\ref{eq2.2}$)$ satisfy the conditions $\partial_vR(0,v)=0$ and $\partial_uR(u,0)=0$ for $|(u,v)|<r'$. Thus, the remainder $R$ of $H$ admits the form
\begin{equation}\label{eq2.5}
R(u,v)=\langle R_{1,1}(u,v)u,v\rangle, \qquad \mathrm{with}\ \ R_{1,1}(0,0)=0.
\end{equation}
\end{lem}
\begin{proof}
Since the local stable (respectively unstable) manifold of the hyperbolic fixed point $\{z=0\}$ lies in the stable (respectively unstable) subspace of the linear flow $\Phi^t_G$. It holds on the $\{v=0\}\cap B_{r'}$ (respectively $\{u=0\}\cap B_{r'}$) that $\dot v=0$ (respectively $\dot u=0$), namely, $\partial_uR(u,0)=0$ and $\partial_vR(0,v)=0$. For the proof of (\ref{eq2.5}), we have
$$
\begin{aligned}
R(u,v)=&\int_0^u\frac{\partial R}{\partial u}(u,v)du+F(v).
\end{aligned}
$$
It follows that
$$
\frac{\partial R}{\partial v}(u,v)=\int_0^u\frac{\partial^2R}{\partial u\partial v}du+\partial F(v),
$$
from which, we obtain $\partial F(v)=0$ by applying the relation $\partial_vR(0,v)=0$. Therefore, one has
$$
R(u,v)=\int_0^v\int_0^u\frac{\partial^2R}{\partial u\partial v}dudv+G(u).
$$
By applying the condition that $\partial_uR(u,0)=0$, we obtain $\partial G(u)=0$.
Since $R(u,v)=O(|(u,v)|^3)$, the condition (\ref{eq2.5}) is proved.
\end{proof}

Let $\Sigma_{\pm r}^-=\{u_1=\pm r\}$ and $\Sigma_{\pm r}^+=\{v_1=\pm r\}$ where the superscript ``+" indicates that the orbit $z^\pm(t)$ is approaching the origin when it passes through the section, while ``$-$" indicates that the orbit is getting away from the origin when it crosses the section. Because of the assumption ({\bf H2}), the homoclinic orbit $z^\pm(t)$ passes through the section $\Sigma^-_{\pm r}$ and $\Sigma^+_{\pm r}$ at the points $z^-_{\pm r}$ and $z^+_{\pm r}$ respectively. Emanating from $z^-_{\pm r}$ at $t=0$, it moves along the homoclinic orbit lying outside of $B_r(0)$ before it arrives at $z^+_{\pm r}$ after a finite time $T_0$. Due to the continuous dependence of solution of ODE on its initial value, some neighborhood $U_r$ of $z^-_{r}$ lying in the section $\Sigma_r^-$ exists satisfying the condition: emanating from any point $z\in U_r$, the orbit $\Phi_H^t(z)$ keeps close to $z^+(t)|_{[0,T_0]}$ before it arrives at $z'\in\Sigma^+_r$ after a finite time. In this way we get a map $\Phi_r$: $U_r\to\Sigma_r^+$ such that $z'=\Phi_r(z)$ and call it {\it outer map}, because it is defined by orbits lying outside of $B_r(0)$. Another outer map $\Phi_{-r}$: $U_{-r}\to\Sigma_{-r}^+$ is defined similarly.

With the Hamiltonian flow $\Phi_H^t$ we define {\it inner map} $\Phi_{r,r}$. Emanating from a point $z\in\Sigma^+_r\backslash W^s$ around $z^+_r$, the orbit keeps close to the stable manifold until arrives at the section $\{u_1=v_1\}$, then it keeps close to the unstable manifold until it arrives at $z'\in\Sigma^-_r$. We define $z'=\Phi_{r,r}(z)$. As we shall see later, it is well-defined only when $z\in H^{-1}(E)$ with $E>0$. For $E<0$, it shall cross the section $\{u_1=-v_1\}$ and hit a point $z'\in\Sigma^-_{-r}$, we get another inner map $\Phi_{r,-r}(z)=z'$ in this case. The inner map $\Phi_{-r,\pm r}$ is defined similarly.

Let $\Sigma^-_{E,\pm r}=H^{-1}(E)\cap\{u_1=\pm r\}$ and $\Sigma^+_{E,\pm r}=H^{-1}(E)\cap\{v_1=\pm r\}$ be the section in the energy level set $H^{-1}(E)$. Correspondingly, $U_r$ admits a foliation of energy level set $U_r=\cup_EU_{E,r}$. The restriction of the outer map on $U_{E,\pm r}$ is denoted by $\Phi_{E,\pm r}$, see Figure \ref{Figure2}. Let $V_r=\Phi_rU_r$, it also admits a foliation of energy level sets $V_r=\cup_EV_{E,r}$. The restriction of $\Phi_{r,r}$ on $V_{E,r}$ is denoted by $\Phi_{E,r,r}$, see Figure \ref{Figure2} also. As we shall see later, the inner map $\Phi_{E,r,r}$ is studied by decomposing it as the composition of $\Phi_{E,r,0}$: $S_{E,r}\to\{u_1=v_1\}$ and $\Phi_{E,0,r}$: $\{u_1=v_1\}\to U_{E,r}$.

\begin{lem}\label{map1.2}
The maps $\Phi_{E,r}$, $\Phi_{E,r,0}$ and $\Phi_{E,0,r}$ preserve the closed 2-form
$$
\hat\omega=d\hat x\wedge d\hat y=d\hat u\wedge d\hat v.
$$
\end{lem}
\begin{proof}
Consider the vortex lines of the form $ydx-Hdt$ in $(2n+1)$-dimensional extended phase space. If $\sigma$ is a piece of vortex tube and $\gamma_1$ and $\gamma_2$ are closed curve encircling the same tube such that $\gamma_1-\gamma_2=\partial\sigma$, one has the integral invariant of Poincar\'e-Cartan
$$
\oint_{\gamma_1}ydx-Hdt=\oint_{\gamma_2}ydx-Hdt.
$$
Let $\gamma_1$ be a closed curve lying in $U_{E,r}$ and $\gamma_2=\Phi_{E,r}\gamma_1$. Because $dH(v)=0$ holds for any vector tangent to $H^{-1}(E)$, $u_1$ keeps constant in $U_{E,r}$ and $v_1$ keeps constant in $V_{E,r}$ we obtain from Stock's formula that the following holds
\begin{equation}\label{symplectic}
\int_{\sigma_1}d\hat u\wedge d\hat v=\int_{\sigma_2}d\hat u\wedge d\hat v
\end{equation}
for 2-dimensional disc $\sigma_1\subset U_{E,r}$, $\sigma_2\subset V_{E,r}$ bounded by $\gamma_1$ and $\gamma_2$ respectively. It finishes the proof for $\Phi_{E,r}$. For $\Phi_{E,r,0}$, let $\gamma_1$ be a closed curve lying in $S_{E,r}$ and let $\gamma_2=\Phi_{E,r,0}\gamma_1$. then the projection of $\sigma_2$ to the plan $\{(u_1,v_1)\}$ does not contain interior if $\sigma_2\subset\{u_1=v_1\}$ is a surface bounded by $\gamma_2$. Therefore, \eqref{symplectic} also holds.
\end{proof}

To study the outer map $\Phi_{E,r}$ and the inner map $\Phi_{E,r,r}$, we introduce some rules of notation. Let $\hat u=(u_2,\cdots,u_n)$ and $\hat v=(v_2,\cdots,v_n)$. The principle of the notation also applies to $\hat x$, $\hat y$. 
Let $\hat\pi:\mathbb{R}^{2n}\to\mathbb{R}^{2(n-1)}$ be the projection so that $\hat\pi(u,v)=(\hat u,\hat v)$ and let $\pi_u,\pi_v$ the be projection such that $\pi_u(u,v)=u$, $\pi_v(u,v)=v$, $\pi_u(\hat u,\hat v)=\hat u$ and $\pi_v(\hat u,\hat v)=\hat v$.

Recall the homoclinic orbit $z^\pm(t)$ passes through $\Sigma^\pm_{\pm r}$ at the points  $z^\pm_{\pm r}$ which are written in coordinates
$$
z^-_{\pm r}=(\pm r,\hat u^-_{\pm r},v^-_{\pm r}),\qquad z^+_{\pm r}=(u^+_{\pm r},\pm r, \hat v^+_{\pm r})
$$
where $\hat u^-_{\pm r}=(u^-_{2,\pm r},\cdots,u^-_{n,\pm r})$ and the same principle of notation also applies to $\hat v^-_{\pm r}$, $\hat u^+_{\pm r}$ and $\hat v^+_{\pm r}$. The hypotheses ({\bf H2},{\bf H4}) imply that
$$
\hat v^-_{\pm r}=0,\quad \hat u^+_{\pm r}=0, \quad |\hat u^-_{\pm r}|=o(r),\quad |\hat v^+_{\pm r}|=o(r).
$$
To find the periodic orbit, let us specify what is the set $U_r$.

\begin{defi}\label{def2.1}
Given small $\delta>0$, let $U_{E,r}\subset H^{-1}(E)$ be a subset such that
$$
\hat\pi U_{E,r}=\hat U_{\delta}=\{|\hat u-\hat u^-_r|\le\delta,|\hat v-\hat v^-_r|\le\delta\}.
$$
Let $U_r=\cup_{-E_0\le E\le E_0}U_{E,r}$.
\end{defi}
Clearly, there exist $E_0,\delta>0$ depending on $r>0$ such that $U_{E,r}$ is well-defined for $|E|\le E_0$. Let $V_{E,r}=\Phi_{E,r}U_{E,r}$ over which the inner map $\Phi_{E,r,r}$ or $\Phi_{E,r,-r}$ may not be well-defined. We define a set $S_{E,r}\subseteq V_{E,r}$ where the inner map $\Phi_{E,r,r}$ or $\Phi_{E,r,-r}$ is well-defined.

Each $z\in U_{E,r}$ is determined by its $\hat z$-component, there exists a unique $v_1=v_1(\hat z,E)$ such that $z=(r,v_1,\hat z)$. The projection of map $\Phi_{E,r}$, denoted by $\hat\Phi_{E,r}$, is well defined such that $\hat\Phi_{E,r}(\hat\pi z)=\hat\pi\Phi_{E,r}(z)$. It is symplectic and smoothly depends on $E$ when $E$ is suitably small. For a $C^1$-map $F$: $\pi_u \hat U_\delta\to\pi_v \hat U_\delta$, each $E\in[-E_0,E_0]$ induces a graph $\mathcal{G}_{F,E}=\{(r,v_1(\hat u,F(\hat u),E),\hat u,F(\hat u)):|\hat u-\hat u^-_r|\le\delta\}$. The transversal intersection property \eqref{eq1.2} makes sure that the outer map $\Phi_{E,r}$ brings the graph $\mathcal{G}_{F,E}$ to a graph $\mathcal{G}_{\Phi_{E,r}^*F}$ of $\Phi_{E,r}^*F$. To check, we define the cones with $\alpha>0$
$$
\begin{aligned}
\hat K^-_{\alpha}&=\{(\xi_{\hat u},\xi_{\hat v})\in\mathbb{R}^{2n-2}:\alpha|\xi_{\hat u}|\ge |\xi_{\hat v}|\},\\
\hat K^+_{\alpha}&=\{(\xi_{\hat u},\xi_{\hat v})\in\mathbb{R}^{2n-2}:\alpha|\xi_{\hat v}|\ge |\xi_{\hat u}|\}.
\end{aligned}
$$
Let $\mathscr{F}=\{F\in C^1(\pi_u \hat U_\delta,\pi_v \hat U_\delta):\|DF\|\le\eta\}$ denote a set of maps, we are going to show that if $\mathcal{G}_F$ is the graph of $F\in\mathscr{F}$, then $\hat\Phi_{E,r}\mathcal{G}_F\subset\hat V_{E,r}$ is also a graph if $\eta>0$ is assumed suitably small. For this end, we consider the tangent map $d\hat\Phi_{E,r}$. For a vector $(\xi_{\hat u},\xi_{\hat v})\in T_{\hat z}\mathcal{G}_F$, i.e. $|\xi_{\hat u}|\ge \eta^{-1}|\xi_{\hat v}|$, let $(\xi'_{\hat u},\xi'_{\hat v})=d\hat\Phi_{E,r}(\hat z)(\xi_{\hat u},\xi_{\hat v})$. It follows from the transversal intersection property (\ref{eq1.2}) that $|\xi'_{\hat u}|\ne 0$ if $\eta>0$ is suitably small. Indeed, we have the lemma

\begin{lem}\label{lem2.2}
For small $|E|\le E_0$, the transversal intersection hypothesis $($\ref{eq1.2}$)$ implies that there exist $\lambda, M>0$, each $\alpha\in(0,\frac{\lambda}{M})$ determines $\alpha^*=\frac{(1+\alpha)M}{\lambda-M\alpha}$ such that

1) for $(\xi_{\hat u},\xi_{\hat v})\in\hat K^-_\alpha$ and $z\in U_{E,r}$, one has $(\xi^*_{\hat u},\xi^*_{\hat v})=d\hat\Phi_{E,r}(\hat z)(\xi_{\hat u},\xi_{\hat v})\in\hat K^-_{\alpha^*}$;

2) for $(\xi_{\hat u}^\star,\xi_{\hat v}^\star)\in\hat K^+_\alpha$ and $z\in S_{E,r}$, one has $(\xi_{\hat u},\xi_{\hat v})=d\hat\Phi^{-1}_{E,r}(\hat z)(\xi_{\hat u}^\star,\xi_{\hat v}^\star)\in\hat K^+_{\alpha^*}$.
\end{lem}
\begin{proof}
In $(u,v)$-coordinates, we consider the differential of $\hat\Phi_{E,r}$ at $\hat z$:
\begin{equation}\label{eq2.13}
d\hat\Phi_{E,r}(\hat z)=\left[\begin{matrix}
A_{11}(E,\hat z) & A_{12}(E,\hat z) \\ A_{21}(E,\hat z) & A_{22}(E,\hat z)\end{matrix}\right],
\end{equation}
where $A_{ij}$ is $(n-1)\times(n-1)$ sub-matrix. We claim $\mathrm{det}(A_{11}(0,\hat z^-_{r}))\ne 0$.

If not, there would be a vector $\xi_{\hat u}\ne 0$ such that $A_{11}\xi_{\hat u}=0$. Since $\Phi_{0,r}$ is symplectic preserving, $d\hat\Phi_{0,r}(\hat z^-_{r})(\xi_{\hat u},0)=(0,\xi'_{\hat v})\ne 0$ is a non-zero vector that must lie in the stable subspace. But it is absurd because the stable manifold intersects the unstable manifold transversally in the sense of (\ref{eq1.2}), it is then impossible that $d\Phi_{0,r}$ maps a vector of $T_{z^-_{r}}W^u$ into $T_{z^+_{r}}W^s$.

The matrix $d\hat\Phi_{E,r}(\hat z)$ continuously depends on $E$ and $\hat z$. Therefore, for $\hat z$ around $\hat z^-_r$ and small $E$, one has $\mathrm{det}(A_{11}(E,\hat z))\ne 0$. Let $\lambda>0$ be smaller than the absolute value of the smallest eigenvalue of $A_{11}$ and let $M>0$ be larger than $|A_{ij}|$, the norm of $A_{ij}$. In this case, for a vector $(\xi_{\hat u},\xi_{\hat v})\in\hat K^-_\alpha$ with $\alpha<\frac{\lambda}{M}$
$$
\begin{aligned}
|\xi^*_{\hat u}|&=|A_{11}(E,\hat z) \xi_{\hat u}+A_{12}(E,\hat z)\xi_{\hat v}|\ge(\lambda-M\alpha)|\xi_{\hat u}|>0\\
|\xi^*_{\hat v}|&=|A_{21}(E,\hat z) \xi_{\hat u}+A_{22}(E,\hat z)\xi_{\hat v}|\le(1+\alpha)M|\xi_{\hat u}|.
\end{aligned}
$$
Let $\alpha^*=\frac{(1+\alpha)M}{\lambda-M\alpha}$, then $(\xi^*_{\hat u},\xi^*_{\hat v})=d\hat\Phi_{E,d,d}(\hat z)(\xi_{\hat u},\xi_{\hat v})\in K^-_{\alpha^*}$. It proves the first item.

The second item is proved similarly. In $(u,v)$-coordinates, we write the differential of $\hat\Phi^{-1}_{E,r}$ at $\hat z$ in the form:
\begin{equation*}\label{eq2.14}
d\hat\Phi^{-1}_{E,r}(\hat z_{r})=\left[\begin{matrix}
B_{11}(E,\hat z_{r}) & B_{12}(E,\hat z_{r}) \\ B_{21}(E,\hat z_{r}) & B_{22}(E,\hat z_{r})\end{matrix}\right],
\end{equation*}
where $B_{ij}$ is $(n-1)\times(n-1)$ sub-matrix. For the same reason to show $\mathrm{det}(A_{11}(0,\hat z^-_{r}))\ne 0$, one has $\mathrm{det}(B_{22}(0,\hat z^+_{r}))\ne 0$. If we let $\lambda>0$ be smaller than the absolute value of the smallest eigenvalue of $B_{11}$ and let $M>0$ be larger than $|B_{ij}|$ for small $E$ and $z$ around $z^+_{r}$, then the rest of the proof is the same as above.
\end{proof}

We next show that, for small $E>0$, $\Phi_{E,r,r}\mathcal{G}_{\Phi_{E,r}^*F}\cap S_{E,r}$ intersects $U_{E,r}$. It is the first step to show the existence of periodic orbits emerging from homoclinics.

\begin{pro}\label{pro2.5}
There exists some $E_0>0$, for each $E\in(0,E_0]$, some set $S_{E,r}\subseteq V_{E,r}$ exists so that $\Phi_{E,r,r}(\mathcal{G}_{\Phi_{E,r}^*F}\cap S_{E,r})\cap U_{E,r}\ne\varnothing$. The map $\Phi_{E,r,r}$ expands $\mathcal{G}_{\Phi_{E,r}^*F}\cap S_{E,r}$ in $\hat u$-component such that $\Phi_{E,r,r}S_{E,r}$ covers $\{|\hat u-\hat u^-_r|\le r\}$ in the sense that
\begin{equation}\label{eq2.6}
\pi_{u}\hat\pi\Phi_{E,r,r}(\mathcal{G}_{\Phi_{E,r}^*F}\cap S_{E,r})\supseteq\{|\hat u-\hat u^-_r|\le r\},
\end{equation}
and it contracts $S_{E,r}$ in the $\hat v$-component such that
\begin{equation}\label{eq2.7}
\pi_{v}\hat\pi(\Phi_{E,r,r}S_{E,r}\cap U_{E,r})\subseteq\{|\hat v|\le cr^{3-2c'r}E^{1-c'r}\}.
\end{equation}
where $c,c'>0$ are constants independent of $E$ and $r$.
\end{pro}
\begin{proof}
It is proved by using the hyperbolic property of the flow $\Phi_H^t$ when it is restricted around the origin. To see the property more clearly, we introduce multi-dimensional polar-spherical coordinates
$$
v_i=\rho\Psi_i(\psi),\quad u_i=\varrho\Phi_i(\phi), \qquad \mathrm{for}\ i=1,\cdots,n
$$
where $\sum\Psi_i^2=1$ and $\sum\Phi_i^2=1$ for $\phi,\psi\in\mathbb{S}^{n-1}$. For instance, let $\Psi_1=\sin\psi_1\sin\psi_2$, $\Psi_2=\sin\psi_1\cos\psi_2$ and $\Psi_3=\cos\psi_1$ for $n=3$. Since $\sum\Psi_i^2=1$ implies $\sum\langle\dot\psi,\partial\Psi_i \rangle\Psi_i=0$, it is reduced from the equation (\ref{eq2.2}) and the relation $\dot v_i=\dot \rho\Psi_i+\rho\langle\dot\psi,\partial\Psi_i\rangle$ that
\begin{equation*}
\begin{aligned}
\dot\rho&=\dot\rho\sum\Psi_i^2+\rho\sum\langle\dot\psi,\partial\Psi_i\rangle\Psi_i =\sum\dot v_i\Psi_i\\
&=-\rho\sum\lambda_i\Psi_i^2-\sum\partial_{u_i}R(\varrho\Phi,\rho\Psi)\Psi_i.\\
\dot\varrho&=\varrho\sum\lambda_i\Phi_i^2+\sum\partial_{v_i}R(\varrho\Phi,\rho\Psi) \Phi_i.
\end{aligned}
\end{equation*}
Since the stable and unstable manifold of $\Phi_H^t$ are assumed to be the stable and unstable subspace of $e^{\mathrm{diag}(\Lambda,-\Lambda)t}$ as shown in (\ref{eq2.4}), there exist smooth functions $U'_i$ and $V'_i$ such that $\partial_{v_i}R=\langle u,U'_i\rangle$ and $\partial_{u_i}R=\langle v,V'_i\rangle$ with $U'_i(0,0)=0$ and $V'_i(0,0)=0$. Therefore, we see that some $c>0$ exists such that
\begin{equation*}
\begin{aligned}
(\lambda_n+cr)\varrho&\ge\dot\varrho\ge(\lambda_1-cr)\varrho,\\
-(\lambda_n+cr)\rho&\le \dot\rho\le-(\lambda_1-cr)\rho
\end{aligned}
\end{equation*}
holds along each orbit $\Phi_H^t(u,v)$ lying $B_{2r}$. By Gro\"nwell's inequality one has
\begin{equation}\label{eq2.8}
\begin{aligned}
\varrho(0)e^{(\lambda_n+cr)t}&\ge\varrho(t)\ge \varrho(0)e^{(\lambda_1-cr)t},\\
\rho(0)e^{-(\lambda_n+cr)t}&\le \rho(t)\le\rho(0)e^{-(\lambda_1-cr)t}.
\end{aligned}
\end{equation}

Let $\mathcal{G}'=\cup_{-E_0\le E\le E_0}\mathcal{G}_{\Phi_{E,r}^*F}$. Because of the property $\partial_{u_1}H(z^+_r)=\lambda_1r(1+O(r))>0$ and Lemma \ref{lem2.2}, there exists some $\delta'>0$ such that $\pi_u\mathcal{G}'\supset\{|u|\le\delta'\}$. We consider the set
$$
\Phi_t\mathcal{G}'=\{\Phi^t_H(z):z\in\mathcal{G}',|\Phi_H^s(z)|\le 2r,\ \forall\ s\in[0,t]\}.
$$
Notice that $\Phi_t\mathcal{G}'$ may not be the same as $\Phi_H^t\mathcal{G}'$, one has $\Phi_t\mathcal{G}'\subsetneq\Phi_H^t\mathcal{G}'$ for large $t$. It follows from  (\ref{eq2.8}) that some $t_1\le\frac 1{\lambda_1-cr}(\ln 2r-\ln\delta')$ exists such that for $t\ge t_1$ one has $\pi_u\Phi_t\mathcal{G}'=\{|u|\le 2r\}$ and $\pi_v\Phi_t\mathcal{G}'\subseteq\{v:|v|\le \frac{\delta'}{2r}\}$. The set $\cup_{t\ge t_1}\Phi_t\mathcal{G}'$ intersects $U_r\subset\Sigma^-_r$ on an $n$-dimensional strip
$$
\Pi_{r}=U_r\cap(\cup_{t\ge t_1}\Phi_t\mathcal{G}').
$$
\begin{lem}\label{pro2.2}
Let $z_E(t)=(u_E(t),v_E(t))$ be an orbit of $\Phi_H^t$ lying in the set $H^{-1}(E)$ with small $|E|>0$, let $t_E>0$ be the number such that $z_E(\pm t_E)\in\partial B_{r}$ and $z_E(t)\in B_r$ $\forall$ $t\in[-t_E,t_E]$. If the boundary values $u_E(t_E)=(u_1(t_E),\cdots,u_n(t_E))$ and $v_E(-t_E)=(v_1(-t_E),\cdots,v_n(-t_E))$ satisfy the condition $|u_1(t_E)|\ge|\hat u(t_E)|$, $|v_1(-t_E)|\ge|\hat v(-t_E)|$ and $|u_1(t_E)|=|v_1(t_E)|=r$ then

1) $E>0$ if $v_1(-t_E)u_1(t_E)>0$ and $E<0$ if $v_1(-t_E)u_1(t_E)<0$;

2) some constant $c_E>0$ exists, uniformly bounded as $|E|\to 0$ such that
\begin{equation}\label{eq2.9}
\frac 1{\lambda_1}\ln\frac 1{|E|}+2\ln r-c_E\le 2t_E\le\frac 1{\lambda_1}\ln\frac 1{|E|}+2\ln r+c_E.
\end{equation}
holds for suitably small $|E|>0$.
\end{lem}
It follows from Lemma \ref{pro2.2} that $H(z)>0$ for any $z\in\Pi_{r}$, we shall present its proof in the end of this section.
\begin{lem}\label{foreq2.8}
Let $\Pi_{E,r}=\Pi_{r}\cap H^{-1}(E)$, then $\pi_u\Pi_{E,r}\supseteq\{|\hat u-\hat u^-_r|\le r\}$ holds for all $E\in(0,E_0]$.
\end{lem}
\begin{proof}
It follows from $|\hat u_r^-|=o(r)$ that $\hat u\in\{|\hat u-\hat u^-_r|\le r\}\subset\{|u|\le 2r\}$. So, we have $\pi_u\Phi_t\mathcal{G}'\cap\Sigma^-_r\supseteq\{|\hat u-\hat u^-_r|\le r\}$, namely, for any $\hat u\in\{|\hat u-\hat u^-_r|\le r\}$ some $v=v(\hat u,t)$ such that $(r,\hat u,v)\in\pi_u\Phi_t\mathcal{G}'$ provided $t\ge t_1$.
Because of \eqref{eq2.9}, some $c>1$ exists such that
$$
c^{-1}r^{2\lambda_1}e^{-\lambda_1t}\le H(r,\hat u,v)\le cr^{2\lambda_1}e^{-\lambda_1t}.
$$
Therefore, we have $H(r,\hat u,v(\hat u,t))\to 0$ as $t\to\infty$. Let $E_0=c^{-1}r^{2\lambda_1}e^{-\lambda_1t_1}$, then for any $E\in(0,E_0]$, there exists $v$ such that $(r,\hat u,v)\in\pi_u\Phi_t\mathcal{G}'$ and $H(r,\hat u,v)=E$. 
\end{proof}

With the initial position in $z_E(0)\in S_{E,r}$, we assume that the orbit arrives at the section $U_{E,r}$ after a time $t_z$. By applying Lemma \ref{pro2.2}, we see that $t_z=2t_E$ is controlled by \eqref{eq2.9}. It follows from \eqref{eq2.8} that $|v_E(t_z)|\le|v_E(0)|e^{-(\lambda_1-cr)t_z}\le cr^{3-2c'r}E^{1-c'r}$. It verifies \eqref{eq2.7}. The proof of Lemma \ref{pro2.5} is completed.
\end{proof}

The arguments also apply to the case of negative energy. Let $U_{-r}=\cup_EU_{E,-r}$ where $U_{E,-r}\subset\{u_1=-r\}\cap H^{-1}(E)$ is defined such that
$\hat\pi U_{E,-r}=\{|\hat u-\hat u^-_{-r}|\le\delta,|\hat v-\hat v^-_{-r}|\le\delta\}.$
In this case, the set $\cup_{t\ge t_1}\Phi_tV_r$ also intersects $U_{-r}$ on a $(2n-1)$-dimensional strip
$\Pi_{-r}=U_{-r}\cap(\cup_{t\ge t_1}V_r),$
it follows from Lemma \ref{pro2.2} that $H(z)<0$ for any $z\in\Pi_{-r}$. Therefore, each $S_{E,r}$ with small $E<0$ also maps to $\Pi_{E,r}$ which satisfies the condition $\Pi_{E,r}=\{|\hat u-\hat u^-_{-r}|\le r\}$. It implies the existence of some small $E'_0>0$ such that for $E\in [-E'_0,0)$ one has
\begin{equation}\label{eq2.10}
\pi_u\Phi_{E,r,-r}(\mathcal{G}_{\Phi_{E,r}^*F}\cap S_{E,r})\supset\pi_u U_{E,-r}, \qquad \pi_v\Phi_{E,r,-r}S_{E,r}\subset\pi_v U_{E,-r}.
\end{equation}
\begin{figure}[htp]
  \centering
  \includegraphics[width=7.5cm,height=4.5cm]{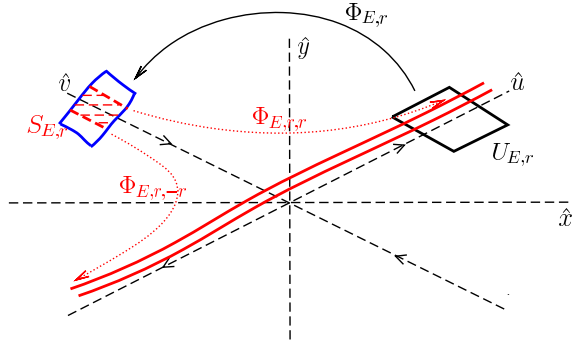}
  \caption{The image of $U_{E,r}$ under $\Phi_{E,r.r}\Phi_{E,r}$ intersects itself.}
  \label{Figure2}
\end{figure}

In the rest of this paper, we always use $c,c',c_i$ to denote positive constants independent of $E$ and $r$. They may be differently valued in different places if there is no danger of confusion.

As the second step to find periodic orbit, we establish the contraction property of the graph transformation. It is induced by the map $\Phi_E=\Phi_{E,r,r}\Phi_{E,r}$ for $E>0$. Given a map $F\in C^1(\pi_u\hat\pi U_{E,r},\pi_v\hat\pi U_{E,r})$, we have a subset $\mathcal{G}_F=\mathrm{graph}F\cap U_{E,r}$.

\begin{pro}\label{theo2.2}
Some $E_0>0$ exists such that for each $E\in(0,E_0]$, there exists a $C^1$-map $F_E\in C^1(\pi_u\hat U_{\delta},\pi_v\hat U_{\delta})$ satisfying the condition $\hat\Phi_E \mathcal{G}_{F_E}\supseteq\mathcal{G}_{F_E}$. Moreover, the inverse of $\hat\Phi_E$, when it is restricted on $\mathcal{G}_{F_E}$, is a contraction map.
\end{pro}
\begin{proof}
It is proved by Banach's fixed point theorem. We recall the outer map $\Phi_{E,r}$: $U_{E,r}\to V_{E,r}$ and the inner map $\Phi_{E,r,r}$: $S_{E,r}\subseteq V_{E,r}\to U_{E,r}$. Their projections $\hat\Phi_{E,r}$, $\hat\Phi_{E,r,r}$, are well defined such that $\hat\Phi_{E,r}(\hat\pi z)=\hat\pi\Phi_{E,r}(z)$ and $\hat\Phi_{E,r,r}(\hat\pi z)=\hat\pi\Phi_{E,r,r}(z)$. By applying the following lemma, of which the proof is postponed to the end of this section, we claim that the map $\Phi_E$ induces a transformation $F\in\mathscr{F}\to\Phi^*_EF\in\mathscr{F}$.

\begin{lem}\label{lem2.4}
For $(\xi_{\hat u},\xi_{\hat v})\in\hat K^-_{1}$ and $z\in S_{E,r}$, let $(\xi^*_{\hat u},\xi^*_{\hat v})=d\hat\Phi_{E,r,r}(\hat z)(\xi_{\hat u},\xi_{\hat v})$. Then, there exist constants $c,c'>0$ such that
\begin{equation}\label{eq2.11}
|\xi^*_{\hat u}|\ge e^{(\lambda_2-cr)t_z}|\xi_{\hat u}|, \qquad |\xi^*_{\hat v}|\le c're^{-(\lambda_1-cr)t_z}t_z|\xi^*_{\hat u}|
\end{equation}
where $t_z$ is the time when $\Phi_H^{t_z}z\in\{u_1=r\}$ and $\Phi_H^{t}z\notin\{u_1=r\}$ for $t\in[0,t_z)$. For any $(\xi_{\hat u},\xi_{\hat v})\in\hat K^+_{1}$ and $z\in S_{E,r}$, let $(\xi^{\star}_{\hat u},\xi^{\star}_{\hat v})=d\hat\Phi_{E,r,r}^{-1}(\Phi_H^{t_z}(\hat z))(\xi_{\hat u},\xi_{\hat v})$, one has
\begin{equation}\label{eq2.12}
|\xi_{\hat v}|\le e^{-(\lambda_2-cr)t_z}|\xi_{\hat v}^{\star}|, \qquad |\xi_{\hat u}^{\star}|\le c're^{-(\lambda_1-cr)t_z}t_z|\xi_{\hat v}^{\star}|.
\end{equation}
\end{lem}

Recall the set of maps $\mathscr{F}=\{F\in C^1(\pi_u \hat U_\delta,\pi_v \hat U_\delta):\|DF\|\le\eta\}$, we set $\eta\le\frac{\lambda}{2M}$. Applying Lemma \ref{lem2.4} we find that each tangent vector $(\xi_{\hat u},\xi_{\hat v})\in T_{\hat z}\mathcal{G}_F$ is mapped by $d\hat\Phi_{E,r}$ to a vector $(\xi^*_{\hat u},\xi^*_{\hat v})\in \hat K^-_{\eta^*}$ with $\eta^*=\frac{(1+\eta)M}{\lambda-M\eta}$. Therefore, $\hat\Phi_{E,r}$ maps the graph $\mathcal{G}_F$ to a graph over $\pi_{u}\hat\Phi_{E,r}\mathcal{G}_F$, denoted by $\mathcal{G}_{\Phi_{E,r}^*F}$.

\begin{lem}
Restricted on $\hat\pi S_{E,r}\subseteq\hat\pi V_{E,r}$, the graph $\hat\Phi_{E,r} \mathcal{G}_F$ is mapped by $\hat\Phi_{E,r,r}$ to a graph $\mathcal{G}_{F'}$ satisfying the condition $\pi_u\mathcal{G}_{F'}\supset\{|\hat u|\le r\}$ and $F'|_{\pi_u\hat U_{\delta}}\in\mathscr{F}$.
\end{lem}
\begin{proof}
For each $\hat z\in \hat\pi S_{E,r}$ we consider each tangent vector $(\xi'_{\hat u},\xi'_{\hat v})\in T_{\hat z}\hat\Phi_{E,r}\mathcal{G}_F$. Since $F\in\mathscr{F}$, $(\xi'_{\hat u},\xi'_{\hat v})\in K^{-1}_{\eta^*}$. Let $(\xi^*_{\hat u},\xi^*_{\hat v})=d\hat\Phi_{E,r,r}(\hat z)(\xi'_{\hat u},\xi'_{\hat v})$. Since $t_z$ is bounded by the estimate (\ref{eq2.9}), it follows from the second estimate in (\ref{eq2.11}) that
\begin{equation}\label{eq2.15}
|\xi^*_{\hat v}|\le c're^{-(\lambda_1-cr)t_z}t_z|\xi^*_{\hat u}|\le c'r^{1-2(\lambda_1-cr)}|\ln E||E|^{\frac{\lambda_1-cr}{\lambda_1}}|\xi^*_{\hat u}|.
\end{equation}
Since $|\ln E||E|^{\frac{\lambda_1-cr}{\lambda_1}}\to 0$ as $E\to 0$, it implies that the set $\hat\Phi_{E,r,r}\hat\Phi_{E,r}\mathcal{G}_F$ is also a graph, almost horizontal in the sense that each tangent vector lies in $\hat K^-_{\alpha}$ with $\alpha=O(|E|)$.

The graph $\hat\Phi_{E,r}\mathcal{G}_F$ induces a map $\hat u\to \hat v=F'(\hat u)$ such that $(\hat u,F'(\hat u))\in \hat\Phi_{E,r}\mathcal{G}_F$, with which we are able to define a map $\Psi_E:\pi_u(\hat\Phi_{E,r}\mathcal{G}_F\cap\hat S_{E,r})\to \mathbb{R}^{n-1}_{\hat u}$ such that
$$
\Psi_E(\hat u)=\pi_u\hat\Phi_{E,r,r}(\hat u,F'(\hat u))
$$
and obtain the expansion property from the first estimate in (\ref{eq2.11}) with (\ref{eq2.9})
\begin{equation*}
|d\Psi_E(\hat u)\xi_{\hat u}|\ge e^{(\lambda_2-cr)t_z}|\xi_{\hat u}|\ge c'r^{2(\lambda_2-cr)} |E|^{-\frac{\lambda_2-cr}{\lambda_1}} |\xi_{\hat u}|.
\end{equation*}
It guarantees $\Psi_E\pi_u(\hat\Phi_{E,r}\mathcal{G}_F\cap\hat S_{E,r})\supset\pi_u U_{\delta}$ if $E$ is suitably small. Indeed, any ball in $\mathbb{R}^{n-1}_{\hat u}$ with small radius $\rho$ is mapped by $\Psi_E^{-1}$ back into a ball with radius not larger than $O(|E|^{(\lambda_2-cr)/\lambda_1}\rho)$ while the size of $\pi_u\hat\Phi_{E,r}\mathcal{G}_F$ is bounded from below uniformly in $E$. Thus, the set $\Phi_E\mathcal{G}_F$ is a graph of some map $F'\in C^1(\pi_u U_{\delta}, \mathbb{R}^{n-1})$, i.e. the map $\Phi_E$ induces a transformation $F\to F'=\Phi_E^*F$ such that $\Phi_E\mathcal{G}_F=\mathcal{G}_{F'}$. For small $E$, \eqref{eq2.15} implies that $\|F'\|\le\eta$, namely, $F'\in\mathscr{F}$.
\end{proof}

We claim that the transformation $\Phi_E^*$ is a contraction in $C^0$-topology. Given two maps $F_1$ and $F_2$, we assume that $\max_{\hat u\in\pi_{u}\hat\pi\Phi_{E}\Phi^{-1}_{E,r}S_{E,r}}|\Phi_E^*F_1(\hat u)-\Phi_E^*F_2(\hat u)|$ is achieved at a point $\hat u'$.  Let $(\xi_{\hat u}^{\star},\xi^{\star}_{\hat v})=d\hat\Phi_{E,r,r}^{-1}(\hat u',\Phi_E^*F_1(\hat u'))(0,\xi'_{\hat v})$ with $\xi'_{\hat v}=(F_2-F_1)(\hat u')$. Since $t_z$ is bounded by (\ref{eq2.9}), by applying Lemma \ref{lem2.4} we obtain that
$$
|\xi'_{\hat v}|\le cr^{-2(\lambda_2-cr)}|E|^{\frac{\lambda_2-cr}{\lambda_1}}|\xi^{\star}_{\hat v}|,\qquad |\xi_{\hat u}^{\star}|\le c'r^{1-2(\lambda_1-cr)}|\ln E||E|^{\frac{\lambda_1-cr}{\lambda_1}}|\xi_{\hat v}^{\star}|,
$$
i.e. $(\xi_{\hat u}^{\star},\xi^{\star}_{\hat v})\in\hat K^+_{\alpha}$ with $\alpha=O(|E|)$. Let $(\xi_{\hat u},\xi_{\hat v})$ be a vector such that $d\hat\Phi_{E,r}(\xi_{\hat u},\xi_{\hat v})=(\xi_{\hat u}^{\star},\xi^{\star}_{\hat v})$. By applying Lemma \ref{lem2.2} one has $(\xi_{\hat u},\xi_{\hat v})\in\hat K^+_{\alpha^*}$ with $\alpha^*=\frac{(1+O(|E|))M}{\lambda-O(|E|)M}$.

Let $(\hat u_i,\hat v_i)=\hat\Phi_{E}^{-1}(\hat u',\Phi_E^*F_i(\hat u'))$ for $i=1,2$, then $\hat v_i=F_i(\hat u_i)$. The demonstration right above shows $(\hat u_1-\hat u_2,F_1(\hat u_1)-F_2(\hat u_2))\in K^+_{\alpha^*}$, i.e. $|\hat u_1-\hat u_2|\le\alpha^*|F_1(\hat u_1)-F_2(\hat u_2)|$. $\|DF_i\|\le\eta\le\frac{\lambda}{2M}$,
$$
\begin{aligned}
|F_1(\hat u_1)-F_2(\hat u_1)|&=|F_1(\hat u_1)-F_2(\hat u_2)+F_2(\hat u_2)-F_2(\hat u_1)|\\
&\ge |F_1(\hat u_1)-F_2(\hat u_2)|-\eta|\hat u_1-\hat u_2|\\
&\ge (1-\eta\alpha^*)|F_1(\hat u_1)-F_2(\hat u_2)|.
\end{aligned}
$$
By the definition, one has $0<\eta\alpha^*<1$. Applying the first estimate in \eqref{eq2.12} one has
$$
\begin{aligned}
\|\Phi_E^*F_1-\Phi_E^*F_2\|&=\max_{\hat u\in\pi_{u}\hat\pi\Phi_{E}\Phi^{-1}_{E,r} S_{E,r}}|\Phi_E^*F_1(\hat u)-\Phi_E^*F_2(\hat u)| \\
&\le e^{-(\lambda_2-cr)t_z}|F_1(\hat u_1)-F_2(\hat u_2)|\\
&\le(1-\eta\alpha^*)^{-1}e^{-(\lambda_2-cr)t_z}\|F_1-F_2\|,
\end{aligned}
$$
i.e. the map is a contraction when $t_z$ is large, it corresponds to small $|E|$.
Hence, the map $\Phi_E$ induces a contraction map $\Phi_E^*$: $\mathscr{F}\to\mathscr{F}$. Banach's fixed point theorem leads to the existence of a unique invariant $F_E\in\mathscr{F}$, it is of course $C^1$-smooth.

The contraction property of $\Phi_E^{-1}|_{\mathcal{G}_{F_E}}$ is shown by checking the expansion property of $\Phi_E$. For $(\xi_{\hat u},\xi_{\hat v})\in K^-_{\eta}$, Lemma \ref{lem2.2} implies $(\xi_{\hat u}^*,\xi_{\hat v}^*)\in K^-_{\eta^*}$ with $|\xi_{\hat u}^*|\ge (\lambda-\eta M)|\xi_{\hat u}|$. Hence, $|d\Phi_E(\xi_{\hat u},\xi_{\hat v})|\ge e^{(\lambda_2-cr)t_z} (\lambda-\eta M)|(\xi_{\hat u},\xi_{\hat v})|$ is got from (\ref{eq2.11}). For large $t_z$ one has $e^{-(\lambda_2-cr)t_z}(\lambda-\eta M)^{-1}<1$. The proof of Theorem \ref{theo2.2} is finished.
\end{proof}

\begin{theo}\label{theo2.3}
Some $E_0>0$ exists such that for each $E\in(0,E_0]$ there is a periodic orbit $z^+_E(t)\subset H^{-1}(E)$ entirely lying in the vicinity of $z^+(t)$.
\end{theo}
\begin{proof}
Due to Proposition \ref{theo2.2} and Banach's fixed point theorem, there is a fixed point $z_{E,r}$ of $\Phi_E$ in $\mathcal{G}_{F_E}$, since the map $\Phi_E^{-1}$ is a contraction when it is restricted on $\mathcal{G}_{F_E}$.
\end{proof}

What remains to complete this section is the proof for the technical lemmas applied before. We now do it.

\begin{proof}[Proof of Lemma \ref{pro2.2}]
We write $z_E(t)=(u_E(t),v_E(t))$ with $u_E(t)=(u_1(t),\cdots,u_n(t))$ and $v_E(t)=(v_1(t),\cdots,v_n(t))$. By the method of variation of constants, we see that the solution of the Hamilton equation generated by $H$ satisfies the equation
\begin{equation}\label{flateq5}
u_{i}(t)=e^{\lambda_{i}t}(u_{i,0}+F_{u,i}), \quad
v_{i}(t)=e^{-\lambda_{i}t}(v_{i,0}-F_{v,i}),
\end{equation}
for $i=1,\cdots,n$, $u_{i}^{0}$ and $v_{i}^{0}$ are the initial value and
$$
F_{u,i}^-=\int_0^te^{-\lambda_{i}s}\partial_{v_{i}}R(u(s),v(s))ds, \quad
F_{v,i}^+=\int_0^te^{\lambda_{i}s}\partial_{u_{i}}R(u(s),v(s))ds.
$$
If $(u(t),v(t))\in B^{2n}_{r}$ for $t\in[-t_E,t_E]$, it follows from an improved Hartman-Grobman Theorem that there is a conjugacy $h$ between $\Phi_H^t$ and $e^{\mathrm{diag}(\Lambda,-\Lambda)t}$ such that
$$
\Phi_H^t(u,v)=h^{-1}e^{\mathrm{diag}(\Lambda,-\Lambda)t}h(u,v),
$$
where $\Lambda=\mathrm{diag}(\lambda_1,\cdots,\lambda_n)$. Moreover, if writing $h=id+f$ and $h^{-1}=id+g$, we obtain from Theorem 1.1 of \cite{vS} that $f=O(\|(u,v)\|^{1+\nu})$ and $g=O(\|(u,v)\|^{1+\nu})$ with $\nu>0$, since $H$ is assumed to be at least $C^3$-smooth.

Let $f=(f_u,f_v)$, $g=(g_u,g_v)$ and $f_u=(f_{u,1}\cdots,f_{u,n})$. The principle of notation for $f_u$ also applies to $f_v,g_u,g_v$. Let $u_0=(u_{1,0}\cdots,u_{n,0})$ and $v_0=(v_{1,0}\cdots,v_{n,0})$, then
\begin{equation}\label{Eq2.17}
\begin{aligned}
u(t)&=e^{\Lambda t}(u_0+f_u)+g_u(e^{\Lambda t}(u_0+f_u),e^{-\Lambda t}(v_0+f_v)),\\
v(t)&=e^{-\Lambda t}(v_0+f_v)+g_v(e^{\Lambda t}(u_0+f_u),e^{-\Lambda t}(v_0+f_v)).
\end{aligned}
\end{equation}
Setting $t=t_E$ in the first equation of \eqref{Eq2.17} and setting $t=-t_E$ in the second we obtain
\begin{equation}\label{eq2.18}
\begin{aligned}
|u_{i,0}+f_{u,i}(u_0,v_0)|&=e^{-\lambda_it_E}\Big(u_i(t_E)-g_{u,i}(e^{\Lambda t_E}(u_0+f_u),e^{-\Lambda t_E} (v_0+f_v))\Big),\\
|v_{i,0}+f_{v,i}(u_0,v_0)|&=e^{-\lambda_it_E}\Big(v_i(-t_E)-g_{v,i}(e^{-\Lambda t_E}(u_0+f_u),e^{\Lambda t_E} (v_0+f_v))\Big).
\end{aligned}
\end{equation}
Since $g=O(\|(u,v)\|^{1+\nu})$, $|e^{\Lambda t_E}(u_0+f_u)|\le 2r,|e^{-\Lambda t_E} (v_0+f_v)|\le 2r$, $|u_1(t_E)|=r$ and $|v_1(-t_E)|=r$, one has
$$
\begin{aligned}
|u_{1,0}+f_{u,1}(u_0,v_0)|&\ge r(1-O(r))e^{-\lambda_1 t_E}, \\
|v_{1,0}+f_{v,1}(u_0,v_0)|&\ge r(1-O(r))e^{-\lambda_1 t_E},\\
|u_{i,0}+f_{u,i}(u_0,v_0)|&\le O(r)e^{-\lambda_i t_E}, \quad \forall\ i\ge 2 \\
|v_{i,0}+f_{v,i}(u_0,v_0)|&\le O(r)e^{-\lambda_i t_E}.
\end{aligned}
$$
Since $f=O(\|(u,v)\|^{1+\nu})$ and $\lambda_i>\lambda_1$, we obtain from these estimates that
\begin{equation*}\label{eq2.19}
\begin{aligned}
|u_{1,0}|&\ge \frac12re^{-\lambda_1 t_E}, \quad
&&|v_{1,0}|\ge \frac12re^{-\lambda_1 t_E},\\
|u_{i,0}|&\le O(r)e^{-\bar\lambda_\ell t_E}, \quad
&&|v_{i,0}|\le O(r)e^{-\bar\lambda_\ell t_E}, \quad \forall\ i\ge 2.
\end{aligned}
\end{equation*}
where $\bar\lambda_i=\min\{2\lambda_1,\lambda_i\}$, provided $t_E$ is suitably large. Substituting $(u(t),v(t))$ into $H$, we obtain a constraint for the initial values
\begin{equation*}\label{eq2.20}
\begin{aligned}
H(u(0),v(0))=&\sum_{i=1}^n\lambda_i u_{i,0}v_{i,0}+R(u_{i,0},v_{i,0})\\
=&\frac 14\mathrm{sign}(u_{1,0}v_{1,0})\lambda_1r^2e^{-2\lambda_1t_E}+O(e^{-2\mu t_E})
\end{aligned}
\end{equation*}
where $\mu=\min\{\frac 32\lambda_1,\bar\lambda_2\}>\lambda_1$. Taking logarithm on both sides we find some constant $c_E>0$ exists, uniformly bounded as $E\to 0$, such that (\ref{eq2.9}) holds, from which we see that $t_E$ is large if $|E|$ is small and $\mathrm{sign}(E)=\mathrm{sign}(v_1(-t_E)u_1(t_E))$.
\end{proof}

\begin{proof}[Proof of Lemma \ref{lem2.4}]
We study the differential of the map $\Phi_H^t$ through the variational equation along an orbit $z(t)$ of the Hamiltonian flow $\Phi_H^t$. In the coordinates $(u,v)$ let $\xi=(\xi_u,\xi_v)=(\delta u,\delta v)$, the equation takes the form
\begin{equation}\label{eq2.21}
\dot\xi=A(t)\xi,
\end{equation}
where the $2n\times 2n$ matrix $A(t)=\mathrm{diag}\{\Lambda,-\Lambda\}+P(u(t),v(t))$, $\Lambda=\mathrm{diag}\{\lambda_1,\cdots,\lambda_n\}$. Let $\Psi(t)$ be the fundamental matrix of the variational equation such that $\Psi(0)=I$. Be aware that $R(z)=O(\|z\|^3)$ in (\ref{eq2.5}), each element of $P(u(t),v(t))$ is bounded by $cr$ for $|z(t)|\le r$. For $\alpha>0$, we consider the cone
\begin{equation*}\label{eq2.22}
\begin{aligned}
K^-_{\alpha,k}=&\{(\xi_u,\xi_v)\in\mathbb{R}^{2n}:\alpha |(\xi_{u_{k+1},\cdots,\xi_{u_n}})|\ge|(\xi_{u_1},\cdots,\xi_{u_k},\xi_v)|\},\\
K^-_{\alpha}=&\{(\xi_u,\xi_v)\in\mathbb{R}^{2n}:\alpha|\xi_u|\ge|\xi_v|\}.
\end{aligned}
\end{equation*}
\begin{lem}\label{lem2.5}
Assume $z(s)\in B_r$ for $s\in[0,t]$, then $\exists$ $\alpha_r>0$ with $\alpha_r\to 0$ as $r\to 0$ such that for $\alpha\in[\alpha_r,1]$, the cones $K^-_{\alpha,k}$ and $K^-_\alpha$ are all invariant for $\Psi(t)$ with $t>0$.
\end{lem}
\begin{proof}
To consider the cone $K^-_{\alpha,k}$, we introduce the polar-spherical coordinates
\begin{equation*}
\begin{aligned}
\xi_{u_i}&=\varrho'\Phi_i(\phi), \quad  i=k+1,\cdots,n, \\
\xi_{u_i}&=\rho'\Psi_i(\psi), \quad i=1,\cdots,k\\
\xi_{v_i}&=\rho'\Psi_{i+k}(\psi), \quad i=1,\cdots,n.
\end{aligned}
\end{equation*}
where $\sum\Phi_i^2=1$ for $\phi\in\mathbb{S}^{n-k}$ and $\sum\Psi_i^2=1$ for $\psi\in\mathbb{S}^{n+k}$. We obtain from Equation (\ref{eq2.21}) that
\begin{equation}\label{eq2.23}
\begin{aligned}
\dot \varrho'&=\varrho'\sum_{i=k+1}^n\lambda_i\Phi_i^2+U'_\varrho\varrho' +U'_\rho\rho'>(\lambda_{k+1}-cr)\varrho'-cr\rho',\\
\dot\rho'&=\rho'(\sum_{i=1}^k\lambda_i\Psi_i^2-\sum_{i=1}^n\lambda_i\Psi_{i+k}^2) +V'_\varrho\varrho'+V'_\rho\rho'<(\lambda_k +cr)\rho'+cr\varrho'
\end{aligned}
\end{equation}
where $U'_\varrho, U'_\rho,V'_\varrho$ and $V'_\rho$ depend on $z(t)=(u(t),v(t)),\phi$ and $\psi$. $|U'_\varrho|, |U'_\rho|, |V'_\varrho|$ and $|V'_\rho|$ are all bounded by $cr$ if $|z(t)|\le r$. So in $K^-_{\alpha,k}$, $\dot\varrho'\ge(\lambda_{k+1} -(1+ \alpha)cr)\varrho'$ and on the boundary of $K^-_{\alpha,k}$ one has $\dot\rho'\le (\lambda_k +(1+\frac 1\alpha)cr)\rho'$. It guarantees that $\alpha\dot\varrho'>\dot \rho'$ holds on the boundary of $K^-_{\alpha,k}$ if $\lambda_{k+1} -(1+ \alpha)cr>\lambda_k +(1+\frac 1\alpha)cr$, i.e. $\frac 1\alpha+\alpha< \frac{\lambda_{k+1}-\lambda_k}{cr}-2$. Since $\lambda_{k+1}>\lambda_k$ and $r>0$ is small, there exist positive numbers $\hat\alpha^-_r<\hat\alpha^+_r$ such that $\frac 1\alpha+\alpha< \frac{\lambda_{k+1}-\lambda_k}{cr}-2$ holds for all $\alpha\in(\hat\alpha^-_r,\hat\alpha^+_r)$. In this case, the cone $K^-_{\alpha,k}$ is invariant.
Since $\lambda_{k+1}>\lambda_k$, $r$ is chosen suitably small, one has $\hat\alpha^+_r=O(r^{-1})$ and $\hat\alpha^-_r=O(r)$.

For the cone $K^-_\alpha=K^-_{\alpha,0}$, we also introduce the polar-spherical coordinates
$$
\xi_{u_i}=\varrho\Phi_i(\phi), \quad\xi_{v_i}=\rho\Psi_i(\psi), \quad i=1,\cdots,n,
$$
where $\sum\Phi_i^2=1$ for $\phi\in\mathbb{S}^{n}$ and $\sum\Psi_i^2=1$ for $\psi\in\mathbb{S}^{n}$. In this case, we have
\begin{equation}\label{eq2.24}
\begin{aligned}
\dot\varrho&=\varrho\Sigma_{i=1}^n\lambda_i\Phi_i^2+U_\varrho\varrho+U_\rho\rho >(\lambda_1-cr)\varrho-cr\rho, \\ \dot\rho&=-\rho\Sigma_{i=1}^n\lambda_i\Psi_i^2+V_\varrho\varrho+V_\rho\rho< -(\lambda_1-cr)\rho+cr\varrho.
\end{aligned}
\end{equation}
For $\alpha\in(\frac{cr}{\lambda_1-cr},\frac{\lambda_1}{cr}-1)$, it holds on the boundary of $K^-_\alpha$ that $\dot\varrho>0$ and $\dot\rho<0$. It implies that $K^-_\alpha$ is invariant. Therefore, for each $\alpha\in(\max\{\hat\alpha^-_r,\frac{cr}{\lambda_1-cr}\},\min\{\hat\alpha^+_r, \frac{\lambda_1}{cr}-1\})$, both $K^-_{\alpha,k}$ and $K^-_\alpha$ are invariant. 
\end{proof}

\begin{lem}\label{lem2.9}
If $\xi=(\xi_u,\xi_v)\in T_zV^\pm_E$ with small $|E|$, then $\xi_{u_1},\xi_{v_1}=o(|(\xi_{\hat u},\xi_{\hat v}|)$ and $(\xi_{\hat u},\xi_{\hat v})\in\hat K^-_{1}$ implies that $(\xi_u,\xi_v)\in K^-_{1,1}\cap K^-_1$.
\end{lem}
\begin{proof}
By the definition, the coordinates of $z=(u,v)\in V_{E,\pm r}$ takes the form $v_1=r$, $|\hat u|=o(r)$, $|\hat v|=o(r)$ and
$\sum\lambda_iu_iv_i+\langle R_{1,1}(u,v)u,v\rangle=E$
where $R_{1,1}(z)=O(|z|)$. As $\xi=(\xi_u,\xi_v)$ is tangent to $V^\pm_E$ at $z$, one has $\xi_{v_1}=0$ and
\begin{equation}\label{eq2.25}
\sum_{i=1}^n\lambda_i(u_i\xi_{v_i}+v_i\xi_{u_i})+\langle R_u(u,v),\xi_{u}\rangle+\langle R_v(u,v),\xi_{v}\rangle=0.
\end{equation}
where $R_u(u,v),R_v(u,v)=O(|z|^2)$. For small $|E|$, $v_1=\pm\sqrt{2\lambda_1}r+o(r)$. Since $|u_i|$,$|v_i|=o(r)$ for $i\ge 2$, it follows from (\ref{eq2.25}) that $|\xi_{u_1}|=o(|\xi_{\hat u}|,|\xi_{\hat v}|)$. Thus, a vector $(\xi_u,\xi_v)\in T_zV^\pm_E$ with $(\xi_{\hat u},\xi_{\hat v})\in\hat K^-_{\alpha}$ implies that $(\xi_u,\xi_v)\in K^-_{\alpha,1}\cap K^-_\alpha$.
\end{proof}

Due to the properties $\partial_{u_i}R(u,0)=0$ and $\partial_{v_i}R(0,v)=0$, the terms $V_\varrho$ and $U_\rho$ in (\ref{eq2.24}) satisfy the condition that $V_r/|v(t)|$ and $U_\rho/|u(t)|$ are bounded as $|v(t)|\to 0$, $|u(t)|\to 0$ respectively. Since $|v_E(t)|\le |v_E(0)|e^{-(\lambda_1-cr)t}$ with $|v_E(0)|=O(r)$ cf. (\ref{eq2.8}), we find from the second equation in (\ref{eq2.24}) that some $c_1>0$ exists such that
$$
\dot\rho\le -(\lambda_1-cr)\rho+c_1re^{-(\lambda_1-cr)t}\varrho(t).
$$
By a variant of the Gro\"nwell inequality we obtain
$$
\begin{aligned}
\rho(t)&\le e^{-(\lambda_1-cr)t}\Big\{\rho(0)+c_1r\int_0^t\varrho(s)ds \Big\}\\
&\le \rho(0)e^{-(\lambda_1-cr)t}+c_1r\max_{s\in[0,t]}\varrho(s)te^{-(\lambda_1-cr)t}.
\end{aligned}
$$
For $(\xi_u,\xi_v)\in K^-_{\alpha}$ with $\alpha\le 1$, we reduce from the first inequality of (\ref{eq2.24}) that $\dot\varrho>0$. In this case, $\max_{s\in[0,t]}\varrho(s)=\varrho(t)$. Notice that $\|\xi_u\|=\varrho$ and $\|\xi_v\|=\rho$ we get from the inequality right above that
\begin{equation}\label{eq2.26}
\frac{|\xi_v(t)|}{|\xi_u(t)|}\le c_2re^{-(\lambda_1-cr)t}t, \quad \mathrm{for}\ t>0.
\end{equation}
To control the growth of $|\xi_{u_1}(t)|$ we make use of Formula (\ref{eq2.23}).
For $(\xi_{\hat u},\xi_{\hat v})\in\hat K^-_1$, we obtain from Lemma \ref{lem2.9} that $(\xi_u,\xi_v)\in K^-_{1,1}\cap K^-_1$ if $(\xi_u,\xi_v)\in T_zS_{E,\pm r}$. Consequently, we have $\dot\varrho'> (\lambda_2-2cr)\varrho'$, through which we induce from (\ref{eq2.23}) that
\begin{equation}\label{eq2.27}
\frac{d\rho'}{d\varrho'}=\frac{\dot\rho'}{\dot\varrho'} \le\frac{\lambda_1+cr}{\lambda_2-2cr}\frac{\rho'}{\varrho'}+\frac{cr}{\lambda_2-2cr}.
\end{equation}
Thus, the initial condition $\rho'(0)\le\alpha\varrho'(0)$ with $\alpha\le 1$ leads to the relation $\frac{d\rho'}{d\varrho'}\le 1$, i.e. $\dot\rho'(t)\le\dot\varrho'(t)$ holds for all $t\ge 0$. So, that $|\xi_{u_1}(t)|\le\rho'(t)\le\varrho'(t)=|\xi_{\hat u}(t)|$ leads to
\begin{equation}\label{eq2.28}
\frac{|\xi_{\hat v}(t)|}{|\xi_{\hat u}(t)|}\le \frac{|\xi_v(t)|}{|\xi_u(t)|}, \qquad \forall\  (\xi_u(0),\xi_v(0))\in T_zS_{E,\pm r}, \ \ (\xi_{\hat u}(0),\xi_{\hat v}(0))\in\hat K^-_1.
\end{equation}
Since $(\xi_u,\xi_v)\in K^-_{1,1}$, it follows from the first inequality of (\ref{eq2.23}) that
\begin{equation}\label{eq2.30}
\varrho'(t)\ge\varrho'(0)e^{(\lambda_2 -2cr)t}.
\end{equation}
By a variable substitution $\rho'=s\varrho'^{\frac{\lambda_1+cr}{\lambda_2-2cr}}$, we obtain from \eqref{eq2.27} that $s(t)$ and $\varrho$ satisfy the inequality
$\frac{ds}{d\varrho'}\le\frac{cr}{\lambda_2-2cr}\varrho'^{-(\lambda_1+cr)/(\lambda_2-2cr)}$.
Consequently, we have
\begin{equation}\label{eq2.29}
\begin{aligned}
\rho'(t)\le&\frac{cr}{\lambda_2-\lambda_1-3cr}\varrho'(t)+ c'\varrho'(t)^{\frac{\lambda_1+cr}
{\lambda_2-2cr}}\\
=&\Big(\frac{cr}{\lambda_2-\lambda_1-3cr}+c'\varrho'(t)^{-\frac{\lambda_2-\lambda_1-3cr}
{\lambda_2-2cr}}\Big)\varrho'(t)
\end{aligned}
\end{equation}
where the constant $c'>0$ is chosen such that it holds for $t=0$. Because $\lambda_2>\lambda_1$, we have $\lambda_2-\lambda_1-3cr>0$ for small $r>0$.
For large $t$, it follows from \eqref{eq2.29} that for $t_z$ lower bounded by \eqref{eq2.9} with small $E$
\begin{equation}\label{eq2.31}
|(\xi_{u_1},\xi_v)(t_z)|\le c_3r|\xi_{\hat u}(t_z)|.
\end{equation}

Next, let us establish the relation between the differential of $\Phi_{E,r,r}$ and of $\Phi_H^t$.
Let $X_H=(X_{u_1},X_{u_2},\cdots X_{u_n},X_{v_1},X_{v_2},\cdots X_{v_n})$ denote the Hamiltonian field, then
\begin{lem}\label{differentialmapofsection}
Let $\xi=(0,\xi_{\hat u},\xi_{v_1},\xi_{\hat v})\in T_zV_{E,r}$ and assume $\Phi_H^t(z)\in U_{E,r}$, then
$$
d\Phi_{E,r,r}(z)\xi=d\Phi_H^t(z)\xi+\nu X_H(\Phi_H^t(z))
$$
where $\nu=-\xi'_{u_1}X^{-1}_{u_1}(\Phi_H^t(z))$ if we write $d\Phi_H^t(z)\xi=(\xi'_{u_1},\xi'_{\hat u},\xi'_ {v_1},\xi'_{\hat v})$.
\end{lem}
\begin{proof} Emanating from the points $z,z'\in V^\pm_E$, the trajectories arrive at the set $U^\pm_E$ after the time $t$ and $t'$ respectively. One has $t'-t\to 0$ if $z'\to z$. We have the identity
$$
\begin{aligned}
\Phi_H^{t'}(z')-\Phi_H^t(z)&=\Phi_H^{t'}(z')-\Phi_H^{t}(z')+\Phi_H^{t}(z')-\Phi_H^t(z)\\
&=\Phi_H^{t'-t}\Phi_H^{t}(z')-\Phi_H^{t}(z')+d\Phi_H^{t}(z)(z'-z)+O(|z'-z|^2))\\
&=X_H(\Phi_H^t(z'))(t'-t)+d\Phi_H^{t}(z)(z'-z)+O(|z'-z|^2,|t'-t|^2).
\end{aligned}
$$
Since the $u_1$-component of $\Phi_H^{t'}(z')-\Phi_H^t(z)$ vanishes and the $u_1$-component of $X_H$ is non-zero, the number $\nu$ is uniquely defined such that the lemma holds.
\end{proof}

To apply the lemma, we denote by $\xi'=d\Phi_{H}^{t_z}(z)\xi=(\xi'_{u_1},\xi'_{\hat u},\xi'_{v_1},\xi'_{\hat v})$. Since the initial vector $(\xi_u,\xi_v)\in T_zV_E$, it follows from Lemma \ref{lem2.9} that $(\xi_u,\xi_v)\in K^-_{\alpha,1}\cap K^-_\alpha$ provided $(\xi_{\hat u},\xi_{\hat v})\in\hat K^-_{\alpha}$. In this case, we get from (\ref{eq2.31}) that $|\xi'_{u_1}|\le O(r)|\xi'_{\hat u}|$. Be aware the special form of $R$ in (\ref{eq2.5}) The Hamiltonian vector field of $H$ takes the form
\begin{equation*}
X_{u_i}=\lambda_iu_i(1+O(r)), \qquad X_{v_i}=-\lambda_iv_i(1+O(r)).
\end{equation*}
For each $z\in\Phi_{E,r,r}V^+_E$, one has $|u_1|=\lambda_1r$, $|\hat u|=o(r)$, $|v|\le O(r)e^{-(\lambda_1-cr)t_z}$ and
$$
\frac{|X_{u_i}(z)|}{|X_{u_1}(z)|}\le O(r),\qquad \frac{|X_{v_i}(z)|}{|X_{u_1}(z)|}\le c'e^{-(\lambda_1-cr)t_z}, \quad \forall\ i\ge 2.
$$
Applying Lemma \ref{differentialmapofsection} to our situation, we find $\nu X_{\hat u}=-\xi'_{u_1}\frac{X_{\hat u}}{X_{u_1}}$ and $\nu X_{\hat v}=-\xi'_{u_1}\frac{X_{\hat v}}{X_{u_1}}$. Recall the notation $(\xi_{\hat u}^*,\xi_{\hat v}^*) =d\hat\Phi_{E,r,r}(\xi_{\hat u},\xi_{\hat v})$ and in view of (\ref{eq2.29})
$$
|\xi_{\hat u}^*|=(1+o(r))|\xi'_{\hat u}|,\qquad |\xi_{\hat v}^*|\le(1+O(r)e^{-(\lambda_1-cr)t_z})|\xi'_{\hat v}|,
$$
with which and \eqref{eq2.30} we get the first estimate in \eqref{eq2.11}, with \eqref{eq2.26} and \eqref{eq2.28} we get the second one in \eqref{eq2.11}. The estimates in \eqref{eq2.12} can be proved in a similar way. The proof of Lemma \ref{lem2.4} is completed.
\end{proof}

\section{Continuation of periodic orbit with negative energy}\label{sec.3}
\setcounter{equation}{0}


The Hamiltonian (\ref{eq1.1}) is symmetric for the operation $\mathbf{s}:(x,y)\to (x,-y)$. With the homoclinic orbit $z^+(t)$ we studied in the last section, one obtains another homoclinic orbit $z^-(t)=\mathbf{s}z^+(-t)$. Such a symmetry may be destroyed during the transformation introduced in the proof of Proposition \ref{flatpro}. 
However, the Hamiltonian flow $\Phi_H^t$ still admits two homoclinic orbits $z^+(t)$ and $z^-(t)$ such that $[z^+(t)]=-[z^-(t)]$, the hypotheses ({\bf H2}) and ({\bf H4}) hold.

Recall $\Sigma^-_{\pm r}=\{u_1=\pm r\}$, $\Sigma^+_{\pm r}=\{v_1=\pm r\}$ and $\Sigma_{E,\pm r}^{\pm}=H^{-1}(E)\cap\Sigma^\pm_{\pm r}$. The homoclinic orbit $z^+(t)$ intersects the sections at the points $z^+_{r}$ and $z^-_{r}$, the orbit $z^-(t)$ intersects the sections at the points $z^-_{-r}$ and $z^+_{-r}$ respectively. In $(u,v)$-coordinate, 
$$
z^-_{\pm r}=(\pm r,\hat u^-_{\pm r},v^-_{\pm r}), \qquad z^+_{\pm r}=(u^+_{\pm r},\pm r, \hat v^+_{\pm r}).
$$
\begin{defi}\label{def3.1}
For small $\delta>0$, let $U_{E,\pm r}\subset H^{-1}(E)\cap\Sigma^-_{\pm r}$ be the subset such that
$$
\hat\pi U_{E,\pm r}=\{|\hat u-\hat u^-_{\pm r}|\le\delta,|\hat v-\hat v^-_{\pm r}|\le \delta\}.
$$
Let $S_{E,\pm r}\subseteq \Phi_{E,\pm r}U_{E,\pm r}$ be the set such that the inner map is well defined. 
\end{defi}

In contrast with the outer map $\Phi_{E,\pm r}$ which is well-defined for any $E\in[-E_0,E_0]$, the inner map $\Phi_{E,r,r}$ is valid only for $E>0$. Forced by Lemma \ref{pro2.2}, we get inner maps $\Phi_{E,r,-r}$ and $\Phi_{E,-r,r}$ for small $E<0$.

Similar to the inner map $\Phi_{E,r,\pm r}$, the inner map $\Phi_{E,-r,r}$: $S_{E,-r}\to\Sigma^-_{E,r}$ is defined as follows: for $z\in S_{E,-r}$, the orbit $\Phi_H^t(z)$ remains in $B_r(0)$ until it arrives at a point $z'\in\Sigma^-_{E,r}$, we set $\Phi_{E,-r,r}(z)=z'$. Due to Lemma \ref{pro2.2}, the energy $E$ must be negative. We also define the projection of the maps such that $\hat\pi\Phi_{E,\pm r,\mp r}(z)=\hat\Phi_{E,\pm r,\mp r}(\hat z)$. As a convention of notation, the selection of  $+$ in $\pm$ leads to the selection of $-$ in $\mp$, e.g. there are only two cases for $\Phi_{E,\pm r,\mp r}$, either $\Phi_{E,r,-r}$ or $\Phi_{E,-r,r}$ because $E<0$.

We have the following results similar to Proposition \ref{pro2.5} plus Lemma \ref{lem2.4}. The proof is also almost the same. Recall the definition of the cones $\hat K^{\pm}_{\alpha}$, $K^\pm_{\alpha,1}$ and $K^\pm_\alpha$. In the proof of Lemma \ref{lem2.4}, the range for $\alpha$ is defined, $\alpha\in(\max\{\hat\alpha^-_r,\frac{cr}{\lambda_1-cr}\}, \min\{\hat\alpha^+_r,\frac{\lambda_1}{cr}-1\})$, both $K^-_{\alpha,1}$ and $K^-_\alpha$ are all invariant. Notice that $\max\{\hat\alpha^-_r,\frac{cr}{\lambda_1-cr}\}\to 0$ as $r\to 0$.

\begin{pro}\label{theo3.1}
Some small $E_0>0$ exists such that for any $E\in[-E_0,0)$, the map $\Phi_{E,\pm r,\mp r}$ expands $S_{E,\pm r}$ in $\hat u$-component such that $\Phi_{E,\pm r,\mp r}S_{E,\pm r}$ covers $\{|\hat u-\hat u^-_{\mp r}|\le r\}$ in the sense
\begin{equation}\label{eq3.1}
\pi_{u}\hat\pi\Phi_{E,\pm r,\mp r}S_{E,\pm r}\supseteq\{|\hat u-\hat u^-_{\mp r}|\le r\},
\end{equation}
and it contracts $S_{E,\pm r}$ in the $\hat v$-component such that for $\bar\lambda=\min\{\lambda_2-cr,2\lambda_1-cr)\}$
\begin{equation}\label{eq3.2}
\pi_{v}\hat\pi(\Phi_{E,\pm r,\mp r}S_{E,\pm r}\cap U_{E,\mp r})\subseteq\{|\hat v-\hat v^-_{\mp r}|\le cr^{3-2c'r}|E|^{1-c'r}\}.
\end{equation}

The differential of the map $\hat\Phi_{E,\pm r,\mp r}$ is hyperbolic. For $(\xi_{\hat u},\xi_{\hat v})\in\hat K^-_{\alpha}$ with $\hat z\in\hat S_{E,\pm r}$, let $(\xi_{\hat u}^*,\xi_{\hat v}^*)=d\hat\Phi_{E,\pm r,\mp r}(\hat z)(\xi_{\hat u},\xi_{\hat v})$. Then there exist constants $c,c'>0$ such that
\begin{equation}\label{eq3.3}
|\xi_{\hat u}^*|\ge e^{(\lambda_2-cr)t_z}|\xi_{\hat u}|, \qquad |\xi_{\hat v}^*|\le c're^{-(\lambda_1-cr)t_z}t_z|\xi_{\hat u}^*|
\end{equation}
where $t_z$ is the time for $\Phi_H^{t_z}z$ arrives at $\{u_1=\mp r\}$.
\end{pro}
\begin{proof}
The set $S_{E,\pm r}$ is treated as a union of the graphs $\mathcal{G}_{\Phi^*_{E,r}F}$, the proof of Proposition \ref{pro2.5} applies here. The proof of (\ref{eq3.1}) is contained in the proof of (\ref{eq2.6}) in Proposition \ref{pro2.5}, see (\ref{eq2.10}). That $E<0$ implies $u_1(t_E)v_1(-t_E)=-r^2$. The proof of (\ref{eq3.2}) is the same as (\ref{eq2.7}). The estimates in \eqref{eq3.3} are proved in Lemma \ref{lem2.4}.
\end{proof}

With the property established in Proposition \ref{theo3.1}, we are able to construct a Smale horseshoe shown in the following figure
\begin{figure}[htp] 
  \centering
  \includegraphics[width=7.0cm,height=4.7cm]{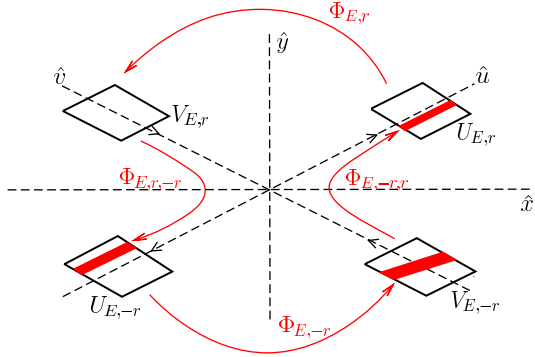}
  \caption{Smale horseshoe for $E<0$.}
  \label{fig3}
\end{figure}

According to Proposition \ref{theo3.1}, the set $\Phi^{-1}_{E,r}S_{E,r}\subseteq U_{E,r}$ is mapped by $\Phi_{E,r,-r}\Phi_{E,r}$ to a set which intersects the set $U_{E,-r}$ in the way such that $\pi_u\hat\pi\Phi_{E,r,-r}S_{E,r}\supset\pi_u\hat\pi U_{E,-r}$ and $\pi_{v}\hat\pi(\Phi_{E,r,-r}S_{E,r}\cap U_{E,-r})\subset\{|\hat v|\le cr^{3-2c'r}|E|^{1-c'r}\}$.

Next, we consider how the map $\Phi_{E,-r}$ acts on the set $\Phi_{E,r,-r}S_{E,r}$. Because of the transversal intersection property ({\bf H2}), we have $\mathrm{det}(A_{11}(z))\ne 0$ (cf. Lemma \ref{lem2.2}) if we write
$$
d\hat\Phi_{E,-r}=\left[\begin{matrix}
A_{11}(E,z) & A_{12}(E,z) \\ A_{21}(E,z) & A_{22}(E,z)\end{matrix}\right],\quad \mathrm{for}\ z\in U_{E,-r}.
$$
For $(\xi_{\hat u},\xi_{\hat v})$ such that $|\xi_{\hat u}|\gg|\xi_{\hat v}|$, there exists some $\eta>0$ such that $|\xi_{\hat u}^*|\ge\eta|\xi_{\hat v}^*|$ if we write $(\xi_{\hat u}^*,\xi_{\hat v}^*)=d\Phi_{E,-r}(\xi_{\hat u},\xi_{\hat v})$. It implies that some $\delta>0$ exists such that
$$
\pi_u\hat\pi\Phi_{E,-r}(\Phi_{E,r,-r}S_{E,r}\cap U_{E,-r})\supseteq\{|\hat u|\le\delta\}.
$$
So, in the same way to prove \eqref{eq3.1} and \eqref{eq3.2}, one can see that there exists some set
$$
S_{E,-r}\subseteq \Phi_{E,-r}(\Phi_{E,r,-r}S_{E,r}\cap U_{E,-r})
$$
so that $\pi_u\hat\pi\Phi_{E,-r,r}S_{E,-r}\supseteq\pi_u\hat\pi U_{E,r}$ and $\pi_{v}\hat\pi\Phi_{E,-r,r}S_{E,-r}\subseteq\{|\hat v|\le cr^{3-2c'r}|E|^{1-c'r}\}$.  Let $\Phi_E=\Phi_{E,-r,r}\Phi_{E,-r}\Phi_{E,r,-r}\Phi_{E,r}$.

\begin{theo}\label{theo3.2}
There exists $E_0>0$ such that for each $E\in[-E_0,0)$, there exists a $C^1$-map $F_E\in C^1(\pi_u\hat\pi U_{E,r},\pi_v\hat\pi U_{E,r})$ satisfying the condition $\mathcal{G}_{F_E}\subseteq\Phi_E \mathcal{G}_{F_E}$. Restricted on $\mathcal{G}_{F_E}$ the inverse of $\Phi_E$ is a contraction map. Consequently, there is a fixed point $z_{E,r}$ of $\Phi_E$ lying in $\mathcal{G}_{F_E}$.
\end{theo}
\begin{proof}
Similar to the proof of Proposition \ref{theo2.2}, let $\mathscr{F}^\pm=\{F\in C^1(\pi_u\hat\pi U_{E,\pm r},\pi_v\hat\pi U_{E,\pm r}):\|DF\|\le\eta\}$ be a set of maps with suitably small $\eta>0$. By applying the proof of Proposition \ref{theo2.2}, we see that the map $\Phi_{E,r,-r}\Phi_{E,r}$ induces a transformation $F\in\mathscr{F}^+\to(\Phi_{E,r,-r}\Phi_{E,r})^*F\in\mathscr{F}^-$ such that $\Phi_{E,r,-r}\Phi_{E,r}\mathcal{G}_F= \mathcal{G}_{(\Phi_{E,r,-r}\Phi_{E,r})^*F}$ and
$$
\|(\Phi_{E,r,-r}\Phi_{E,r})^*F_1-(\Phi_{E,r,-r}\Phi_{E,r})^*F_2\|<\mu\|F_1-F_2\|
$$
holds for any two maps $F_1,F_2\in\mathscr{F}$ with $0<\mu<1$.

In the same reason, we see that the map $\Phi_{E,-r,r}\Phi_{E,-r}$ also induces a transformation $F\in\mathscr{F}^-\to(\Phi_{E,-r,r}\Phi_{E,-r})^*F\in\mathscr{F}^+$ such that $\Phi_{E,-r,r}\Phi_{E,-r}\mathcal{G}_F=\mathcal{G}_{(\Phi_{E,-r,r} \Phi_{E,-r})^*F}$ is a contraction map also. Since $\Phi_E$ is the composition of the two maps $\Phi_{E,-r,r}\Phi_{E,-r}$ and $\Phi_{E,r,-r}\Phi_{E,r}$, it induces a transformation on $\mathscr{F}$: $F\to\Phi_E^*F$ which is obviously a contraction map either. Therefore, there exists a unique fixed point $F_E$ of the map $\Phi_E^*$. Restricted on the graph of $F_E$, the inverse map $\Phi_E^{-1}$ is also contracting.

By Banach's fixed point theorem, $\Phi_E$ has a unique fixed point $z_{E,r}$ in the graph $\mathcal{G}_{F_E}$. It corresponds to a periodic orbit $z_E(t)$ on negative energy level set $H^{-1}(E)$.
\end{proof}

\section{Periodic orbit with compound type homology class}\label{sec.5}
The continuation of periodic orbits takes place not only from single homoclinic orbit but also from a compound of homoclinic orbits. 

\begin{theo}\label{theo5.1}
Assume $k$ pairs of homoclinic orbits $\{z^\pm_1(t),\cdots,z^\pm_k(t)\}$ satisfying the hypotheses $(${\bf H1},{\bf H2}$)$. There exist $E_0>0$ such that for each $E\in(0,E_0]$ there exists a unique periodic orbit $z^+_E(t)$ $(z^-_E(t)$ resp.$)$ which shadows the orbits $\{z^+_1(t),\cdots,z^+_k(t)\}$ $(\{z^-_k(t),\cdots,z^-_1(t)\}$ resp.$)$ in the prescribed order. As a subset in $\mathbb{T}^n\times\mathbb{R}^n$ depending on $E$, $\cup_tz^{\pm}_E(t)$ approaches $\Gamma^\pm=\cup_i\cup_tz^{\pm}_i(t)$ in Hausdorff metric as $E\downarrow0$;
\end{theo}

\begin{proof}
Recall $\Sigma^-_{\pm r}=\{u_1=\pm r\}$ and $\Sigma^+_{\pm r}=\{v_1=\pm r\}$. Let $z^{\pm}_{r,i}$ denote the point where the homoclinic orbit $z^+_i(t)$ intersects the section $\Sigma^{\pm }_{r}$ respectively and let $z^{\pm}_{- r,i}$ denote the intersection point of the homoclinic orbit $z^-_i(t)$ with the section $\Sigma^{\pm}_{-r}$ respectively. Let $\hat U_{\pm r,i}$ a cube centered at $\hat z^-_{\pm r,i}$ with side length $2\delta$, namely,
$$
\hat U_{\pm r,i}=\{\hat z\in\mathbb{R}^{2(n-1)}:|\hat z-\hat z^-_{\pm r,i}|<\delta\}.
$$
It uniquely determines a set $U_{E,\pm r,i}\subset H^{-1}(E)$ for small $E$ such that $\hat\pi U_{E,\pm r,i}=\hat U_{\pm r,i}$. Once $r>0$ is fixed, some suitably small $\delta>0$ exists such that $U_{E,\pm r,i}\cap U_{E,\pm r,j}=\varnothing$ if $i\ne j$.

Similar to the case of single homology class for $E>0$, the periodic orbit is found by searching for invariant graph via Banach's fixed point theorem. A map $F$: $\pi_u \hat U_{r,i}\to \pi_v \hat U_{r,i}$ determines a graph $\mathcal{G}_{F,E}=\cup_{\hat u}(r,\hat u,v_1(E),F(\hat u))\subset H^{-1}(E)$. For each small $E$, $\mathcal{G}_{F,E}$ is sent by the outer map $\Phi_{E,r}$ to a graph $\Phi_{E,r}\mathcal{G}_{F,E}$, because the submatrix $A_{11}$ of \eqref{eq2.13} is non-degenerate in the sense $\mathrm{det}A_{11}\ne 0$, guaranteed by the hypothesis ({\bf H2}). Around the point $z^+_{r,i}$ it follows from $\partial_{u_1}H=\lambda_1r+o(r)$ that $\partial_E u_1=-(\lambda_1r+o(r))^{-1}$. It implies that  $\mathcal{G}_i=\cup_{E\in[-E_0,E]}\Phi_{E,r}\mathcal{G}_{F,E}$ is a graph over $\{|u|\le\delta'\}$ for some $\delta'>0$, i.e. $\pi_u\mathcal{G}_i\supset\{|u|\le\delta'\}$. Let
\begin{equation}\label{compound}
\Phi_t\mathcal{G}_{i}=\{\Phi^t_H(z):z\in\mathcal{G}_{i},|\Phi_H^s(z)|\le 2r,\ \forall\ s\in[0,t]\},
\end{equation}
It follows from the first inequality of \eqref{eq2.8} that some $t_1\le\frac 1{\lambda_1-cr}(\ln 2r-\ln\delta')$ exists such that for $t\ge t_1$ one has $\pi_u\Phi_t\mathcal{G}_{i}=\{|u|\le 2r\}$.

Let $\Pi_{r,j}=U_{r,j}\cap(\cup_{t\ge t_1}\Phi_t\mathcal{G}_i)$. Lemma \ref{pro2.2} implies $H(z)>0$ for any $z\in\Pi_{r,j}$ and the second inequality of \eqref{eq2.8} implies that for any $z\in\Phi_t\mathcal{G}_{i}$ one has $d(z,\{u=0\})\to 0$ as $t\to\infty$. Hence, by applying Lemma \ref{foreq2.8} we see that it admits a foliation of energy level sets
$
\Pi_{r,j}=\cup_{E\in(0,E_0]}\Pi_{E,r,j}
$
such that $\pi_u\Pi_{E,r,j}\supseteq\{|\hat u-\hat u^-_{r,j}|\le\delta\}\cap\Sigma^-_r$ for any small $E>0$. 

Therefore, the Hamiltonian flow $\Phi_H^t$ establishes 1-1 correspondence between $\mathcal{G}_F$ and $\Pi_{E,r,j}$, namely, a map $\Phi_{E,i,j}$ exists such that $\Pi_{E,r,j}=\Phi_{E,i,j}\mathcal{G}_F$. Because
$\Pi_{E,r,j}$ is the graph of some function $F_{E,j}$ defined on $\{|\hat u-\hat u^-_{r,j}|\le\delta\}$, we get a map $\Phi^*_{E,i,j}$ such that $F_{E,j}=\Phi^*_{E,i,j}F$

We extend the set $\cup_{E\in(0,E_0]}\Phi_{E,r}\Pi_{F,E,r,j}$ to the part $\{H^{-1}(E): E\in[-E_0,0]\}$ to construct a graph $\mathcal{G}_j$ such that $\pi_u\mathcal{G}_j \supset\{|u|\le\delta''\}$ for some $\delta''>0$ and $\mathcal{G}_j\cap\{H^{-1}(E):E\in(0,E_0]\}= \cup_{E\in(0,E_0]}\Phi_{E,r}\Pi_{F,E,r,j}$. Since we are only concerned about the graph in positive energy level sets, the extension of $\cup_{E\in(0,E_0]}\Phi_{E,r} \Pi_{F,E,r,j}$ to negative energy part can be arbitrary. Let
$\Phi_t\mathcal{G}_{j}=\{\Phi^t_H(z):z\in\mathcal{G}_{j},|\Phi_H^s(z)|\le Kr,\ \forall\ s\in[0,t]\}$, then $\Pi_{r,k}=U_{r,k}\cap(\cup_{t\ge t_1}\Phi_t\mathcal{G}_j)$ admits a foliation of energy level sets $\Pi_{r,k}=\cup_{E\in(0,E_0]}\Pi_{E,r,k}$. Again, the Hamiltonian flow $\Phi_H^t$ establishes the 1-1 correspondence $\Phi_{E,j,k}$: $\Pi_{E,r,j}\to\Pi_{E,r,k}$ and the associated map $\Phi^*_{E,j,k}$ such that $\Pi_{E,r,k}=\mathcal{G}_{\Phi^*_{E,j,k}\Phi^*_{E,i,j}F}$.

Repeating the process for $i=1,2,\cdots,k$, we obtain the transformations $\Phi^*_{E,i,i+1}$ mod $k$. The composition of the transformations $\Phi_E=\Pi_{i=1}^k\Phi_{E,i,i+1}$ maps the graph $\mathcal{G}_E$ to a graph $\Phi_E\mathcal{G}_E$ over $\{|\hat u-\hat u^-_{r,i}|\le\delta\}$. In the same way to prove Theorem \ref{theo2.2}, we see that each map $\Phi^*_{E,i,i+1}$ is a contraction map. Thus, $\Phi^*_E=\Pi_{i=1}^k\Phi^*_{E,i,i+1}$ is a contraction map from $\mathscr{F}=\{F\in C^1(\pi_u\hat\pi U_{E,\pm r,i},\pi_v\hat\pi U_{E,\pm r,i}):\|F\|\le\eta\}$ to itself. Therefore, there exists a unique invariant function $F_E$ such that $\Phi^*_EF_E=F_E$. Restricted on $\mathcal{G}_{F_E}$, the map $\Phi_E^{-1}$ is a contraction. The existence of the fixed point of $\Phi_E$ proves the existence of the periodic orbit $z^+_E(t)$ that shadows the orbits $\{z^+_1(t),\cdots,z^+_k(t)\}$ in the prescribed order. Because of the $\mathbf{s}$-symmetry, the orbit $z^-_E(t)=\mathbf{s}z^+_E(t)$ shadows the orbits $\{z^-_k(t),\cdots,z^-_1(t)\}$ in the order.
\end{proof}

\section{Uniqueness of the periodic orbit}\label{sec.4}
We are going to show that there exists only one periodic orbit in each level set which entirely lies in a small neighborhood of the homoclinic orbit(s). We study the periodic orbit shadowing a single homoclinic orbit first.
For an orbit $z(t)$, let $S(z(t))=\overline{\{z(t):t\in\mathbb{R}\}}$. Let $d_H(S_1,S_2)$ denote the Hausdorff distance between two set $S_1,S_2$.

\begin{theo}\label{theo4.1}
Some $E_0>0$ exists, for any $E\in(0,E_0]$ the level set $H^{-1}(E)$ contains exactly one periodic orbit $z^\pm_E(t)$ such that $d_H(S(z^\pm_E(t)),S(z^\pm(t)))\to 0$ as $E\to 0$.
\end{theo}
\begin{proof}
A periodic orbit $z^+_E(t)$ corresponds to a fixed point $z^-_{E,r}$ of the map $\Phi_{E,r,r}\Phi_{E,r}$. When $z^+_E(t)$ moves from $z^+_{E,r}$ to $z^-_{E,r}=\Phi_{E,r,r}z^+_E(t)$, its $\hat z$-component remains in $o(r)$-neighborhood of $\hat z=0$ while its $v_1$-component decreases from $v_1=r$ to $v_1=0$.

The inner map $\Phi_{E,r,r}$ is defined only on a subset $S_{E,r}$ of $\Phi_{E,r}(U_{E,r})$. Starting from $z\in\Phi_{E,r}(U_{E,r})\backslash S_{E,r}$, the orbit may still hit the cube $U_{E,r}$ after it passes some part outside $B_{r'}$. Therefore, the flow $\Phi_H^t$ defines a map $\Phi'_{E,r,r}$ from some part $S'_{E,r}\supset S_{E,r}$ of $\Phi_{E,r}(U_{E,r})$, by which the set $S'_{E,r}$ will be stretched and folded such that the set $\Phi'_{E,r,r}\Phi_{E,r}(U_{E,r})$ may intersect the cube $U_{E,r}$ several times. It results in the existence of Smale horseshoe. At first glance, there are $k$ fixed points if the set $\Phi_{E}(U_{E,r})\cap U_{E,r}$ contains $k$ connected components. Each fixed point corresponds to a periodic orbit of $\Phi_H^t$ lying in the energy level set $H^{-1}(E)$. However, the multiplicity of the fixed points does not damage the unique continuation of periodic orbits from homoclinical orbit.

By the definition, a point $z$ is said to lie in $S_{E,r}$ if and only if, starting from the point $z\in\Sigma_{E,r}^+$, the orbit $\Phi_H^t(z)$ remains in the ball $\{|z|\le r'\}$ before it touches the section $\Sigma^-_{E,r}$ after a time $t_z$. It has been proved that passing through $S_{E,r}$ there is only one periodic orbit, which corresponds to the fixed point lying in the graph of an invariant function $\Phi^*_EF_E=F_E$. Restricted on the graph, the map $\Phi_E$ has only one fixed point. Therefore, if there is another periodic orbit $z'_E(t)\subseteq H^{-1}(E)$ that intersects $V_{E,r}$ at a point not in $S_{E,r}$, there must be a point on the orbit $z'_E(t^*)=(u^*_1,\hat u^*,v^*_1,\hat v^*)$ such that $|\hat z'_E(t^*)|>r'$ while $0<|v^*_1|<r$.

By the hypothesis ({\bf H2}), the homoclinic orbit approaches to origin in direction of the eigenvector for $\lambda_1$. If we write the homoclinic orbit $z^+(t)=(u_1^+(t),\hat u^+(t),v_1^+(t),\hat v^+(t))$, the hypothesis ({\bf H2}) implies that $|\hat u^+(t),\hat v^+(t)|=o(r')$ if $|v_1^+(t)|\le r'$.
Therefore, no matter how small the energy $E>0$ will be, any periodic orbit $z'_E(t)\subset H^{-1}(E)$ other than $z_E(t)$ will deviate from the homoclinic orbit $z^+(t)$ if it passes through the section $\Sigma^+_{E,r}$ at some point not contained in $S_{E,r}$. Thus, one has an estimate on the Hausdorff distance
$d_H(S(z'_E(t)),S(z^+(t)))\ge r'-o(r')$ for all small $E>0$. We illustrate the situation by the following figure.
\begin{figure}[htp] 
  \centering
  \includegraphics[width=8.5cm,height=3.5cm]{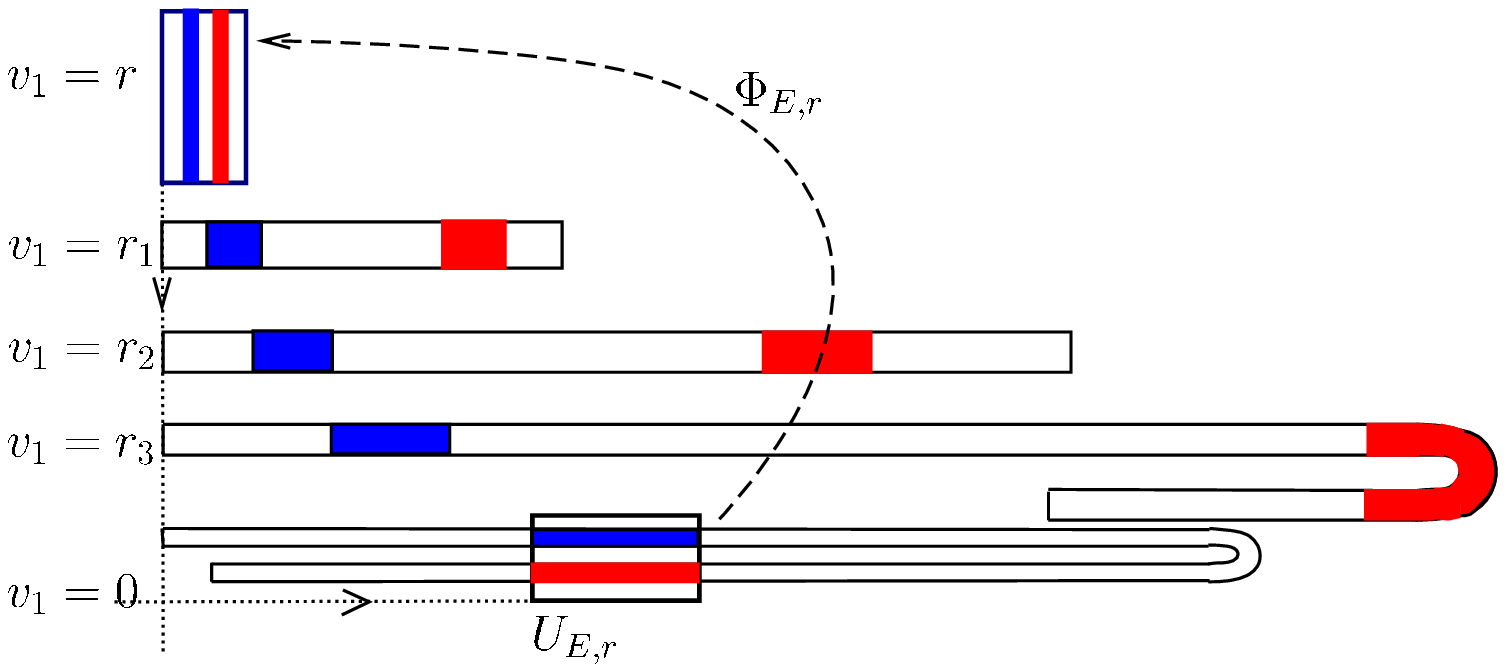}
  \label{fig4}
\end{figure}

The rectangle $\pi U_{E,r}\subset\Sigma^-_{E,r}$ is mapped to a set lying in $\Sigma^+_{E,r}$ containing the vertical rectangle. When $v_1$ decreases from $r$ through $r_1>r_2>r_3$ approaching $v_1=0$, the vertical rectangle is stretched in $\hat u$ and compressed in $\hat v$ and folded. The blue strip always stays in $B_{r'}$ as $v_1$ decreases from $r$ to $0$, while the red strip has to pass through some place outside of $\{|\hat z|\le r'\}$.

From \eqref{eq2.7} we see the $\hat v$-component of $\hat z^-_{E,r}$ is at least $c|E|^{1-cr}$-close to that of $\hat z^-_r$. In the same principle, we derive that the point $\hat z^+_{E,r}$ falls into a strip $\{\hat z:|\hat u|\le c|E|^{1-cr}\}$, i.e. $|\hat u^+_{E,r}|\le c|E|^{1-cr}$.

To measure how the $\hat u$-component of $\hat z^-_{E,r}$ deviates from $\hat z^-_r$, we apply the transversal intersection property ({\bf H2}), it implies $\mathrm{det}(A_{11}(E,\hat z))\ne 0$ (cf. Lemma \ref{lem2.2}) if we write
\begin{equation*}\label{zhuanyijuzhen}
d\hat\Phi_{E,r}(\hat z)=\left[\begin{matrix}
A_{11}(E, \hat z) & A_{12}(E, \hat z) \\ A_{21}(E, \hat z) & A_{22}(E, \hat z)\end{matrix}\right],\quad \mathrm{for}\ z\in U_{E,r}.
\end{equation*}
Because $\hat u^+_r=0$, $\hat v^-_r=0$ and $|\hat z^-_{E,r}-\hat z^-_r|$ is small, we have
$$
\hat u^+_{E,r}=\pi_u(z^+_{E,r}-z^+_r)=A_{11}(\hat u^-_{E,r}-\hat u^-_r)+ A_{12}\hat v^-_{E,r}+\partial_E\hat\Phi_{E,r} E
$$
where $A_{11}$, $A_{12}$ and $\partial_E\hat\Phi_{E,r}$ are valued at some place between $z^-_{E,r}$ and $z^-_r$. It follows from $\mathrm{det}A_{11}\ne 0$, $|\hat u^+_{E,r}|\le c|E|^{1-cr}$ and $|\hat v^-_{E,r}|\le c|E|^{1-cr}$ that $|\hat u^-_{E,r}-\hat u^-_r|\le c'|E|^{1-cr}$. By applying the same method, we also see $|\hat v^+_{E,r}-\hat v^+_{r}|\le c'|E|^{1-cr}$. Consequently, $z^+_{E}(t)$ keeps $c'|E|^{1-cr}$-close to the homoclinic orbit when it moves from the section $\{u_1=r\}$ to $\{v_1=r\}$. It leads to the conclusion that $d_H(S(z^\pm_E(t)),S(z^\pm(t)))\to 0$ as $E\to 0$.
\end{proof}

The estimate on the position of $\hat z^-_{E,r}$ is not so precise that can be used to study the smoothness of the cylinder. We shall get more precise estimation later.

The idea is applicable to prove the same result for the periodic orbit in the case of compound type homology class as well as the case $E<0$.

\begin{theo}\label{theo4.2}
In the case of compound type homology class, there exists $E_0>0$ such that for each $E\in(0,E_0]$, the level set $H^{-1}(E)$ admits exactly one periodic orbit $z^\pm_E(t)$ which entirely lies in the vicinity of $\cup z^\pm_i(t)$ such that
\begin{equation*}
d_H(S(z^\pm_E(t)),S(\cup_iz^\pm_i(t)))\to 0\qquad \mathrm{as}\ E\downarrow0.
\end{equation*}
For each $E\in[-E_0,0)$, the level set $H^{-1}(E)$ admits exactly one periodic orbit $z_{E,i}(t)$ which entirely lies in the vicinity of $z^+_i(t)\cup z^-_i(t)$ such that
\begin{equation*}
d_H(S(z_{E,i}(t)),S(z^+_i\cup z^-_i(t)))\to 0\qquad \mathrm{as}\ E\uparrow0.
\end{equation*}
\end{theo}
\begin{proof}
In the case of compound type homology class, it corresponds to the fixed point of $\Phi^k_E=(\Phi_{E,r,r}\Phi_{E,r})^k$. The orbit passes through $B_r(0)$ for $k$ times.

During each time when the orbit passes through the neighborhood, we have a Smale horseshoe which may contain many strips. Each strip determines a periodic orbit. As shown in the figure right above, the orbit we got must stay in the only strip that is entirely contained in the neighborhood when the $v_1$-coordinate decreases from $r$ to $0$ (the blue strip in the figure). Any other periodic orbit $z'_E(t)$, if it passes through $V_{E,r}$, it shall not intersect the set $S_{E,r}$. It implies that $z'_E(t)$ shall pass through some place out of $B_{r'}$ before it returns back to the cube $U_{E,r}$. In other words, $z'_E(t)$ does not lie entirely in some neighborhood of $S(z^+(t)\cup z^+(t))$ no matter how small the energy $E$ is. Restricted on the strip that entirely lies in the neighborhood of the origin, the uniqueness is guaranteed by Banach's fixed point. Indeed, the flow is hyperbolic when it is restricted in the strip, it allows  only one fixed point. It proves the uniqueness. The argument also applies to the case $E<0$ to show the uniqueness.

To show the convergence, let $z^\pm_{E,r,i}$ denote the point where the periodic orbit $z^+_{E}(t)$ intersects the section $\Sigma^\pm_{E,r}$ which is close to the point $z^\pm_{r,i}$ where the homoclinic orbit $z^+_i(t)$ intersects the section $\Sigma^\pm_{r}$. Let $\hat z^\pm_{E,r,i}=(\hat u^\pm_{E,r,i}, \hat v^\pm_{E,r,i})$, then $|\hat u^+_{E,r,i}|\le c|E|^{1-cr}$ and $|\hat v^-_{E,r,i}|\le c|E|^{1-cr}$. Because $\mathrm{det}A_{11}\ne 0$, it also follows from
$$
\hat u^+_{E,r,i}=\hat\pi(z^+_{E,r,i}-z^+_r)=A_{11}(\hat u^-_{E,r,i}-\hat u^-_{r,i})+ A_{12}\hat v^-_{E,r,i}+\partial_E\hat\Phi_{E,r} E
$$
that $|\hat u^-_{E,r,i}-\hat u^-_{r,i}|\le c'|E|^{1-cr}$. By the same method, we see $|\hat v^+_{E,r,i}-\hat v^+_{r,i}|\le c'|E|^{1-cr}$ also. So, $z^+_{E}(t)$ keeps $c'|E|^{1-cr}$-close to the homoclinic orbit $z^+_i(t)$ when it moves from the section $\{u_1=r\}$ to $\{v_1=r\}$. Thus we have $d_H(S(z^\pm_E(t)),(\cup_iz^\pm_i(t)))\to 0$ as $E\downarrow 0$. The case of $E<0$, let $z_{E,i}(t)$ be the period orbit shadowing $\{z^+_i(t),z^-_i(t)\}$, then the proof of $d_H(S(z^+_i(t),z^-_i(t),z_{E,i}(t))\to 0$ as $E\to 0$ is similar.
\end{proof}

We return back to the original Hamiltonian (\ref{eq1.1}). It is symmetric under the operation $\mathbf{s}$: $(x,y)\to(x,-y)$, $H(\sigma(x,y))=H(x,y)$.
An orbit $z(t)$ is called $\mathbf{s}$-{\it symmetric} if the set $S(z(t))=\overline{\{z(t):t\in\mathbb{R}\}}$ is invariant for the operation $\sigma$, i.e. $S(z(t))=\mathbf{s} S(z(t))$.

\begin{pro}
The periodic orbit $z_E(t)$ for $E<0$ is $\mathbf{s}$-symmetric and passes through the section $\{y=0\}$ twice during one period.
\end{pro}
\begin{proof}
As it has been proved in the last section, the orbit $z_E(t)$ is the only periodic orbit that lies entirely in a small neighborhood of $S(z^+(t)\cup z^-(t))$. If it is not $\mathbf{s}$-symmetric, then $\sigma z_E(t)$ is also a periodic orbit lying around $S(z^+(t)\cup z^+(t))$. But it contradicts the uniqueness. By the construction of the periodic orbit, it passes through the neighborhood of point twice during one period. If it does not pass through the section $\{y=0\}$ twice during one period, it would pass through the neighborhood more than two times.
\end{proof}

The periodic orbits on each positive energy level set are related by the $\mathbf{s}$-symmetry. Once one obtains one periodic orbit $z^+_E(t)$ around the homoclinics $z^+(t)$, then $z^-_E(t)=\mathbf{s} z^+_E(t)$ is the periodic orbit around $z^-(t)$.

\section{$C^1$-smoothness of the cylinder}\label{sec.6}
By the work in the previous sections, a singular invariant cylinder has been proved to exist, illustrated in Figure \ref{fig1}. It consists of periodic orbits of $\Phi_H^t$ and some pair(s) of homoclinics
$$
\Pi=\Pi^+\cup \Pi^-\cup\Gamma^+\cup\Gamma^-.
$$
where $\Pi^+=\cup_{E\in(0,E_0]}z^+_E(t)\cup z^-_E(t)$, $\Pi^-=\cup_{E\in[-E_0,0)}\cup_{i=1}^k z_{E,i}(t)$ and $\Gamma^\pm$ is the closure of the set $\cup_{t\in\mathbb{R}} (z^\pm_1(t)\cup\cdots\cup z^\pm_k(t))$. In the case of single homology class, the topological structure is clear, $\Pi$ is a cylinder with one hole lying in negative energy region.

In the case of $k\ge 2$, the set $\Pi$ is not a manifold, although it still has nice structure. To reveal it, we work in a finite covering space $\mathbb{T}^n_h$ of $\mathbb{T}^n$. Recall the curve $\bar\Gamma$ introduced before the statement of the condition ({\bf H3}), due to which $\tilde\Gamma^+=\bar\pi_h (\bar\Gamma\ast\sigma\bar\Gamma \ast\cdots\ast\sigma^\ell\bar\Gamma)$ is a closed curve without self-intersection, shadowed by an orbit $\tilde z^+_E(t)$ in the lift of $z^+_E(t)$. Therefore, the set $\tilde\Pi^+_+=\cup_{E>0}(\cup_t\tilde z^{+}_E(t))$ is a cylinder taking $\tilde\Gamma^+$ as its boundary lying in $H^{-1}(0)$. Let $\tilde\Pi^+_-$, $\tilde\Gamma^-$ be the counterpart of $\tilde\Pi^+_+$, $\tilde\Gamma^+$ via the symmetry $\mathbf{s}$ if both are pushed forward to the original coordinates, then the set $\tilde\Pi^+_+\cup\tilde\Gamma^+$ touches the set $\tilde\Gamma^-\cup\tilde\Pi^+_-$ at $(\ell+1)k$ points, $\tilde\Pi_{\ge 0}=\tilde\Pi^+_+\cup\tilde\Gamma^+\cup \tilde\Gamma^-\cup\tilde\Pi^+_-$ is a cylinder with $(\ell+1)k$ holes. Let $D_{j,i}$ denote the holes for $j=0,1,\cdots,\ell$ and $i=1,2,\cdots,k$ and let $\partial D_{j,i}$ denote their boundary, then $\pi_h\partial D_{j,i}=\cup_{t\in\mathbb{R}} (z^+_{i}(t)\cup z^-_{i}(t))\cup\{z=0\}$. For small $E<0$, the periodic orbit $z_{E,i}(t)$ shadows $\{z^+_i(t),z^-_i(t)\}$, the set $\Pi^-_i=\cup_{E\in(0,-E_0]}\cup_t z_{E,i}(t)$ looks like an annulus, shrinkable in $\mathbb{T}^n\times\mathbb{R}^n$. The pull back of $\Pi^-_i$ to $\mathbb{T}^n_h\times \mathbb{R}^n$ consists of shrinkable annuli. $\tilde\Pi_{\ge 0}$ is connected to these annuli, denoted by $\tilde\Pi_{j,i}$, along $\{\partial D_{j,i}\}$. Let
\begin{equation}\label{tilde-Pi}
\tilde\Pi=\tilde\Pi^+_+\cup\tilde\Gamma^+\cup \tilde\Gamma^-\cup\tilde\Pi^+_-\cup_{i,j} \tilde\Pi_{j,i}
\end{equation}
it is a cylinder with $(\ell+1)k$ holes, as illustrated in Figure \ref{fig1}.
This section is devoted to the study its $C^1$-smoothness. We study the case $k=1$ first.

\begin{theo}\label{theo6.1}
In the single homology class case, $\Pi$ is a $C^{1}$-smooth cylinder with one hole, invariant for the flow $\Phi^t_H$.
\end{theo}
Both manifolds $\Pi^+$ and $\Pi^-$ consist of periodic orbits, all of them are hyperbolic. Thus, it follows from the implicit function theorem that $\Pi$ is differentiable everywhere except along the homoclinic orbits. So, the proof includes three steps. The first step is to show the differentiability of $\Pi$ at the fixed point $z=0$, the second is to show the tangent space $T_z\Pi^+$ and $T_z\Pi^-$ converges as $z$ approaches the boundary on $H^{-1}(0)$ and finally to show that $\Pi^+$ and $\Pi^-$ are $C^1$-joined together along the homoclinics.

\subsection{Differentiability at the fixed point} Restricted around the origin, $\Pi$ appears to be a graph $\mathcal{G}$ of a map $(u_1,v_1)\to\hat z(u_1,v_1)$. We will show $d\hat z(0,0)=0$. Let $z(t)=(u_{1}(t)\cdots u_{n}(t),v_{1}(t)\cdots v_{n}(t))$ be an orbit with $z(0)=(u_{1,0}\cdots u_{n,0},v_{1,0}\cdots v_{n,0})\in\mathcal{G}$. If it is not homoclinic orbit, large $t^-,t^+>0$ exist such that $|v_{1}(-t^-)|=2r$, $|u_{1}(t^+)|=2r$, $|\hat z(-t^-)|=o(r)$ and $|\hat z(t^+)|=o(r)$, see the figure below.
\begin{figure}[htp] 
  \centering
  \includegraphics[width=5.0cm,height=3.3cm]{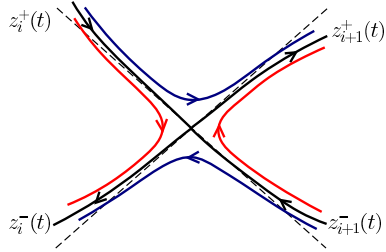}
  \caption{if $k=1$, one has $z_i^\pm(t)=z^\pm_{i+1}(t)$.}
  \label{fig5}
\end{figure}
The closer the point $z_0$ is getting to the origin, the larger the numbers $t^-$ and $t^+$ will be. According to Hartman-Grobman Theorem, there exists a conjugacy $h$ between $\Phi_H^t$ and $e^{\mathrm{diag}(\Lambda,-\Lambda)t}$ such that
$$
\Phi_H^t(u,v)=h^{-1}e^{\mathrm{diag}(\Lambda,-\Lambda)t}h(u,v),
$$
where $\Lambda=\mathrm{diag}(\lambda_1,\cdots,\lambda_n)$. If writing $h=id+f$ and $h^{-1}=id+g$, we obtain from Theorem 1.1 of \cite{vS} that $f=O(\|(u,v)\|^{1+\nu})$ and $g=O(\|(u,v)\|^{1+\nu})$ with $\nu>0$. Let $f=(f_u,f_v)$, $g=(g_u,g_v)$ and $f_u=(f_{u,1}\cdots,f_{u,n})$. The principle of notation for $f_u$ also applies to $f_v,g_u,g_v$. In the same way to get (\ref{eq2.18}) we have
\begin{equation*}
\begin{aligned}
|u_{i,0}+f_{u,i}(u_0,v_0)|&=e^{-\lambda_it^+}\Big(u_i(t^+)-g_{u,i}(e^{\Lambda t^+}(u_0+f_u),e^{-\Lambda t^+} (v_0+f_v))\Big),\\
|v_{i,0}+f_{v,i}(u_0,v_0)|&=e^{-\lambda_it^-}\Big(v_i(-t^-)-g_{v,i}(e^{-\Lambda t^-}(u_0+f_u),e^{\Lambda t^-} (v_0+f_v))\Big),
\end{aligned}
\end{equation*}
for $i=1,2,\cdots,n$.
Since $g=O(\|(u,v)\|^{1+\nu})$, $|e^{\Lambda t^+}(u_0+f_u)|,|e^{-\Lambda t^-} (v_0+f_v)|\le r$, $|u_1(t^+)|=|v_1(-t^-)|=2r$, $\hat u(t^+)=o(r)$ and $\hat v(-t^-)=o(r)$ one has
\begin{equation}\label{components}
\begin{aligned}
\frac 32re^{-\lambda_1 t^+}\le |u_{1,0}+f_{u,1}(u_0,v_0)|&\le \frac 52re^{-\lambda_1 t^+}, \\
\frac 32re^{-\lambda_1 t^-}\le |v_{1,0}+f_{v,1}(u_0,v_0)|&\le \frac 52de^{-\lambda_1 t^-},\\
|u_{i,0}+f_{u,i}(u_0,v_0)|&\le o(r)e^{-\lambda_i t^+}, \\
|v_{i,0}+f_{v,i}(u_0,v_0)|&\le o(r)e^{-\lambda_i t^-}, \quad \forall\ i\ge 2
\end{aligned}
\end{equation}
if $r>0$ is small. Let $t^*=\min\{t^+,t^-\}$, that $\lambda_i>\lambda_1$ for $i\ge 2$ results in the estimate
$$
\frac 32r e^{-\lambda_1 t^*}\le|(u_0,v_0)+f(u_0,v_0)|\le\frac 52r e^{-\lambda_1 t^*}.
$$
Since $f=O(\|(u,v)\|^{1+\nu})$, for suitably small $r>0$ one has
$$
re^{-\lambda_1 t^*}\le|(u_0,v_0)|\le 2re^{-\lambda_1 t^*}.
$$
Consequently, it follows from the first two inequalities in (\ref{components}) and the property $f=O(\|(u,v)\|^{1+\nu})$ that
\begin{equation}\label{position}
\begin{aligned}
re^{-\lambda_1 t^*}\le|(u_{1,0},v_{1,0})|\le 2re^{-\lambda_1 t^*},&\\
|(u_{i,0},v_{i,0})|\le c|(u_{1,0},v_{1,0})|^{1+\nu}+o(r)(e^{-\lambda_1 t^*})^{\frac{\lambda_i}{\lambda_1}},&\quad \forall\ i\ge 2.\\
\end{aligned}
\end{equation}
Therefore, we have
$$
|(\hat u_0,\hat v_0)|\le c|(u_{1,0},v_{1,0})|^{1+\nu'}, \qquad \nu'=\min\Big\{\nu,\frac{\lambda_2}{\lambda_1}-1 \Big\}
$$
namely, $d\hat z(u_1,v_1)|_{(u_1,v_1)=0}=0$. It proves the differentiability of $\Pi$ at $\{z=0\}$.

\subsection{$C^1$-smoothness of $\Pi^\pm$} For small $E\ne 0$, the periodic orbit intersects the section $\Sigma^-_{r}$ at the point $z^-_{E,r}$, which is a fixed point of the return map $\Phi_E$. However, no return map is defined for $E=0$. When $E\to 0$, the return time approaches infinity. It makes complicated to check the $C^1$-differentiability around the homoclinics. To this end, we apply the Birkhoff normal form \eqref{eq2.3} where $k$ satisfies the condition $k\lambda_1>\lambda_n$.

Recall $U_{E,\pm r}\subset H^{-1}(E)\cap\{u_1=\pm r\}$, $S_{E,\pm r}\subset \Phi_{E,\pm r}U_{E,\pm r}\subset H^{-1}(E)\cap\{u_1=\pm r\}$ introduced in Definition \ref{def3.1}. Let $U_{\pm r}=\cup_{|E|\le E_0} U_{E,\pm r}$, $V_{\pm r}=\cup_{|E|\le E_0} \Phi_{E,r}U_{E,\pm r}$  and $S_{\pm r}=\cup_{|E|<E_0}S_{E,\pm r}$. The Hamiltonian flow $\Phi_H^t$ defines two types of maps
\begin{enumerate}
  \item the outer map $\Phi_r$: $U_r\to V_r$. Emanating from $z\in U_r$ the orbit $\Phi_H^t(z)$ keeps close to a segment of the homoclinic orbit $z^+(t)$ that is from $z^-_{r}$ to $z^+_r$;
  \item the inner map $\Phi_{r,\pm r}$: $S_r|_{\pm E>0}\subset V_r\to U_{\pm r}$. Emanating from $z\in S_r|_{\pm E>0}$, the orbit remains in $B_r(0)$ until it reaches $U_{\pm r}$.
\end{enumerate}
Restricted on $H^{-1}(E)$ with $E\ne 0$, one has $\Phi_{r,r}=\Phi_{E,r,r}$ and $\Phi_r=\Phi_{E,r}$. In coordinate components, the inner map takes the form $\Phi_{r,r}$: $(u_1,\hat u,r,\hat v)\to(r,\hat u,v_1,\hat v)$.

\begin{defi}\label{defi6.2}
A vector $\eta=(\eta_1,\eta_{\hat u},\eta_{\hat v})$ is said to be an eigenvector of $d\Phi_{r,r}$ for the eigenvalue $\sigma$ if $d\Phi_{r,r}(\eta_1,\eta_{\hat u},0,\eta_{\hat v})=\sigma(0,\eta_{\hat u},\eta_1,\eta_{\hat v})$. A vector $\eta=(\eta_1,\eta_{\hat u},\eta_{\hat v})\in T_zS_{E,r}$ is said to be an eigenvector of $d\Phi_{E,r,r}$ for the eigenvalue $\sigma$ if some number $\eta'_1$ exists such that $d\Phi_{r,r}(\eta_1,\eta_{\hat u},0,\eta_{\hat v})=\sigma(0,\eta_{\hat u},\eta'_1,\eta_{\hat v})$ with $(\eta_{\hat u},\eta'_1,\eta_{\hat v})\in T_{\Phi_{r,r}z}U_{E,r}$.
\end{defi}

As usual, we let $e_i$ denote a unit vector whose elements are all equal to zero except for the $i$-th element which is equal to 1.

\begin{pro}\label{pro6.2}
The map $d\Phi_{r,r}$ has an eigenvalue $\sigma_1=1+o(r)$ associated with the eigenvector $\eta_1=e_1+b_1$ with $|b_1|=o(r)$. The map $d\Phi_{E,r,r}$ has $(n-1)$ pairs of eigenvalues $\{\sigma_i, \sigma_{i+n}=\sigma_i^{-1}:2\le i\le n\}$, associated with the eigenvectors $\eta_i=e_{i}+b_i$ and $\eta_{i+n}=e_{i+n-1}+b_{i+n}$ respectively, where $|b_i|,|b_{i+n}|=o(r)$, $\sigma_i=\mu_i |E|^{-\lambda_i/\lambda_1}$ with $0<\inf_E\mu_i<\sup_E\mu_i<\infty$ as $E\downarrow 0$. Let $\hat\eta_i=(\eta_{i,\hat u},\eta_{i,\hat v})$, the matrix $T_E=[\hat\eta_2,\cdots,\hat\eta_n, \hat\eta_{2+n},\cdots,\hat\eta_{2n}]$ is symplectic if a suitable factor $\nu_i=1+o(r)$ is multiplied to $\hat\eta_i\to\nu_i\hat\eta_i$ for each $i\le n$. The result also holds for $d\Phi_{-E,\pm r,\mp r} (z^+_{E,\pm r})$.
\end{pro}

We apply the proposition to check the $C^1$-smoothness first and postpone its proof to the next section. By the notation, $\eta_1=(\eta_{1,1},\eta_{1,\hat u},\eta_{1,\hat v})$ is the eigenvector of $d\Phi_{r,r}$ for $\sigma_1$, $\eta_i=(\eta_{i,1},\eta_{i,\hat u},\eta_{i,\hat v})$ and $\eta_{i+n}=(\eta_{i+n,1},\eta_{i+n,\hat u}, \eta_{i+n,\hat v})$ are the eigenvector of $d\Phi_{E,r,r}$ for $\sigma_i$ and $\sigma_i^{-1}$ respectively. By Definition \ref{defi6.2}, we have $\eta'_1=(\eta_{1,\hat u},\eta_{1,1},\eta_{1,\hat v})$, $\eta'_i=(\eta_{i,\hat u},\eta'_{i,n+1},\eta_{i,\hat v})\in T_{z^-_{E,r}}U_{E,r}$ and $\eta'_{i+n}=(\eta_{i+n,\hat u},\eta'_{i+n,n+1},\eta_{i+n,\hat v})\in T_{z^-_{E,r}}U_{E,r}$. We claim that some $O(r)>0$ exists such that for all $i\ge 2$, it holds that
\begin{equation}\label{bili}
\begin{aligned}
&(1-O(r))|\eta'_i|\le|\eta_i|\le (1+O(r))|\eta'_i|,\\
&(1-O(r))|\eta'_{i+n}|\le|\eta_{i+n}|\le (1+O(r))|\eta'_{i+n}|.
\end{aligned}
\end{equation}
Indeed, each pair $(\frac{\partial}{\partial u_i}, \frac{\partial}{\partial v_i})$ determines four numbers $\eta_{i,1},\eta_{i+n,1},\eta'_{i,1+n},\eta'_{i+n,1+n}$ such that $\frac{\partial}{\partial u_i}+\eta_{i,1}\frac{\partial}{\partial u_1}$, $\frac{\partial}{\partial v_i}+\eta_{i,1+n}\frac{\partial}{\partial u_1}\in T_{z^+_{E,r}}S_{E,r}$, $\frac{\partial}{\partial u_i}+\eta'_{i+n,1}\frac{\partial}{\partial u_1}$, $\frac{\partial}{\partial v_i}+\eta'_{i+n,1+n}\frac{\partial}{\partial u_1}\in T_{z^-_{E,r}}U_{E,r}$. We have $|\eta_{i,1}|=|\frac{\lambda_iv_i}{\lambda_1v_1}(1+o(r))|=O(r)$ since $v_1=r$ and $|v_i|=o(r)$ hold for small $E$, cf. the condition ({\bf H2}). Similarly,  $|\eta_{i,1+n}|, |\eta'_{i+n,1}|,|\eta'_{i+n,1+n}|=O(r)$ also hold.

Defining $E_{\hat u}=\mathrm{span}\{\eta_2,\cdots,\eta_n\}$, $E_{\hat v}=\mathrm{span}\{\eta_{n+2},\cdots,\eta_{2n}\}$, $E^+_1=\frac{\partial}{\partial u_1}\mathbb{R}$ and $E^-_1=\frac{\partial}{\partial v_1}\mathbb{R}$, we have the decomposition $T_{z^+_{E,r}}S_{r}=E^+_1\oplus E_{\hat u}\oplus E_{\hat v}$ and $T_{z^-_{E,r}}S_{r}=E^-_1\oplus E_{\hat u}\oplus E_{\hat v}$. Let $\pi_{E,1}$, $\pi_{E,\hat u}$ and $\pi_{E,\hat v}$ denote the projection from $T_{z^+_{E,r}}S_{r}$ to $E^+_1$, $E_{\hat u}$ and $E_{\hat v}$ and from $T_{z^-_{E,r}}U_{r}$ to $E^-_1$, $E_{\hat u}$ and $E_{\hat v}$ respectively. We put $E$ in the subscripts to remind that the projection $\pi_{E,\hat u}$ and $\pi_{E,\hat v}$ depend on the energy $E$. Hence, we use $\eta^\pm_E=(\eta^\pm_{E,1},\eta^\pm_{E,\hat u},\eta^\pm_{E,\hat v})$ to denote tangent vectors in the corresponding tangent spaces, where $\eta^\pm_{E,\hat u}=\pi_{E,\hat u}\eta^\pm_E$, $\eta^\pm_{E,\hat v}=\pi_{E,\hat v}\eta^\pm_E$ and $\eta^\pm_{E,1}=\pi_{1}\eta^\pm_E$.


The differential $d\Phi_{r}(z^-_{E,r})$ of the outer map $\Phi_r$ at $z^-_{E,r}$ is represented by a matrix
\begin{equation}\label{eq6.2new}
\left[\begin{matrix}\eta^+_{E,\hat u}\\ \eta^+_{E,\hat v}\\ \eta^+_{E,1}\end{matrix}\right]=\left[\begin{matrix} A_{E,11} & A_{E,12} & A_{E,13}\\
A_{E,21} & A_{E,22} & A_{E,23}\\
A_{E,31} & A_{E,32} & A_{E,33}
\end{matrix}\right] \left[\begin{matrix}\eta^-_{E,\hat u}\\ \eta^-_{E,\hat v}\\ \eta^-_{E,1}\end{matrix}\right]=A_E\left[\begin{matrix}\eta^-_{E,\hat u}\\ \eta^-_{E,\hat v}\\ \eta^-_{E,1}\end{matrix}\right].
\end{equation}
Let $\tau_E>0$ be the time to define the outer map $\Phi_r$, namely, $\Phi_H^{\tau_E}(z^-_{E,r})=z^+_{E,r}$, then it continuously depends on $E$ and have their limit as $E\to 0$. Therefore, all elements in the matrix of \eqref{eq6.2new} continuously depend on $E$ and remain bounded as $E\to 0$.
Since $\eta_i=e_{i}+o(r)$ and $\eta_{i+n}=e_{i+n-1}+o(r)$ and $r$ can be set suitably small, we find from \eqref{eq2.13} that $\mathrm{det}A_{E,11}\ne 0$. It is guaranteed by the fact that the stable and unstable manifolds intersect ``transversally" in the sense of \eqref{eq1.2}.

Let $\Delta z^\pm=z^\pm_{E',r}-z^\pm_{E,r}$. Because the time for $\Phi_H^t$ to go from $U_{E,r}$ to $V_{E,r}$ is finite, some $\nu\ge1$ exists such that
\begin{equation}\label{conservation}
\nu^{-1}|\Delta z^+|\le|\Delta z^-|\le\nu|\Delta z^+|,
\end{equation}
and it follows from \eqref{eq6.2new} that
\begin{equation}\label{eq6.?}
\begin{aligned}
\pi_{E,\hat u}\Delta z^+&=A_{E,11}\pi_{E,\hat u}\Delta z^-+A_{E,12}\pi_{E,\hat v}\Delta z^-+A_{E,13}\pi_{E,1}\Delta z^-+O(|\Delta z^-|^2),\\
\pi_{E,\hat v}\Delta z^+&=A_{E,21}\pi_{E,\hat u}\Delta z^-+A_{E,22}\pi_{E,\hat v}\Delta z^-+A_{E,23}\pi_{E,1}\Delta z^-+O(|\Delta z^-|^2).
\end{aligned}
\end{equation}
The inner map also establishes a relation between $\Delta z^-$ and $\Delta z^+$, $\Delta z^-=d\Phi_{r,r}\Delta z^++O(|\Delta z^+|^2)$. Hence, for sufficiently small $|\Delta z^+|$ and in view of \eqref{bili}, we have
\begin{equation}\label{equation6.4}
|\pi_{E,\hat u}\Delta z^+|\le\frac1{2\mu_2}|E|^{\frac{\lambda_2}{\lambda_1}} |\pi_{E,\hat u}\Delta z^-|,\qquad |\pi_{E,\hat v}\Delta z^-|\le \frac 1{2\mu_2}|E|^{\frac{\lambda_2}{\lambda_1}}|\pi_{E,\hat v}\Delta z^+|.
\end{equation}

\begin{lem}\label{lem6.1}
For small $E>0$, let $z^-_{E,r}$ denote the point where the periodic orbit $z_E(t)$ intersects the section $\Sigma^-_r$. Then, the tangent vector $\partial _Ez^-_{E,r}$ of $\Pi|_{\Sigma^-_r}$ at the point $z^-_{E,r}$ satisfies the condition that $\pi_{E,1}\partial_E z^-_{E,r}=\frac1{\mu_1r}(1+O(r))$, $\pi_{E,\hat v}\partial_E z^-_{E,r}=o(|E|)$ and
$$
\pi_{E,\hat u}\partial_E z^-_{E,r}= -A^{-1}_{E,11}A_{E,13}\pi_{1}\partial_E z^-_{E,r}+o(|E|).
$$
\end{lem}
\begin{proof}
The cylinder $\Pi|_{E>0}$ is obviously smooth, it makes sense to consider the tangent space $T_z\Pi|_{E>0}$ for each $z\in\Pi|_{E>0}$. It follows from \eqref{conservation} and the second inequality of \eqref{equation6.4} that $|\pi_{E,\hat v}\partial_E z^-_{E,r}|\le c|E|^{\lambda_2/\lambda_1}|\partial_E z^-_{E,r}|$, with which and the first inequality of \eqref{equation6.4} we find from the first equation of \eqref{eq6.?} that
$$
|A_{E,11}\pi_{E,\hat u}\partial_E z^-_{E,r}+A_{E,13}\pi_{E,1}\partial_E z^-_{E,r}|\le c_2|E|^{\lambda_2/\lambda_1}|\partial_E z^-_{E,r}|.
$$
Since $\mathrm{det}A_{E,11}\ne 0$, the term $|\pi_{E,\hat u}\partial_E z^-_{E,r}|$ is controlled by $\partial_E z^-_{E,r}$.

Restricted on the section $\{u_1=r\}$, we take first derivative in $E$ on both sides of the equation $H(z)=E$ at $z^-_{E,r}$, we get
$$
1=\lambda_1r\partial_E v_1+\sum_{i=2}^n\lambda_iu_i\partial_E v_i+ \lambda_iv_i\partial_E u_i+\partial_E R.
$$
At the point $z^-_{E,r}$ we have $|u_i|,|v_i|=o(r)$ for $i\ge 2$. So, it follows from the definition of $\eta'_1$, $\eta'_i$ and $\eta'_{i+n}$ that
$1=\lambda_1r(1+O(r))\pi_{1}\partial_Ez^-_{E,r}+\langle o(r),\pi_{E,\hat u}\partial_Ez^-_{E,r}+\pi_{E,\hat u}\partial_Ez^-_{E,r}\rangle$.
\end{proof}

Notice that $\Pi^+$ has two connected components, one consists of $\{z^+_E(t):E\in (0,E_0]\}$, the other one consists of $\{z^-_E(t):E\in (0,E_0]\}$. Recall that $z^-_E(t)$ denotes the periodic orbit shadowing the homoclinical orbit $z^-_E(t)$. Let $z^-_{E,-r}$, $z_{-E,-r}$ be the point where the periodic orbit $z^-_E(t)$ and $z_{-E}(t)$ intersects the section $\Sigma^-_{-r}$ respectively, then Lemma \ref{lem6.1} also holds for the tangent vector $\partial _Ez^-_{E,-r}$ of $\Pi|_{\Sigma^-_{-r}}$ at the point $z^-_{E,-r}$.

To consider the case $E<0$, we recall that $z^-_{E,\pm r}$ and $z^+_{E,\pm r}$ denote the point where the periodic orbit $z_E(t)$ intersects the section $\{u_1=\pm r\}$ and $\{v_1=\pm r\}$ respectively, see Figure \ref{fig3}. Applying Proposition \ref{pro6.2} to $d\Phi_{r,-r}(z^+_{E,r})$ and $d\Phi_{E,r,-r}(z^+_{E,r})$, we see that $d\Phi_{r,-r}(z^+_{E,r})$ has an eigenvalue $\sigma_1=1+O(r)$ with the eigenvector $\eta_1=(\eta_{1,1},\eta_{1,\hat u},\eta_{1,\hat v})$, $d\Phi_{E,r,-r}(z^+_{E,r})$ has
$(n-1)$ pairs of eigenvalues $\{\sigma_i,\sigma_i^{-1},i=2,\cdots,n\}$ associated with the eigenvector $\eta_i=(\eta_{i,1},\eta_{i,\hat u},\eta_{i,\hat v})$ and $\eta_{i+n}=(\eta_{i+n,1},\eta_{i+n,\hat u}, \eta_{i+n,\hat v})$ respectively. We also apply Proposition \ref{pro6.2} to $d\Phi_{-r,r}(z^+_{E,r})$ and $d\Phi_{E,-r,r}(z^+_{E,r})$, let $\sigma'_1,\sigma'_i$ and $\sigma'^{-1}_i$ be the eigenvalues associated with an eigenvectors $\eta'_1=(\eta'_{1,\hat u},\eta'_{1,1},\eta'_{1,\hat v})\in T_{z^-_{E,r}}U_{r}$ and other $2n-2$ eigenvectors $\eta'_i=(\eta'_{i,\hat u},\eta'_{i,n+1},\eta'_{i,\hat v}),\eta'_{i+n}=(\eta'_{i+n,\hat u},\eta'_{i+n,n+1},\eta'_{i+n,\hat v})\in T_{z^-_{E,r}}U_{E,r}$ respectively.

Similarly, we set $E^+_{\hat u}=\mathrm{span}\{\eta_2,\cdots,\eta_n\}$, $E^+_{\hat v}=\mathrm{span}\{\eta_{n+2},\cdots,\eta_{2n}\}$, $E^+_1=\frac{\partial}{\partial u_1}\mathbb{R}$, $E^-_1=\frac{\partial}{\partial v_1}\mathbb{R}$, $E^+_{\hat u}=\mathrm{span}\{\eta'_2,\cdots,\eta'_n\}$ and  $E^-_{\hat v}=\mathrm{span}\{\eta'_{n+2},\cdots,\eta'_{2n}\}$. So we have the decomposition $T_{z^+_{E,r}}S_{r}=E^+_1\oplus E^+_{\hat u}\oplus E^+_{\hat v}$ and $T_{z^-_{E,r}}S_{r}=E^-_1\oplus E^-_{\hat u}\oplus E^-_{\hat v}$. The tangent spaces $T_{z^-_{E,-r}}U_{-r}$ and $T_{z^+_{E,-r}}S_{-r}$ also admit similar decomposition. Similar to the case $E>0$, we define $\pi_{E,1}$, $\pi_{E,\hat u}$ and $\pi_{E,\hat v}$ to be the projection from $T_{z^+_{E,\pm r}}S_{\pm r}$, $T_{z^-_{E,\pm r}}U_{\pm r}$ to the corresponding subspaces respectively.

Let $z^-_{E',\pm r}$ and $z^+_{E',\pm r}$ denote the point where the periodic orbit $z_{E'}(t)$ intersects the section $\Sigma^-_{\pm r}$ and $\Sigma^+_{\pm r}$ respectively. Let $\Delta z^\pm_{\pm r}=z^\pm_{E',\pm r}-z^\pm_{E,\pm r}$. Because $\Phi_{E,r,-r}(z^+_{E,r})=z^-_{E,-r}$ and $\Phi_{E,-r,r}(z^+_{E,-r}) =z^-_{E,r}$, some constant $c>0$ exists, independent of $E$, such that
\begin{equation}\label{eq6.7}
|\pi_{E,\hat u}\Delta\hat z^+_{r}|\le c|E|^{\lambda_2/\lambda_1}|\pi_{E,\hat u} \Delta\hat z^-_{-r}|,\qquad |\pi_{E,\hat v}\Delta\hat z^-_{r}|\le c|E|^{\lambda_2/\lambda_1}|\pi_{E,\hat v}\Delta\hat z^+_{-r}|.
\end{equation}
Being aware that the Hamiltonian is reduced from the one with $\mathbf{s}$-symmetry, we see that the coordinate change is close to identity so that \eqref{eq2.4} in Proposition \ref{flatpro} holds, especially it is down for the Birkhoff normal form. Therefore, there exists some $\nu'\ge 1$ such that
\begin{equation}\label{eq6.8}
\begin{aligned}
&\nu'^{-1}|\Delta z^+_{-r}|\le |\Delta z^-_r|\le \nu'|\Delta z^+_{-r}|,\\
&\nu'^{-1}|\Delta z^-_{-r}|\le |\Delta z^+_{r}|\le \nu'|\Delta z^-_{-r}|.
\end{aligned}
\end{equation}
We see from \eqref{eq6.7} and \eqref{eq6.8} that $|\pi_{E,\hat v}\partial_E z^-_{E,r}\le c|E|^{\lambda_2/\lambda_1}||\partial_E z^-_{E,r}|$. Similar to Equation \eqref{eq6.?}, we also have
\begin{equation}\label{eq6.9}
\begin{aligned}
\pi_{E,\hat u}\Delta z^+_r&=A'_{E,11}\pi_{E,\hat u}\Delta z^-_r+A'_{E,12}\pi_{E,\hat v}\Delta z^-_r+A'_{E,13}\pi_{E,1}\Delta z^-_r+O(|\Delta z^-_r|^2),\\
\pi_{E,\hat v}\Delta z^+_r&=A'_{E,21}\pi_{E,\hat u}\Delta z^-_r+A'_{E,22}\pi_{E,\hat v}\Delta z^-_r +A'_{E,23}\pi_{E,1}\Delta z^-_r+O(|\Delta z^-_r|^2)
\end{aligned}
\end{equation}
where $\mathrm{det}A'_{E,11}\ne 0$. In view of \eqref{eq6.7} and \eqref{eq6.8}, we obtain from the first equation of \ref{eq6.9} that
$$
|A'_{E,11}\pi_{E,\hat u}\partial_E z^-_{E,r}+A'_{E,13}\pi_{E,1}\partial_E z^-_{E,r}|\le c|E|^{\lambda_2/\lambda_1}||\partial_E z^-_{E,r}|.
$$
By the experience to prove Lemma \ref{lem6.1}, these arguments lead to the following:
\begin{lem}\label{lem6.2}
For small $E<0$, let $z^-_{E,r}$ denote the point where the periodic orbit $z_E(t)$ intersects the section $\Sigma^-_r$. Then, the tangent vector $\partial _Ez^-_{E,r}$ of $\Pi|_{\Sigma^-_r}$ at the point $z^-_{E,r}$ satisfies the condition that $\pi_{1}\partial_E z^-_{E,r}=\frac1{\mu_1r}(1+O(r))$, $\pi_{E,\hat v}\partial_E z^-_{E,r}=o(|E|)$ and
$$
\pi_{E,\hat u}\partial_E z^-_{E,r}= -A^{-1}_{E,11}A_{E,13}\pi_{E,1}\partial_E z^-_{E,r}+o(|E|).
$$
\end{lem}

Because of Lemma \ref{lem6.1} and \ref{lem6.2}, the $C^1$-smoothness of $\Pi^\pm$ extends to their boundary if the decomposition $T_{z^-_{E,r}}U_{r}=E^-_1\oplus E_{\hat u}\oplus E_{\hat v}$ for $E>0$ and $T_{z^-_{E,r}}U_{r}=E^-_1\oplus E^-_{\hat u}\oplus E^-_{\hat v}$ for $E<0$ is convergent as $E\to 0$, since the quantities $A^{-1}_{E,11}A_{E,13}$ and $A'^{-1}_{E,11}A'_{E,13}$ continuously depend on the point. What remains to show is that they are $C^1$-joined together.

\subsection{Differentiability along the homoclinics}
As the final step, we verify that $\Pi^+$ is $C^1$-joined to $\Pi^-$ along the homoclinic orbit. In the original coordinate $(u_1,\hat u,v_1,\hat v)$ let $\pi_{\hat u}$, $\pi_{\hat v}$ be the projection to the subspace $\mathrm{Span}\{\frac{\partial}{\partial u_i}:i=2,\cdots n\}$ and to $\mathrm{Span}\{\frac{\partial}{\partial v_i}:i=2,\cdots n\}$ respectively. Let $\eta_{\hat u}=\pi_{\hat u}\eta$ and $\eta_{\hat v}=\pi_{\hat v}\eta$. According to Lemma \ref{lem6.1} and \ref{lem6.2}, we shall see it enough to prove the following theorem
\begin{pro}\label{pro6.6}
For $E>0$, let $\{\eta_i,\eta_{i+n}:i=2,\cdots,n\}$ denote the eigenvectors of $d\Phi_{E,r,r}(z^+_{E,r})$, or of $d\Phi_{-E,\pm r,\mp r}$. Let $\eta_i=(\eta_{i,\hat u},\eta_{i,\hat v})$, where $\eta_{i,\hat u},\eta_{i,\hat v} \in\mathbb{R}^{n-1}$ denote its $\hat u$- and $\hat v$-components respectively, then
$$
|\eta_{i,\hat v}|\le O(|E|)|\eta_{i,\hat u}|,\quad
|\eta_{i+n,\hat u}|\le O(|E|)|\eta_{i+n,\hat v}|, \quad \mathrm{as}\ |E|\to 0.
$$
\end{pro}

Postponing the proof of the theorem to the next section, we apply it to check the differentiability along the homoclinics. Notice that the differential $d\Phi_{r}$ of the outer map at the point $z^-_{E,r}$ be represented by the matrix $A$
\begin{equation*}
\left[\begin{matrix}\eta^+_{\hat u}\\ \eta^+_{\hat v}\\ \eta^+_{1}\end{matrix}\right]=\left[\begin{matrix} A_{11} & A_{12} & A_{13}\\
A_{21} & A_{22} & A_{23}\\
A_{31} & A_{32} & A_{33}
\end{matrix}\right] \left[\begin{matrix}\eta^-_{\hat u}\\ \eta^-_{\hat v}\\ \eta^-_{1}\end{matrix}\right]=A\left[\begin{matrix}\eta^-_{\hat u}\\ \eta^-_{\hat v}\\ \eta^-_{1}\end{matrix}\right],
\end{equation*}
where $\eta_1^-$, $\eta_1^+$ are the projection of $\eta^-$, $\eta^+$ to $\frac{\partial}{\partial v_1}$ and to $\frac{\partial}{\partial u_1}$ respectively, i.e. $\eta_1^\pm=\pi_1\eta^\pm$, all elements in $A$ continuously depend on $E$ around zero energy. Then, we claim
\begin{equation}\label{eq6.10}
\begin{aligned}
&\pi_{\hat u}\frac{\partial z^-_{E,r}}{\partial E}\Big|_{E\uparrow0}=\pi_{\hat u}\frac{\partial z^-_{E,r}}{\partial E}\Big|_{E\downarrow0}=-A^{-1}_{11}A_{13}\pi_{1}\frac{\partial z^-_{E,r}}{\partial E}\Big|_{E=0}; \\
&\pi_{\hat v}\frac{\partial z^-_{E,r}}{\partial E}\Big|_{E\uparrow0}=\pi_{\hat v}\frac{\partial z^-_{E,r}}{\partial E}\Big|_{E\downarrow0}=0;\\
&\Big\langle\partial H(z^-_{E,r}),\frac{\partial z^-_{E,r}}{\partial E}\Big\rangle\Big|_{E\uparrow0}=1.
\end{aligned}
\end{equation}
It implies the $C^1$-differentiability along the homoclinics since $\Pi^+$ is joined to $\Pi^-$ along the homoclinics. To check it, we obtain from the statement of Lemma \ref{lem6.1} that
\begin{equation}\label{eq6.11}
\pi_{E,\hat u}\frac{\partial z^-_{E,r}}{\partial E}=-A^{-1}_{E,11}\pi_{E,\hat u}A_{E,13}\pi_{E,1}\frac{\partial z^-_{E,r}}{\partial E}+o(|E|), \qquad \pi_{E,\hat v}\frac{\partial z^-_{E,r}}{\partial E}=o(|E|).
\end{equation}
which is obtained in view of \eqref{eq6.2new}, where we are in the coordinates
\begin{equation}\label{eq6.14new}
\left[\begin{matrix}\eta^\pm_{\hat u}\\ \eta^\pm_{\hat v} \\ \eta^\pm_1\end{matrix}\right]=
\left[\begin{matrix}\Psi^\pm_{E,11} & \Psi^\pm_{E,12} & 0\\
\Psi^\pm_{E,21} & \Psi^\pm_{E,22} & 0\\
\Psi^\pm_{E,31} & \Psi^\pm_{E,32} & 1\\
\end{matrix}\right]
\left[\begin{matrix}\eta^\pm_{E,\hat u}\\ \eta^\pm_{E,\hat v} \\ \eta^\pm_{E,1}\end{matrix}\right]=M^\pm_E\left[\begin{matrix}\eta^\pm_{E,\hat u}\\ \eta^\pm_{E,\hat v} \\ \eta^\pm_{E,1}\end{matrix}\right],
\end{equation}
in the case of negative $E$, $M^+_E$ may not be the same as $M^-_E$. It follows from Proposition \ref{pro6.6} that $\Psi^\pm_{E,12}\to 0$ and $\Psi^\pm_{E,21}\to 0$ as $E\to 0$. Therefore, the inverse of $M_E^+$ with small $|E|$ takes a special form
$$
(M_E^+)^{-1}=\left[\begin{matrix}
(\Psi^+_{E,11})^{-1} & 0 & 0 \\
0 & (\Psi^+_{E,22})^{-1} & 0 \\
-\Psi^+_{E,31}(\Psi^+_{E,11})^{-1} & -\Psi^+_{E,32}(\Psi^+_{E,22})^{-1} & 1
\end{matrix}\right]+O(|E|).
$$
Because $A_E=(M_E^+)^{-1}AM_E^-$,  $|\pi_{E,\hat u}\frac{\partial z^+_{E,r}}{\partial E}|\to 0$ and $|\pi_{E,\hat v}\frac{\partial z^-_{E,r}}{\partial E}|\to 0$ as $E\to 0$, the first equation of \eqref{eq6.2new} turns out to be
$$
(\Psi^+_{E,11})^{-1}\Big(A_{11}\Psi^-_{E,11}\pi_{E,\hat u}\frac{\partial z^-_{E,r}}{\partial E} +A_{13} (\Psi_{E,31}^-\pi_{E,\hat u}\frac{\partial z^-_{E,r}}{\partial E}+\pi_{E,1}\frac{\partial z^-_{E,r}}{\partial E}) \Big)=O(E).
$$
Since $\Psi^+_{E,11}$ is non-singular, we obtain from \eqref{eq6.14new} that
\begin{equation}\label{finaldaoshu}
A_{11}\pi_{\hat u}\frac{\partial z^-_{E,r}}{\partial E}+A_{13}\pi_{E,1}\frac{\partial z^-_{E,r}}{\partial E}=O(E), \quad \mathrm{as}\ E\to 0.
\end{equation}
Consequently, we get from \eqref{eq6.14new} that $\pi_{\hat u}\frac{\partial z^-_{E,r}}{\partial E}=O(E)$. Notice the matrix $A$ represents the differential $d\Phi_{r}$ of the outer map $\Phi_r$ at the point $z^-_{E,r}$, which continuously depends on $E$. It completes the proof of \eqref{eq6.10}.

Next, we study the case of compound type homology class. In this case, $\Pi$ is no longer a sub-manifold, but its pull-back $\tilde\Pi=\pi_h^{-1}\Pi$ is a NHIC.

\begin{theo}\label{pro6.4}
The manifold $\tilde\Pi=\pi_h^{-1}\Pi$ defined by \eqref{tilde-Pi} is a $C^1$-invariant cylinder with $(\ell+1)k$ holes.
\end{theo}
\begin{proof}
By the assumptions, the periodic orbit $z^+_{E}(t)$ successively passes through the section $\Sigma_{E,r}^-$ at the point $\{z^-_{E,r,j}:j=1,\cdots,k\}$ in the way $z^+_{E,r,j}=\Phi_{E,r}z^-_{E,r,j}\in \Sigma_{E,r}^+$ and $\Phi_{E,r,r}z^+_{E,r,j}=z^-_{E,r,j+1}\mod k$.

According to Proposition \ref{pro6.2}, the differential map $d\Phi_{E,r,r}(z^+_{E,r,j})$ has
$(n-1)$ pairs of eigenvalues $\{\sigma_{j,i}=\mu_{j,i}^{-1}E^{\lambda_i/\lambda_1}, \sigma_{j,i}^{-1},i=2,\cdots,n\}$ associated with the eigenvector $\eta_{j,i}=(\eta_{j,i,1},\eta_{j,i,\hat u},\eta_{j,i,\hat v})$ and $\eta_{j,i+n}=(\eta_{j,i+n,1},\eta_{j,i+n,\hat u}, \eta_{j,i+n,\hat v})$ respectively.

We set $E^+_{j,\hat u}=\mathrm{span}\{\eta_{j,2},\cdots,\eta_{j,n}\}$, $E^+_{j,\hat v}=\mathrm{span}\{\eta_{j,n+2},\cdots,\eta_{j,2n}\}$, $E^-_1=\frac{\partial}{\partial u_1}\mathbb{R}$ and $E^+_1=\frac{\partial}{\partial v_1}\mathbb{R}$ which lead to the decomposition
$$
T_{z^+_{E,r,j}}S_{r}=E^+_1\oplus E_{j,\hat u}\oplus E_{j,\hat v},\quad T_{z^-_{E,r,j}}S_{r}=E^-_1\oplus E_{j,\hat u}\oplus E_{j,\hat v}.
$$
Let $\pi_{E,j,1}$, $\pi_{E,j,\hat u}$ and $\pi_{E,j,\hat v}$ denote the projection from $T_{z^+_{E,r,j}}S_{r}$ to $E^+_1$, $E_{j,\hat u}$ and $E_{j,\hat v}$ and from $T_{z^-_{E,r,j}}U_{r}$ to $E^-_1$, $E_{j,\hat u}$ and $E_{j,\hat v}$ respectively. Let $\eta^\pm_{E,j}=(\eta^\pm_{E,j,1},\eta^\pm_{E,j,\hat u},\eta^\pm_{E,j,\hat v})$ be tangent vector with $\eta^\pm_{E,j,\hat u}=\pi_{E,j,\hat u}\eta^\pm_{E,j}$, $\eta^\pm_{E,j,\hat v}=\pi_{E,j,\hat v}\eta^\pm_{E,j}$ and $\eta^\pm_{E,j,1}=\pi_{E,j,1}\eta^\pm_{E,j}$.

Similar to Equation \eqref{eq6.2new}, the differential $d\Phi_{r}(z^-_{E,r,j})$ of the outer map $\Phi_r$ at $z^-_{E,r,j}$ is represented by a matrix
\begin{equation*}
\left[\begin{matrix}\eta^+_{E,j,\hat u}\\ \eta^+_{E,j,\hat v}\\ \eta^+_{E,j,1}\end{matrix}\right]=\left[\begin{matrix} A_{E,j,11} & A_{E,j,12} & A_{E,j,13}\\
A_{E,j,21} & A_{E,j,22} & A_{E,j,23}\\
A_{E,j,31} & A_{E,j,32} & A_{E,j,33}
\end{matrix}\right] \left[\begin{matrix}\eta^-_{E,j,\hat u}\\ \eta^-_{E,j,\hat v}\\ \eta^-_{E,j,1}\end{matrix}\right]=A_{E,j}\left[\begin{matrix}\eta^-_{E,j,\hat u}\\ \eta^-_{E,j,\hat v}\\ \eta^-_{E,j,1}\end{matrix}\right]
\end{equation*}
where $\mathrm{det}A_{E,j,11}\ne 0$. In view of Lemma \ref{lem6.1}, we claim that for $E>0$ the tangent vector of $\Pi|_{\Sigma^-_r}$ at $z^-_{E,r,j}$ satisfies the condition that
\begin{equation}\label{daoshufork}
\begin{aligned}
\pi_{E,j,\hat u}\partial_E z^-_{E,r,j}&= -A^{-1}_{E,j,11}A_{E,j,13}\pi_{E,j,1}\partial_E z^-_{E,r,j}+o(|E|),\\
\pi_{E,j,\hat v}\partial_E z^-_{E,r,j}&=o(|E|),\\
\pi_{E,j,1}\partial_E z^-_{E,r,j}&=(\mu_1r)^{-1}(1+O(r)).
\end{aligned}
\end{equation}
Indeed, let $\Delta z^\pm_j=z^\pm_{E',r,j}-z^\pm_{E,r,j}$,
$\nu^{-1}|\Delta z^+_j|\le|\Delta z^-_j|\le\nu|\Delta z^+_j|$ holds for some $\nu>1$ and Equation \eqref{eq6.?} holds if we replace $\Delta z^\pm$, $A_{E,\ell k}$ by $\Delta z^\pm_j$ and $A_{E,j,\ell k}$ respectively. The quantities $\Delta z^+_j$ and $\Delta z^-_{j+1}$ are related by the inner map also, $\Delta z^-_{j+1}=d\Phi_{r,r}\Delta z^+_j+O(|\Delta z^+_j|^2)$. Therefore, for sufficiently small $|\Delta z^+|$ and in view of \eqref{bili}, we have
\begin{equation*}\label{}
|\pi_{E,\hat u}\Delta z^+_j|\le\frac1{2\mu_2}|E|^{\frac{\lambda_2}{\lambda_1}} |\pi_{E,\hat u}\Delta z^-_{j+1}|,\qquad |\pi_{E,\hat v}\Delta z^-_{j+1}|\le \frac 1{2\mu_2}|E|^{\frac{\lambda_2}{\lambda_1}}|\pi_{E,\hat v}\Delta z^+_j|.
\end{equation*}
From the second one, we find $|\pi_{E,j,\hat v}\Delta z^-_{j+1}|\le \frac {\nu}{2\mu_2}|E|^{\lambda_2/\lambda_1}|\pi_{E,j,\hat v}\Delta z^-_j|$. Let $j$ range over $\{1,\cdots,k\}$, we find the $|\pi_{E,j,\hat v}\partial_E z^-_{E,r,j}|\le c|E|^{\lambda_2/\lambda_1}$, from which and the first equation of \eqref{eq6.?} it follows that
$$
|A_{E,j,11}\pi_{E,\hat u}\partial_E z^-_{E,r,j}+A_{E,j,13}\pi_{E,j,1}\partial_E z^-_{E,r,j}|\le c_2|E|^{\lambda_2/\lambda_1}(|\partial_E z^-_{E,r,j+1}|+|\partial_E z^-_{E,r,j-1}|)
$$
holds for each $j\mod k$. Since all eigenvalue of $A_{E,j,11}$ for all $j\le k$ are  uniformly away from zero in $E$, the first equation of \eqref{daoshufork} holds for all $j\le k$. The third follows from the identity $\langle\partial H,\partial_E z^-_{E,r,j}\rangle =1$.

In the same way to show \eqref{finaldaoshu}, we are able to get from \eqref{daoshufork} that for $j=1,\cdots,k$
\begin{equation}\label{daoshuforE>0}
A_{j,11}\pi_{\hat u}\frac{\partial z^-_{E,r,j}}{\partial E}+A_{j,13}\pi_{1}\frac{\partial z^-_{E,r,j}}{\partial E}=O(E), \quad \mathrm{as}\ E\downarrow0.
\end{equation}
where $A_{j,i\ell}$ denotes the submatrix of $A_j$ which represents $d\Phi_r(z_{E,r,j})$ in the coordinate $(u,v)$.  Recall the constitution of $\Pi$. Its negative energy part is made up by shrinkable periodic orbits extending from each pairs of homoclinic orbits $\{z^\pm_j(t):j=1,\cdots,k\}$. It follows from the proof for single homology class case that \eqref{daoshuforE>0} holds as $E\downarrow 0$.

Since $\Pi=\pi_h\tilde\Pi$, it consists of $k$ pieces of surface when it is restricted around the origin, they are made up by the orbits shown in Figure \ref{fig5}.
Since $|\hat z^\pm_i(t)|=o(|z^\pm_{i,1}(t)|)$ when $z^\pm_i(t)$ is close to the origin,   each piece can be treated as the graph $\mathcal{G}_i$ of a map $(u_1,v_1)\to\hat z_i(u_1,v_1)$. Applying the same argument for the case $k=1$, we see that $d\hat z_i(0)=0$, i.e. all leaves are tangent to each other at the origin. So, Proposition \ref{pro6.4} is proved.
\end{proof}

To check the normally hyperbolic property of $\Pi$, we only need to consider the points on the homoclinic orbits $\{z^\pm(t):t\in\mathbb{R}\}$. Since the cylinder is made up by hyperbolic periodic orbits, along which the Lyapunov exponents with respect to the tangent space are equal to zero while they are non-zero when they are restricted on the normal space.

Obviously, $T_{z=0}\Pi=\mathrm{Span}\{\frac{\partial}{\partial x_1},\frac{\partial}{\partial y_1}\}$. Given suitably small $r>0$, there exists $T_r>0$ such that it holds for any $z\in\{z^\pm(t):t\in\mathbb{R}\}$ that $\Phi_H^t(z)\in B_r$ if $t\ge T_r$. Therefore, for each $z\in\{z^\pm(t):t\in\mathbb{R}\}$ there exists a decomposition $T_z\mathbb{R}^{2n}=T_z\Pi\oplus T_zN^+\oplus T_zN^-$ and $c_1\ge 1$ such that
\begin{equation*}
\begin{aligned}
c_1^{-1}e^{-(\lambda_1-cr)(t-T_r)}<\frac {\|d\Phi^t_H(z)v\|}{\|v\|}&<c_1e^{(\lambda_1+cr)(t-T_r)},\qquad &&\forall\ v\in T_z\Pi,\\
\frac{\|d\Phi_{H}^{t}(z)v\|}{\|v\|}&\ge c_1e^{(\lambda_2-cr)(t-T_r)}, &&\forall\ v\in T_zN^+,\\
\frac{\|d\Phi_{H}^{t}(z)v\|}{\|v\|}&\le c_1^{-1}e^{-(\lambda_2-cr)(t-T_r)}, &&\forall\ v\in T_zN^-.
\end{aligned}
\end{equation*}
Reader can refer to the proof of formula (\ref{eq2.12}) for details. With Theorem \ref{theo6.1} and \ref{pro6.4}, the whole proof of Theorem \ref{mainresult} is completed.

\noindent{\bf Remark}. The homoclinic orbits $\{z^\pm_1(t),\cdots,z^\pm_k(t)\}$ are not required to be all different, it is possible that some $z^\pm_i(t)$ is multiply counted, e.g. some $i_1,\cdots i_m\ne i$ such that $z^\pm_{i_1}(t)=\cdots=z^\pm_{i_m}(t)=z^\pm_i(t)$. Although  $\Pi$ is multiply folded along $z^\pm_i(t)$ in this case, it follows from \eqref{daoshufork} that $T_{\tilde z}\tilde\Pi=T_{\tilde z'}\tilde\Pi$ if $\pi_h\tilde z=\pi_h\tilde z'\in\cup_{t\in\mathbb{R}}z^\pm_i(t)$.

\section{The fundamental matrix and its eigenvectors}\label{sec.7}
This section is denoted to prove Proposition \ref{pro6.2} and \ref{pro6.6}. To this end, we study the variational equation along an orbit of $\Phi_H^t$ that starts from the section $\Sigma^+_r=\{v_1=r\}$, remains in $B_r(0)$ until it arrives the section $\Sigma^-_r=\{u_1=r\}$ after a time $t_E$.
\begin{equation}\label{equation7.1}
\dot\xi_z=(\mathrm{diag}\{\Lambda,-\Lambda\}+J\partial^2P(z(t)))\xi_z.
\end{equation}
We assume $H$ is in the Birkhoff normal form $H=\sum_{i=1}^n\lambda_iu_iv_i+N+R$ with
\begin{equation}\label{birkhoffnormalform}
N=N_\kappa(I_1,\cdots,I_n),\quad R=O(\|(u,v)\|^{2\kappa+2})
\end{equation}
where $I_i=u_iv_i$, $N_\kappa$ is a polynomial of degree $\kappa$ without constant and linear part, the integer $\kappa$ is chosen so that $(\kappa-1)\lambda_1>\lambda_n$.

Instead of studying the variational equation along the orbit which is from $\{v_1=r\}$ to $\{u_1=r\}$ in one step for small $E>0$, we study the equation in two steps, from the section $\{v_1=r\}$ to $\{v_1=u_1\}$ first, then to $\{u_1=r\}$. For small $E<0$, we study the equation also in two steps, from $\{v_1=r\}$ to $\{v_1=-u_1\}$ first, then to $\{u_1=-r\}$. Let $\Lambda=\mathrm{diag}\{\lambda_1,\lambda_2,\cdots,\lambda_n\}$ and $I$ denote the identity matrix. 

\begin{lem}\label{lemmaforvariationalequation}
For $E>0$, let $z^+_E(t)=(u^+_E(t),v^+_E(t))$ be the periodic orbit that starts from the section $\{v_1=r\}$ at $t=0$, remains in $\{|z|\le r\}$ before it arrives the section $\{u_1=v_1\}$ after a time $\tau_E\approx\frac{1}{2\lambda_1}\ln\frac 1E$ bounded by $($\ref{eq2.9}$)$. Then, the variational equation along the orbit $z^+_{E}(t)|_{[0,t']}$ or along $z_{-E}(t)|_{[0,t']}$ with $t'\le \tau_{E}$ takes the form
\begin{equation}\label{variational-equation-for-Birkhoff-normal-form}
(\dot\xi_u,\dot\xi_v)=[\mathrm{diag}\{\Lambda,-\Lambda\}+r^2(B'(t)+B''(t))](\xi_u,\xi_v)
\end{equation}
where $B'$ and $B''$ satisfy the conditions

1, let $e^{-\Lambda_Et}=\mathrm{diag} \{e^{-\lambda_{1,E}t},\cdots,e^{-\lambda_{n,E}t}\}$ where $\lambda_{i,E}=\lambda_i-\nu r\sqrt{|E|}$ with $\nu>0$, the matrix $B'$ takes the form
$$
B'=\left[\begin{matrix} \sqrt{|E|}I & 0\\0 & e^{-\Lambda_Et}\end{matrix}\right]B'_0(t)\left[\begin{matrix} e^{-\Lambda_Et} & 0\\0 & \sqrt{|E|}I\end{matrix}\right]
$$
all elements of $B'_0(t)$ are bounded by some $\nu>0$;

2, all elements of $B''$ are bounded by $\nu r^{2\kappa-2}e^{-2\kappa\lambda_{1,E}t}$.

The properties also hold for the variational equation along the orbit $z_{-E}(t)$ that starts from the section $\{v_1=r\}$ at $t=0$, remains in $\{|z|\le r\}$ before it arrives the section $\{u_1=-v_1\}$ after a time $\tau_E$.
\end{lem}
\begin{proof}
In the Birkhoff normal form, the variational equation takes the form
\begin{equation}\label{variationalequation}
\left[\begin{matrix}\dot\xi_u\\
\dot\xi_v
\end{matrix}\right]=(\mathrm{diag}\{\Lambda,-\Lambda\}+A(t))\left[\begin{matrix}\xi_u\\ \xi_v
\end{matrix}\right]
\end{equation}
where the matrix $A$ admits a decomposition $A=A'+A''$, $A'$ is from the main part $N$ and $A''$ is from the remaining part $R$. Denoting by $a'_{ij}$ the element at the crossroad of the $i$-th row and $j$-th column of $A'$ we find that for $1\le i,j\le n$
$$
\begin{aligned}
&a'_{ij}=(b_{ij}+\delta_{ij}b_i)u_iv_j, \quad &&a'_{i(j+n)}=b_{ij}u_iu_j,\\
&a'_{(i+n)j}=b_{ij}v_iv_j, \quad &&a'_{(i+n)(j+n)}=(b_{ij}+\delta_{ij}b_i)v_iu_j,
\end{aligned}
$$
where $b_{ij}=\frac{\partial^2 N}{\partial I_i\partial I_j}$, $\delta_{ij}$ is the  Kronecker Delta and $b_{i}=I_i^{-1}\frac{\partial N}{\partial I_i}$. Therefore,  the matrix $A'$ admits the form
\begin{equation}\label{matrix-A'}
A'=\left[\begin{matrix} u & 0\\0 & v\end{matrix}\right] \left[\begin{matrix}\bar A_1+\bar A_0 & \bar A_1 \\
-\bar A_1 & -\bar A_1-\bar A_0\end{matrix}\right]\left[\begin{matrix} v & 0\\0 & u\end{matrix}\right]
\end{equation}
where $u=\mathrm{diag}\{u_1,\cdots,u_n\}$, $v=\mathrm{diag}\{v_1,\cdots,v_n\}$, $\bar A_1=\{b_{ij}\}_{1\le i,j\le n}$ is a matrix of order $n$ and $\bar A_0=\mathrm{diag}\{I_1^{-1}\frac{\partial N}{\partial I_1},\cdots,I_n^{-1}\frac{\partial N}{\partial I_n}\}$. Since $N$ does not have linear term, $b_i$ is bounded, so all elements of $\bar A_1$ and $\bar A_0$ are bounded.

Restricted to a segment of the periodic orbit $\Gamma_E=\{u_E(t),v_E(t):t\in[0, \tau_E]\}$ that starts from the section $\{v_1=r\}$ when $t=0$ and arrives the section $\{v_1=\pm u_1\}$ after a time $\tau_E$. For small $|E|>0$, it follows from \eqref{eq2.8}, \eqref{eq2.9} and \eqref{position} that $|\tau_E-\frac{1}{2\lambda_1}\ln\frac 1E|$ is uniformly bounded in $|E|$
\begin{equation}\label{Birkhoffcontrol1}
|v_E(t)|\le |v_E(0)|e^{-(\lambda_1-cr)t}, \quad |u_E(t)|\le |u_E(\tau_E)|\le cr\sqrt{|E|}, \quad \forall\ t\in[0,\tau_E],
\end{equation}
which leads to the estimate on the main part and the remainder of the Birkhoff normal form when they are restricted on $\Gamma_E$. According to (\ref{Birkhoffcontrol1}), some constant $\nu>0$ exists such that
\begin{equation}\label{Birkhoffcontrol2}
|\partial_{I_i}N|\le\nu r\sqrt{|E|},\quad
|\partial_{u_i}R(u_E(t),v_E(t))|\le\nu |v_E(0)|^{2\kappa}e^{-2\kappa(\lambda_1-cr)t}
\end{equation}
where the second inequality is got by applying the properties that $|\partial R(u(t),v(t))|=O(\|u(t),v(t)\|^{2\kappa+1})$ and $|v_E(t)|\ge |v_E(0)|e^{-(\lambda_1-cr)t}=|v_E(0)||E|^{\frac{\lambda_1-cr}{2\lambda_1}}\ge |v_E(0)|\sqrt{|E|}\ge c|u_E(t)|$ holds for $t\in[0,\tau_E]$. Applying the estimates \eqref{Birkhoffcontrol1} and \eqref{Birkhoffcontrol2} to the Hamiltonian equation associated with the Birkhoff normal form, we find
$$
\dot v_{E,i}\le -(\lambda_i-\nu r\sqrt{|E|})v_{E,i}+\nu r^{2\kappa}e^{-(2\kappa+1)(\lambda_1-cr)t}, \qquad \forall \, t\in[0,t'_E].
$$
It follows from a variant of Gr\"onwell inequality that $v_{E,i}(t)$ is bounded by
$$
v_{E,i}(t)\le e^{-\lambda_{i,E}t}\Big(v_{E,i}(0)+\nu r^{2\kappa}\int_{0}^{t} e^{(\lambda_{i,E}-(2\kappa+1)(\lambda_1-cr))s}ds\Big), \qquad \forall\ t\in[0,\tau_E].
$$
Since $|\tau_E-\frac{1}{2\lambda_1}\ln\frac 1E|$ is uniformly bounded in $E$, $\kappa\lambda_1>\lambda_n$ and $r>0$ is small, one has $e^{-(2\kappa+1)(\lambda_1-cr)t}< e^{-\lambda_{i,E}t}$ for $t\in[0,t'_E]$. Therefore, one has
\begin{equation}\label{Birkhoffcontrol3}
|v_{E,i}(t)|\le (|v_{E,i}(0)|+\nu r^{2\kappa})e^{-\lambda_{i,E}t}, \qquad \forall \ t\in[0,\tau_E],\ i=1,\cdots,n.
\end{equation}
Notice that $|u_i|$ and $|v_i|$ in \eqref{matrix-A'} are bounded by the second estimate in \eqref{Birkhoffcontrol1} and \eqref{Birkhoffcontrol3} respectively, and the elements of $A''$ are from the second derivative of the remainder $R$ of order $O(\|(u,v)\|^{2\kappa})$, $|v(t)|\ge|u(t)|$ holds for $t\in[0,\tau_E]$, i.e. the elements of $B''$ are bounded $\nu r^{2\kappa-2}e^{-2\kappa\lambda_{1,E}t}$.
\end{proof}

In the following, a matrix $M(t)$ is said to {\it be dominated} by another matrix $\bar M(t)$ if any element of $M$ is bounded by the corresponding element of $\bar M$, i.e.
$|m_{ij}(t)|\le \bar m_{ij}(t)$ holds for all $i,j\le n$. We denote the relation by $M(t)\prec\bar M(t)$, or $M\prec\bar M$ for short. Let $\mathbb{I}$ denote the matrix in which all elements are equal to $1$. By the notation, we have
\begin{equation}\label{dominate}
\begin{aligned}
B'(t)&\prec \nu\bar B'(t)=\nu\left[\begin{matrix} \sqrt{|E|}I & 0\\0 & e^{-\Lambda't}\end{matrix}\right] \mathbb{I}\left[\begin{matrix} e^{-\Lambda't} & 0\\0 & \sqrt{|E|}I\end{matrix}\right],\\
r^{-2\kappa+2}B''(t)&\prec\nu\bar B''(t)= \nu e^{-\lambda_{\kappa}t}\mathbb{I}
\end{aligned}
\end{equation}
where $e^{-\Lambda' t}=\mathrm{diag}\{e^{-\lambda'_1t}, \cdots,e^{-\lambda'_nt}\}$ and $\lambda_\kappa=2\kappa\lambda_{1,E}$. As $\lambda_{i,E}=\lambda_i-O(\sqrt{|E|})$, we are able to choose $\lambda'_i\le \lambda_{i,E}$ such that $3(\lambda_i-\lambda'_i)<\lambda_1$ holds for each $i\le n$.

To apply the relation (\ref{dominate}) to study the fundamental matrix $Z(t)$ of Equation (\ref{variational-equation-for-Birkhoff-normal-form}), by adopting the notation $e^{(\Lambda,-\Lambda)t}=\mathrm{diag}\{e^{\Lambda t},e^{-\Lambda t}\}$ and $e^{(-\Lambda,\Lambda)t}=\mathrm{diag}\{e^{-\Lambda t},e^{\Lambda t}\}$ we consider the matrices
\begin{equation}\label{MandNfort}
\begin{aligned}
N_t&=e^{(-\Lambda,\Lambda)t}\bar B'e^{(\Lambda,-\Lambda)t}\\
&=\left[\begin{matrix}\sqrt{|E|}e^{-\Lambda t}\mathbb{I}e^{(\Lambda-\Lambda')t} &
|E|e^{-\Lambda t}\mathbb{I}e^{-\Lambda t}\\
e^{(\Lambda-\Lambda')t}\mathbb{I}e^{(\Lambda-\Lambda')t} & \sqrt{|E|}e^{(\Lambda-\Lambda')t}\mathbb{I}e^{-\Lambda t}
\end{matrix}\right],\\
M_t&=e^{(-\Lambda,\Lambda)t}\bar B''e^{(\Lambda,-\Lambda)t}
=e^{-\lambda_\kappa t}\left[\begin{matrix}e^{-\Lambda t}\mathbb{I}e^{\Lambda t} &
e^{-\Lambda t}\mathbb{I}e^{-\Lambda t}\\
e^{\Lambda t}\mathbb{I}e^{\Lambda t} & e^{\Lambda t}\mathbb{I}e^{-\Lambda t}
\end{matrix}\right]
\end{aligned}
\end{equation}
Let $N_{t,ij}$, $M_{t,ij}$ denote the element of $N_t$, $M_t$ at the crossroad at the $i$-th row and the $j$-th column respectively, then for $1\le i,j\le n$ we have
\begin{equation}\label{element-N}
\begin{aligned}
&N_{t,ij}=\sqrt{|E|}e^{-(\lambda_i+\lambda'_j-\lambda_j)t},\quad
&&N_{t,i(j+n)}=|E|e^{-(\lambda_i+\lambda_j)t},\\
&N_{t,(i+n)j}=e^{(\lambda_i-\lambda'_i+\lambda_j-\lambda'_j)t},\quad
&&N_{t,(i+n)(j+n)}=\sqrt{|E|}e^{-(\lambda_j+\lambda'_i-\lambda_i)t}.\\
&M_{t,ij}=e^{-(\lambda_i-\lambda_j+\lambda_\kappa)t},\quad &&M_{t,i(j+n)}=e^{-(\lambda_i+\lambda_j+\lambda_\kappa)t},\\
&M_{t,(i+n)j}=e^{-(\lambda_\kappa-\lambda_i-\lambda_j)t},\quad &&M_{t,(i+n)(j+n)}=e^{-(\lambda_\kappa-\lambda_i+\lambda_j)t}.
\end{aligned}
\end{equation}
Since $\kappa\lambda_1>2\lambda_n$ is assumed, all elements of $M_t$ are smaller than 1 for $t>0$.

\begin{lem}\label{lem7.2}
Let $Z(t)$ with $Z(0)=I_{2n}$ be the fundamental matrix of the variational equation $($\ref{variational-equation-for-Birkhoff-normal-form}$)$ which satisfies the conditions listed in Lemma \ref{lemmaforvariationalequation}. Then, some constant matrices $D_0,D_1$ exist such that for $t\in [0,\tau_E]$
\begin{equation}\label{fundamentalM}
Z(t)-e^{(\Lambda,-\Lambda)t}\prec r^2e^{(\Lambda,-\Lambda)t}(N_tD_0+D_1).
\end{equation}
\end{lem}
\begin{proof} Treating $\epsilon=\nu r^2$ as a small parameter, we develop the fundamental matrix into a series of $\epsilon$
\begin{equation}\label{seriesforZ}
Z(t)=\sum_{\ell=0}^\infty \epsilon^\ell Z_\ell(t).
\end{equation}
Substituting $z$ in Equation \eqref{variational-equation-for-Birkhoff-normal-form}) with \eqref{seriesforZ}, one obtains a series of linear equations
$$
\begin{aligned}
\dot Z_0(t)&=\mathrm{diag}\{\Lambda,-\Lambda\} Z_0(t);\\
\dot Z_\ell(t)&=\mathrm{diag}\{\Lambda,-\Lambda\} Z_\ell(t)+(B'(t)+B''(t))Z_{\ell-1}(t),\qquad \ell=1,2,\cdots.
\end{aligned}
$$
Hence, we have $Z_0(t)=e^{(\Lambda,-\Lambda)t}$ and
\begin{equation}\label{inductive1}
Z_\ell(t)=\frac{1}{\nu}e^{(\Lambda,-\Lambda)t}\int_0^te^{(-\Lambda,\Lambda)s} (B'(s)+B''(s))Z_{\ell-1}(s)ds,\quad \forall\,\ell\ge 1.
\end{equation}
For $\ell=1$, we decompose $Z_1(t)=Z'_1(t)+Z''_1(t)$ where
$$
\begin{aligned}
Z'_1(t)&=\frac{1}{\nu}e^{(\Lambda,-\Lambda)t}\int_0^te^{(-\Lambda,\Lambda)s}B'(s) e^{(\Lambda,-\Lambda)s}ds
\prec e^{(\Lambda,-\Lambda)t}\int_0^tN_sds,\\
Z''_1(t)&=\frac{r^{2k}}{\nu}e^{(\Lambda,-\Lambda)t}\int_0^te^{(-\Lambda,\Lambda)s} B''(s)e^{(\Lambda,-\Lambda)s}ds
\prec r^{2k-2}e^{(\Lambda,-\Lambda)t}\int_0^tM_sds.
\end{aligned}
$$
Let $\sigma=\max_{1\le i,j\le n} \{|\lambda_i-\lambda'_i+\lambda_j-\lambda'_j|^{-1}, |\lambda_i-\lambda_j+\lambda'_j|^{-1}, |\lambda_\kappa-\lambda_i-\lambda_j|^{-1},(\lambda_i+\lambda_j)^{-1}\}$. Since all elements of $N_t$ and of $M_t$ are exponential function, in particular, the elements of $M_t$ have negative exponents guaranteed by $\lambda_\kappa>\max_{1\le i,j\le n}\{\lambda_i+\lambda_j\}$, we find from \eqref{element-N} that
$$
\int_0^tN_sds\prec \sigma(N_t+\mathbb{I}),\qquad \int_0^tM_sds\prec \sigma(\mathbb{I}-M_t)
$$
Let $\mu_1=1+r^{2\kappa-2}$, we have
\begin{equation}\label{k=1}
\begin{aligned}
Z_1(t)&\prec\sigma e^{(\Lambda,-\Lambda)t}(N_t+\mathbb{I}+r^{2\kappa-2} (\mathbb{I}-M_t))\\
&\prec\sigma e^{(\Lambda,-\Lambda)t}(N_t+\mu_1\mathbb{I})=\bar Z_1.
\end{aligned}
\end{equation}
To get the dominating matrix $\bar Z_2\succ Z_2$, we apply \eqref{inductive1} and \eqref{k=1} while being aware of \eqref{MandNfort}
\begin{equation}\label{barZ-2}
\begin{aligned}
Z_2(t)=&\frac{1}{\nu}e^{(\Lambda,-\Lambda)t}\int_0^te^{(-\Lambda,\Lambda)s} (B'(s)+B''(s))Z_1(s)ds\\
\prec&\sigma e^{(\Lambda,-\Lambda)t}\int_0^te^{(-\Lambda,\Lambda)s}(\bar B'(s)+r^{2\kappa-2}\bar B''(s))\bar Z_{1}(s)ds\\
\prec&\sigma e^{(\Lambda,-\Lambda)t}\int_0^t(N_s+r^{2\kappa-2}M_s)(N_s+\mu_1\mathbb{I})ds.
\end{aligned}
\end{equation}
Writing $N_t$ in a block matrix, we get from \eqref{MandNfort} that for $t\ge 0$
$$
\begin{aligned}
N_t^2=&2\sqrt{|E|}\left[\begin{matrix}\sqrt{|E|}e^{-\Lambda t}\mathbb{I}e^{-\Lambda't}\mathbb{I}e^{(\Lambda-\Lambda')t} &
|E|e^{-\Lambda t}\mathbb{I}e^{-\Lambda't}\mathbb{I}e^{-\Lambda t}\\
e^{(\Lambda-\Lambda')t}\mathbb{I}e^{-\Lambda't}\mathbb{I}e^{(\Lambda-\Lambda')t} & \sqrt{|E|}e^{(\Lambda-\Lambda')t}\mathbb{I}e^{-\Lambda't}\mathbb{I}e^{-\Lambda t}
\end{matrix}\right]\\
\prec&2n\sqrt{|E|}N_t.
\end{aligned}
$$
Each element of $M_tN_t$ is a sum of $2n$ exponential function in $t$. Notice $\kappa\lambda_1>\lambda_n$. We derive from \eqref{MandNfort} that all of the functions have negative exponent $M_tN_t\prec 2ne^{-t/\sigma}\mathbb{I}$. Since each element of $M_t$ is also an exponential function in $t$ with negative exponent, we obtain from \eqref{barZ-2} that
\begin{equation*}\label{barZ-2-2}
\begin{aligned}
Z_2(t)\prec& \sigma e^{(\Lambda,-\Lambda)t}\int_0^tN_s(\mu_1I+2n\sqrt{|E|}\mathbb{I})ds\\
&+\sigma e^{(\Lambda,-\Lambda)t}\int_0^t\mu_1r^{2k-2} M_s\mathbb{I}+2nr^{2k-2}e^{-t/\sigma} \mathbb{I}ds\\
\prec& \sigma e^{(\Lambda,-\Lambda)t}\int_0^tN_s(\mu_1+2n\sqrt{|E|}) \mathbb{I}ds\\
&+\sigma e^{(\Lambda,-\Lambda)t}\int_0^t(1+\mu_1)2nr^{2k-2}e^{-t/\sigma} \mathbb{I}ds\\
\prec&\nu_2\sigma^2e^{(\Lambda,-\Lambda)t}(N_t+\mu_2\mathbb{I}) \mathbb{I}=\bar Z_2(t)
\end{aligned}
\end{equation*}
where $\nu_2=\mu_1+2n\sqrt{|E|}$ and $\mu_2=1+\frac{1+\mu_1}{\mu_1+2n\sqrt{|E|}}r^{2\kappa-2}$. By induction, we have
\begin{equation}\label{barZ}
\bar Z_\ell(t)=(\Pi_{j=2}^\ell\nu_j)\sigma^\ell e^{(\Lambda,-\Lambda)t}(N_t+\mu_\ell\mathbb{I}) \mathbb{I}^{\ell-1}
\end{equation}
where $\mu_{j+1}=1+\frac{1+\mu_j}{\mu_j+2n\sqrt{|E|}}r^{2\kappa-2}$ and $\nu_{j+1}=\mu_j+2n\sqrt{|E|}$.
Let $\mu^*=1+2\sum_{\ell=1}^\infty r^{(2\kappa-1)\ell}$, then $\mu_\ell\le\mu^*$ and $\nu_\ell\le\mu^*+2n\sqrt{|E|}$ for all $\ell$. Since $\mathbb{I}^\ell=(2n)^{\ell-1}\mathbb{I}$, to make the series of matrices
$$
D_0=\sigma\nu I_{2n}+\sum_{\ell=2}^{\infty}(\nu\sigma)^\ell r^{2\ell-2} \prod_{j=2}^\ell\nu_j\mathbb{I}^{\ell-1}
$$
convergent we only need to set $r\le\frac1{4n(\mu^*+2n\sqrt{|E|})\nu\sigma}$. Let $D_1=\mu\mathbb{I}D_0$, we find $Z(t)-e^{(\Lambda,-\Lambda)t}\prec r^2e^{(\Lambda,-\Lambda)t}(N_tD_0+D_1)$.
\end{proof}

\begin{lem}\label{lem7.3new}
Let $z'_{ji}$ be the element of $Z(\tau_E)$ at the crossroad of $i$-th column and $j$-th row, some constant $\mu'_i\ge 1$ and $c>0$ exist such that
\begin{equation}\label{eq7.18}
\begin{aligned}
\mu'^{-1}_i|E|^{-\frac{\lambda_i}{2\lambda_1}}\le z'_{ii}\le&\mu'_i|E|^{-\frac{\lambda_i}{2\lambda_1}}, && i=1,\cdots,n,\\
|z'_{ji}|\le &o(r)\mu'_j|E|^{-\frac{\lambda_j}{2\lambda_1}}, \quad && 1\le j\le n,\ j\ne i,\\
|z'_{j(i+n)}|\le &o(r)\mu'_j|E|^{-\frac{\lambda_j}{2\lambda_1}},\quad && 1\le j\le n,\\
|z'_{(j+n)i}|\le& c|E|^{\frac 16}, \quad && 1\le j\le n,\\
|z'_{(j+n)(i+n)}|\le &c|E|^{\frac 16}, \quad && 1\le j\le n.
\end{aligned}
\end{equation}
\end{lem}
\begin{proof}
Since $\lambda'_i$ is set such that $\lambda_i-\lambda'_i<\frac 13\lambda_1$, all elements in the matrix
$$
e^{(\Lambda,-\Lambda)\tau_E}N_{\tau_E}D_0=\left[\begin{matrix} \sqrt{|E|}\mathbb{I}e^{(\Lambda- \Lambda')\tau_E} & |E|\mathbb{I}e^{-\Lambda \tau_E}\\
e^{-\Lambda' \tau_E}\mathbb{I}e^{(\Lambda-\Lambda')\tau_E} & \sqrt{|E|}e^{- \Lambda'\tau_E} \mathbb{I}e^{- \Lambda'\tau_E}
\end{matrix}\right]D_0,
$$
are bounded by $ce^{\frac13\lambda_1\tau_E}\approx c|E|^{\frac 16}$ because $|\tau_E-\frac{1}{2\lambda_1}\ln\frac{1}{|E|}|$ is bounded as $|E|\to 0$. In this case, some $c\ge 1$ exists such that $c^{-1}|E|^{-\lambda_i/2\lambda_1}\le e^{\lambda_i\tau_E}\le c|E|^{-\lambda_i/2\lambda_1}$. From \eqref{fundamentalM} we see that $z'_{ii}=(1+o(r))e^{\lambda_i\tau_E}$, it leads to the first inequality of \eqref{eq7.18}. The rest of the proof can be done similarly.
\end{proof}

By applying Lemma \ref{lem7.3new} we are able to calculate the differential $d\Phi_{r,0}$. Recall the inner map $\Phi_{r,0}$: $S_r\subset\{v_1=r\}\to\{u_1=v_1\}$ is defined by the flow $\Phi_H^t$. Emanating from a point $z\in S_r$, the orbit $\Phi_H^t(z)$ keeps close to the stable manifold until it arrives at $\{u_1=v_1\}$ after a time $t_z$, we define $\Phi_{r,0}(z)=\Phi_H^{t_z}(z)$.
\begin{lem}\label{lem7.4}
Let $\xi=(\xi_{u_1},\xi_{\hat u},0,\xi_{\hat v})\in T_zS_r$ be a tangent vector. If we write $d\Phi_H^{t_z}(z)\xi=(\xi'_{u_1},\xi'_{\hat u},\xi'_ {v_1},\xi'_{\hat v})$ and $X_H(\Phi_H^{t_z}(z))=(X_{u_1},X_{\hat u},X_{v_1},X_{\hat v})$, then
\begin{equation}\label{differentialmapofsection-1}
d\Phi_{r,0}(z)\xi=d\Phi_H^{t_z}(z)\xi-\frac{\xi'_{u_1}-\xi'_{v_1}}{X_{u_1}-X_{v_1}} X_H(\Phi_H^{t_z}(z)).
\end{equation}
\end{lem}
\begin{proof} Emanating from the points $z,z'\in S_r$, the trajectories arrive at the section $\{u_1=v_1\}$ after the time $t_z$ and $t_{z'}$ respectively. One has $t_{z'}-t_z\to 0$ if $z'\to z$. We have the identity
$$
\begin{aligned}
\Phi_H^{t_{z'}}(z')-\Phi_H^{t_z}(z)=&\Phi_H^{t_{z'}}(z')-\Phi_H^{t_z}(z')+ \Phi_H^{t_z}(z')-\Phi_H^{t_z}(z)\\
=&\Phi_H^{t_{z'}-t_z}\Phi_H^{t_z}(z')-\Phi_H^{t_z}(z')+d\Phi_H^{t_z}(z)(z'-z) +O(|z'-z|^2))\\
=&X_H(\Phi_H^{t_z}(z'))(t_{z'}-t_z)+d\Phi_H^{t_z}(z)(z'-z)\\
&+O(|z'-z|^2,|t_{z'}-t_z|^2).
\end{aligned}
$$
It implies that $d\Phi_{r,0}(z)\xi=d\Phi_H^{t_z}(z)\xi-\nu X_H(\Phi_H^{t_z}(z))$ holds for some $\nu$. The property that the $u_1$-component of $\Phi_H^{t_{z'}}(z')-\Phi_H^{t_z}(z)$ is equal to its $v_1$-component requires $\nu=\frac{\xi'_{u_1}-\xi'_{v_1}}{X_{u_1}-X_{v_1}}$. We are aware that $X_{u_1}\ne X_{v_1}$ at the point $\Phi_H^{t_z}(z)$ because the orbit passes the section $\{u_1=v_1\}$ transversally.
\end{proof}

Treating each column $z'_i=\{z'_{ji}:j\le 2n\}$ of $Z(\tau_E)$ as a vector, we have  $z'_i=Z(\tau_E)e_i$. We set $z_i=z'_i-\frac{z'_{1i}-z'_{(n+1)i}}{X_{u_1}-X_{v_1}} X_H(\Phi_H^{\tau_E}(z))$ for $i=1,2,\cdots,2n$. It follows from Lemma \ref{lem7.4} that $z_1=z_{n+1}$. Because $T_zS_r=\mathrm{Span}\{e_1,\cdots,e_n,e_{n+2},\cdots e_{2n}\}$, we find
$$
\Psi=[z_1,z_2,\cdots,z_n,z_{n+2},\cdots,z_{2n}]
$$
which represents the differential $d\Phi_{r,0}$. For the periodic orbit lying on $H^{-1}(E)$, the property \eqref{position} guarantees
\begin{equation}\label{eq7.20}
\frac{|X_{\hat u}|}{|X_{u_1}-X_{v_1}|},\frac{|X_{\hat v}|}{|X_{u_1}-X_{v_1}|}\le cr^2|E|^\frac{\nu'}{2}.
\end{equation}
We obtain from \eqref{eq7.18} and \eqref{differentialmapofsection-1} that some $\{\mu_i=\mu_i(E)>0:i=1,2,\cdots,n\}$ exist such that $\inf_{E\ne 0}\mu_i(E)>0$
\begin{equation}\label{eq7.21}
\begin{aligned}
z_{ii}=&\mu_i|E|^{-\frac{\lambda_i}{2\lambda_1}}, && i=1,\cdots,n,\\
|z_{ji}|\le &o(r)\mu_j|E|^{-\frac{\lambda_j}{2\lambda_1}}, \quad && 1\le j\le n,\ j\ne i,\\
|z_{j(i+n)}|\le &o(r)\mu_j|E|^{-\frac{\lambda_j}{2\lambda_1}},\quad && 1\le j\le n,\\
|z_{(j+n)i}|\le &o(r)|E|^{-\frac{1-\nu'}{2}}, \quad && 1\le j\le n,\\
|z_{(j+n)(i+n)}|\le &o(r)|E|^{-\frac{1-\nu'}{2}}, \quad && 1\le j\le n.
\end{aligned}
\end{equation}
By the same method, we get the matrix $\Psi_E$ representing the differential $d\Phi_{E,r,0}$. To make $\xi_{1,i}\frac{\partial}{\partial u_1}+\frac{\partial}{\partial u_i},\xi_{1,i+n}\frac{\partial}{\partial u_1}+\frac{\partial}{\partial v_i}\in T_zS_{E,r}$ we find $\xi_{1,i}=-\frac{\lambda_iv_i}{\lambda_1r}(1+O(r))$ and $\xi_{1,i+n}=-\frac{\lambda_iu_i}{\lambda_1r}(1+O(r))$ since $z^+_{E,r}\in\{v_1=r\}$. In view of \eqref{Birkhoffcontrol3}, we find $|\xi_{1,i+n}|\le c|E|^{\frac{1-cr}{2}}$. Because the tangent space $T_zS_{E,r}=\mathrm{Span}\{ \xi_{1,i}\frac{\partial}{\partial u_1}+\frac{\partial}{\partial u_i},\xi_{1,i+n}\frac{\partial}{\partial u_1}+\frac{\partial}{\partial v_i}:i=2,\cdots,n\}$, we obtain the matrix $\Psi_E$ representing $d\Phi_{E,r,0}$ as follows
$$
\Psi_E=[z_{E,2},\cdots,z_{E,n},z_{E,n+2},\cdots,z_{E,2n}]
$$
where $z_{E,j}=z_j+\xi_{1,j}z_1$ and $z_{E,j+n}=z_j+\xi_{1,j+n}z_1$ for $j=2,\cdots,n$. Let $\imath\Psi_E$ be the matrix obtained from $\Psi_E$ by eliminating the $1$-st and the $(n+1)$-th row and let $\imath\Psi$ be the matrix obtained from $\Psi$ by eliminating the $(n+1)$-th row.

\begin{lem}\label{lem7.5}
The matrix $\imath\Psi_E$ has $n-1$ pairs of eigenvalues $\{\sigma_i|E|^{-\frac{\lambda_i}{2\lambda_1}}, \sigma^{-1}_i|E|^{\frac{\lambda_i} {2\lambda_1}}: 2\le i\le n\}$ associated with the eigenvectors $\{\xi_i=e_i+b_i,\xi_{i+n}=e_{i+n}+b_{i+n}\}$ respectively, where $\inf_E\sigma_i>0$, all elements of $b_i,b_{i+n}\in\mathbb{R}^{2(n-1)}$ are bounded by $o(r)$.

The matrix $\imath\Psi$ has $n$ large eigenvalues $\{\bar\sigma_i|E|^{-\frac{\lambda_i} {2\lambda_1}}: 1\le i\le n\}$ associated with the eigenvectors $\{\zeta_i=e_i+o(r)\}$ respectively. Other $n-1$ eigenvalues are not larger than $|E|^{-(1-v')/2}$, the first $n$ elements of their normalized eigenvector are of order $o(r)$.
\end{lem}
\begin{proof}
Since $z_{E,j}=z_j+\xi_{1,j}z_1$ and $z_{E,j+n}=z_j+\xi_{1,j+n}z_1$, we see from \eqref{eq7.21} that the diagonal element in the first $n-1$ rows of the matrix $\Psi_E$ is much larger then other elements in the same row for small $|E|$,
$$
\mu_i|E|^{-\lambda_i/2\lambda_1}+o(r)\xi_{1,i}\mu_1|E|^{-1/2}\gg o(r)(\mu_i|E|^{-\lambda_i/2\lambda_1}+\xi_{1,j}\mu_j|E|^{-1/2})
$$
holds for all $j\ne i$. To consider the characteristic polynomial $F(\sigma)$ of $\Psi_E$, we notice that for $\sigma\ge |E|^{-1/2}$, the diagonal element of $\Psi_E-\sigma I$ in other $(n-1)$ rows is much larger than other elements in the same row. Since
$$
|E|^{-\lambda_2/2\lambda_1}\ll |E|^{-\lambda_3/2\lambda_1}\ll\cdots \ll |E|^{-\lambda_n/2\lambda_1},
$$
we have $F(\frac12 \mu_i|E|^{-\lambda_i/2\lambda_1})F(\frac32 \mu_i|E|^{-\lambda_i/2\lambda_1})<0$ for $i=2,\cdots,n$. It implies that there are at least eigenvalues which are larger than $\frac 12\mu_i|E|^{-\lambda_i/2\lambda_1}$. Since $\imath\Psi_E$ is symplectic, guaranteed by Lemma \ref{map1.2}, the eigenvalues appear in paired $(\sigma,\sigma^{-1})$. Therefore, there exists exactly one eigenvalue lying between $\frac12 \mu_i|E|^{-\lambda_i/2\lambda_1}$ and $\frac32 \mu_i|E|^{-\lambda_i/2\lambda_1}$.

To study the eigenvector $\xi_i$ for $\sigma_i|E|^{-\lambda_i/2\lambda_1}$, we see that the diagonal element of $\Psi_E-\sigma_i|E|^{-\lambda_i/2\lambda_1}I$ in the $j$-th row with $j\ne i$ is much larger than other elements in the same row. So we have $|\xi_{i\ell}|\le o(r)|\xi_{ii}|$ if the notation $\xi_i=\{\xi_{i\ell}:\ell\le 2n, \ell\ne 1,n+1\}$, otherwise one would have $\sum_{\ell\ne j}z_{E,j\ell}\xi_{i\ell}+(z_{E,jj}-\sigma_i|E|^{-\lambda_i/2\lambda_1})\xi_{ij}\ne 0$. Hence, there exists $b_i\in\mathbb{R}^{2(n-1)}$ with $|b_i|=o(r)$ such that $\xi_i=e_i+b_i$.

Let $\xi_{n+i}$ be the normalized eigenvector for $\sigma_i^{-1} |E|^{\lambda_i/2\lambda_1}$, its first $n-1$ elements have to be $o(r)$. Otherwise one would have $\sum_{\ell\ne j}z_{E,j\ell}\xi_{i+n,\ell}+(z_{E,jj}- \sigma_i^{-1}|E|^{\lambda_i/2\lambda_1})\xi_{i+n,j}\ne 0$ if $\xi_{i+n,j}$ is larger than $o(r)$. Since $\imath\Psi_E$ is symplectic, $\langle\xi_{n+i},J\xi_{j} \rangle=\langle\imath\Psi_E\xi_{n+i},J\imath\Psi_E\xi_{j}\rangle= \sigma_j^{-1}\sigma_i|E|^{(\lambda_i-\lambda_j)/2\lambda_1}\langle\xi_{n+i},J\xi_{j} \rangle=0$ holds for all $j\ne i$, the element $\xi_{n+i,j}$ can not be larger than $o(r)$ either for $j\ge n+2$ with $j\ne n+i$. So we have $\xi_{n+i,n+i}=1+o(r)$.

The proof for the properties of $\imath\Psi$ is similar. Due to the lack of symplectic structure in $\imath\Psi$, we only know the smallness of the first $n$ elements of the eigenvector $\zeta_{n+i}$, we are unable to get that $\zeta_{n+i}$ is close to $e_{n+i}$.
\end{proof}

Notice that the $1$-st row of $\Psi$ is the same as its $(n+1)$-th row. If $\zeta_i=(\zeta_{i,1},\zeta_{i,\hat u},\zeta_{i,\hat v})$ is an eigenvector for the eigenvalue $\mu_i$, then $d\Phi_{r,0}(\zeta_{i,1},\zeta_{i,\hat u},0,\zeta_{i,\hat v})=\mu_i(\zeta_{i,1},\zeta_{i,\hat u},\zeta_{i,1},\zeta_{i,\hat v})$. If $(\xi_{i,\hat u},\xi_{i,\hat v})$ is an eigenvector of $\Psi_E$ for the eigenvalue $\mu_i$, then $\exists$ components $\xi_{i,1}$ and $\bar\xi_{i,1}$ such that $d\Phi_{E,r,0}(\xi_{i,1},\xi_{i,\hat u},0,\xi_{i,\hat v})=\mu_i(\bar\xi_{i,1}, \xi_{i,\hat u},\bar\xi_{i,1},\xi_{i,\hat v})$ and both vectors lie in the tangent space of the energy level set.

Let $\Psi'$ and $\Psi'_E$ be the matrix of the tangent map $d\Phi_{0,r}$ and $d\Phi_{E,0,r}$ respectively.
\begin{cor}
The matrix $\imath\Psi'_E$ has $2(n-1)$ eigenvalues $\{ \sigma'_i|E|^{-\frac{\lambda_i}{2\lambda_1}},\sigma'^{-1}_i|E|^{\frac{\lambda_i} {2\lambda_1}}:  2\le i\le n\}$ associated with the eigenvectors $\{\xi'_i=e_i+b'_i,\xi'_{i+n}=e_{i+n}+b'_{i+n}\}$ respectively, where $\inf_E\sigma'_i>0$, all elements of $b'_i,b'_{i+n}\in\mathbb{R}^{2(n-1)}$ are bounded by $o(r)$.

The matrix $\imath\Psi'$ has $n$ small eigenvalues $\{\bar\sigma'^{-1}_i|E|^{\frac{\lambda_i} {2\lambda_1}}: 1\le i\le n\}$ associated with the eigenvectors $\{\zeta'_i=e_{i+n-1}+o(r)\}$ respectively. Other $n-1$ eigenvalues are not smaller than $|E|^{(1-v')/2}$, the last $n$ elements of their normalized eigenvector $\{\zeta'_{n+2},\cdots,\zeta'_{2n}\}$ are of order $o(r)$.
\end{cor}
\begin{proof}
If we exchange the places of $u$ with $v$, the differential $d\Phi^{-1}_{0,r}$ and $d\Phi^{-1}_{E,0,r}$ is found by the same method to find $d\Phi_{r,0}$ and $d\Phi_{E,r,0}$. Let $\Psi'_{-}$ and $\Psi'_{E-}$ be the matrix of $d\Phi^{-1}_{0,r}$ and $d\Phi^{-1}_{E,0,r}$ respectively, Lemma \ref{lem7.5} works for the eigenvectors and the eigenvalues of $\imath\Psi'_{-}$ and $\imath\Psi'_{E-}$. Since the inverse of the map has the same eigenvectors, the proof is completed if we exchange the place of $u$ with $v$.
\end{proof}

To prove Proposition \ref{pro6.2} concerning the composition of $d\Phi_{E,r,r}=d\Phi_{E,0,r}d\Phi_{E,r,0}$ and $d\Phi_{r,r}=d\Phi_{0,r}d\Phi_{r,0}$, we apply the following proposition by postponing the proof to the end of this section.

\begin{pro}\label{productofeigenvalue}
Let $\Psi$ and $\Psi'$ be linear maps $\mathbb{R}^d\to\mathbb{R}^d$. Assume $\mathbb{R}^d$ admits decomposition of subspaces
$\mathbb{R}^d=E_s\oplus E_0\oplus E_\ell=E'_s\oplus E'_0\oplus E'_\ell$ such that
\begin{enumerate}
  \item $E_s,E,E_\ell$ are invariant for $\Psi$ and $E'_s,E',E'_\ell$ are invariant for $\Psi'$;
  \item $\mathrm{dim}E_s=\mathrm{dim}E'_s$, $\mathrm{dim}E_0=\mathrm{dim}E'_0=1$ and $\mathrm{dim}E_\ell=\mathrm{dim}E'_\ell$;
  \item some numbers $\sigma_\ell>\sigma_0>\sigma_s>0$ and $\sigma'_\ell>\sigma'_0>\sigma'_s>0$ exist such that
    $$
    \begin{aligned}
    &|\Psi v|\ge \sigma_\ell|v|,\ \ \ \forall\ v\in E_\ell,\quad &&|\Psi' v|\ge \sigma'_\ell|v|,\ \ \ \forall\ v\in E'_\ell;\\
    &|\Psi v|\le \sigma_s|v|,\ \ \ \forall\ v\in E_s,\quad &&|\Psi' v|\le \sigma'_s|v|,\ \ \ \forall\ v\in E'_s;\\
    &|\Psi v|= \sigma_0|v|,\ \ \ \forall\ v\in E_0,\quad &&|\Psi' v|= \sigma'_0|v|,\ \ \ \forall\ v\in E'_0
    \end{aligned}
    $$
    and
    $$
    \min\Big\{\frac{\sigma'_\ell}{\sigma'_0},\frac{\sigma_\ell}{\sigma_0}, \frac{\sigma_0}{\sigma_s},\frac{\sigma'_0}{\sigma'_s}\Big\}\ge 4;
    $$
  \item $\exists$ $\alpha<1$ such that $|\langle v_\imath,v_\jmath\rangle|\le\alpha|v_\imath||v_\jmath|$ holds $\forall$ $v_\imath\in E_\imath,v_\jmath\in E_\jmath$ with $\imath\ne\jmath$ where $\imath,\jmath\in\{s,0,\ell\}$;
  \item the subspace $E_\imath$ is close to $E'_\imath$ for $\imath=s,0,\ell$ in the following sense, for any $v\in E_\imath$ $(v'\in E'_\imath\ resp.)$, some $v'\in E'_\imath$ $(v\in E_\imath\ resp.)$ exists such that $\langle v,v'\rangle\ge |v||v'|(1-\delta)$ holds for some small $\delta\ge 0$.
\end{enumerate}
Then, $\exists$ small $\delta_0>0$ such that for $\delta\in(0,\delta_0]$, the map $\Psi^*=\Psi'\Psi$ has an eigenvalue $\sigma^*_0=\sigma'_0\sigma_0(1+O(\delta))$ associated with an eigenvector $v^*_0$ satisfying the condition that $\langle v,v^*_0\rangle\ge |v||v^*_0|(1-O(\delta))$ holds for $v\in E_0$. The quantities $O(\delta)$ are independent of the size of the eigenvalues, only depend on the ratio $|\frac{\sigma_0\sigma'_0}{\sigma_s\sigma'_s}|$ and $|\frac{\sigma_\ell\sigma'_\ell}{\sigma_0\sigma'_0}|$.
\end{pro}

\begin{proof}[Proof of Proposition \ref{pro6.2}]
We set $E_\ell=\mathrm{Span}\{\zeta_2,\cdots,\zeta_n\}$, $E_s=\mathrm{Span} \{\zeta_{n+2}, \cdots,\zeta_{2n}\}$ and $E_0=\mathrm{Span}\{\zeta_1\}$ for the matrix $\Psi$, set $E'_\ell=\mathrm{Span}\{\zeta'_{n+2},\cdots, \zeta'_{2n}\}$, $E_0=\mathrm{Span}\{\zeta'_1\}$ and $E_s=\mathrm{Span}\{\zeta'_2,\cdots,\zeta'_n\}$ for the matrix $\Psi'$. Under such setting, we have
$$
\begin{aligned}
&\frac{\sigma_\ell}{\sigma_0}=\frac{\sigma_2}{\sigma_1}|E|^{-\frac{\lambda_2-\lambda_1} {2}}, \qquad &&\frac{\sigma_0}{\sigma_s}\ge \sigma_1|E|^{-\frac{\nu'}{2}},\\
&\frac{\sigma'_\ell}{\sigma'_0}\ge \frac 1{\sigma_1}|E|^{-\frac{\nu'}{2}}, \quad && \frac{\sigma'_0}{\sigma'_s}=\frac{\sigma_1}{\sigma_2}|E|^{-\frac{\lambda_2-\lambda_1} {2}}.
\end{aligned}
$$
They are quite large for small $|E|$. Clearly, $E_\imath$ is close to $E'_\imath$ for $\imath=s,0,\ell$. Applying Proposition \ref{productofeigenvalue} we see that $\imath\Psi'\imath\Psi$ has an eigenvalue $\sigma^*_1=1+o(r)$ associated with an eigenvector $e_1+o(r)$. It corresponds to a vector $\eta_1=(1,\eta_{1,\hat u},0,\eta_{1,\hat v})\in T_{z^+_{E,r}}S_r$ that is mapped by $d\Phi_{r,r}$ to a vector $(1+o(r))(1,\eta_{1,\hat u},0,\eta_{1,\hat v})\in T_{z^-_{E,r}}U_r$ where $|\eta_{1,\hat u}|,|\eta_{1,\hat v}|=o(r)$.
The same method applies in the study of the eigenvectors of $\Psi_E^*=\Psi'_E\Psi_E$. Hence, the proof of the proposition is completed if we prove the following

\begin{lem}\label{lem7.6}
Given a symplectic matrix $M$, if its spectrum consists of $2d$ different real numbers $\{\sigma_i,\sigma_i^{-1}:i=1,\cdots,d\}$, associated with the eigenvectors $\eta_i$ and $\eta_{i+d}$ respectively, then the matrix $\Psi=[\eta_1,\cdots,\eta_d,\eta_{1+d},\cdots,\eta_{2d}]$ is symplectic if a suitable factor $\nu_i$ is applied to each $\eta_i$ for $i=1,\cdots,d$.
\end{lem}

Indeed, because $M$ is symplectic, we have $\lambda_j\lambda_i\langle \eta_i,J\eta_j\rangle=\langle M\eta_i,JM\eta_j\rangle= \langle\eta_i,J\eta_j\rangle$. So, it has to be zero if $\lambda_j\lambda_i\ne 1$. It implies that
$$
\Psi^tJ\Psi=\left[\begin{matrix} 0 & \Upsilon \\
-\Upsilon & 0\end{matrix}\right]
$$
where $\Upsilon=\mathrm{diag}\{\langle\eta_1,J\eta_{d+1}\rangle,\cdots, \langle\eta_d,J\eta_{2d}\rangle\}$. Clearly, $\langle\eta_i,J\eta_{i+d}\rangle\ne 0$, otherwise $M$ would be degenerate. Let $\nu_i^{-1}=\langle\eta_i, J\eta_{d+i}\rangle$, one has $\langle\nu_i\eta_i,J\eta_{d+i}\rangle=1$. Applying this lemma to the eigenvectors of $d\Phi_{E,r,r}$, the matrix $T_E=[\xi^*_2,\cdots,\xi^*_n, \xi^*_{2+n},\cdots,\xi^*_{2n}]$ of the eigenvectors can be made symplectic.
\end{proof}


\begin{proof}[Proof of Proposition \ref{pro6.6}]
We consider the inner map $\Phi_{r,r}$. For $E>0$, emanating from the point $z^+_{E,r}$ at $t=0$, the periodic orbit $z^+_E(t)$ arrives at the point $z^-_{E,r}$ after a time $t_E$ satisfying the condition \eqref{eq2.9}. Given any small $\varepsilon>0$, there exists $E(\varepsilon)>0$ such that for $E\in(0,E(\varepsilon)]$, the orbit $z^+_E(t)$ passes through the disk $|z|\le\varepsilon$ before it arrives at $z^-_{E,r}$.
Let $t'_E<t''_E$ be the time when the periodic orbit passes through the section $\{v_1=\varepsilon\}$ and $\{u_1=\varepsilon\}$ respectively, then $z^+_E|_{[0,t'_E]}$ keeps close to the stable manifold, $z^+_E|_{[t''_E,t_E]}$ keeps close to the unstable manifold and $z^+_E|_{[t'_E,t''_E]}$ remains in the disk $|z|\le\varepsilon$. Clearly, some finite $t',t''$ exists such that $t'_E\to t'$, $t_E-t''\to t''$ and $t''_E-t'_E\to\infty$ as $E\to 0$.

Recall that $z_E(t)$ passes through the section $\{u_1=v_1\}$ at the time $t=\tau_E$. With the experience to prove Proposition \ref{pro6.2}, let $Z'_{E}(t)$, $Z^*_{E}(t)$ be the fundamental matrix of the variational equation of the Hamiltonian \eqref{eq2.2} along the orbit $z^+_E|_{[0,t'_E]}$, $z^+_E|_{[t'_E,\tau_E]}$ respectively with $Z'_{E}(0)=Z^*_{E}(0)=I$. So,  $Z_E(\tau_E)=Z^*_{E,\varepsilon} (\tau_E-t'_E)Z'_{E,\varepsilon}(t'_E)$ is the fundamental matrix of the variational equation along the orbit $z^+_E|_{[0,\tau_E]}$.

Let $z^+_+(t)|_{[0,t']}$ be a piece of the homoclinic orbit $z^+(t)$ such that $z^+_+(0)=z^+_r$, thus we have $z^+_E(t'_E)\to z^+(t')$ as $E\to 0$. Let $Z'_0(t)$ be the fundamental matrix along $z^+_+(t)|_{[0,t']}$ such that $Z'_0(0)=I$. Clearly, $Z'_{E}(t'_E)\to Z'_0(t')$ as $E\to 0$.

From the special form of the Hamiltonian \eqref{eq2.5}, we are able to get more information about the fundamental matrix $Z'_{\varepsilon}(t')$. Notice $\partial_IN=0$ when it is restricted on the stable or unstable manifold since $N$ is a function of $(u_1v_1,\cdots,u_nv_n)$ without linear term. Because $\partial^2_{vv}R(z^+_+(t)|_{[0,t']})=0$, the variational equation of the the Birkhoff normal form \eqref{eq2.3} along $z^+_+(t)|_{[0,t']}$ takes the form
\begin{equation}\label{equation6-16}
\begin{aligned}
\dot\xi_u&=(\Lambda+\partial_{vu}R)\xi_u,\\
\dot\xi_v&=-(\Lambda+\partial_{uv}R)\xi_v-\partial^2_{uu}R\xi_u,
\end{aligned}
\end{equation}
where $\Lambda=\mathrm{diag}\{\lambda_1,\cdots,\lambda_n\}$. The terms $\partial^2_{uv}R$ and $\partial^2_{vv}R$ depend on the $v$-component of $z^+_+(t)$  only if we write $z^+_+(t)=(u^+_+(t),v^+_+(t))$ since $u^+_+(t)=0$. Notice that the first equation is independent of $\xi_v$, we find that the fundamental matrix takes the form
$$
Z'_0(t)=\left[\begin{matrix}\Psi_{11}(t) & 0\\
\Psi_{12}(t) & \Psi_{22}(t)
\end{matrix}\right]
$$
where $\Psi_{11}(t)$ is the fundamental matrix of the first equation of \eqref{equation6-16}, $\Psi_{22}(t)$ is the one of the equation $\dot\xi_v=-(\Lambda+\partial_{uv}R)\xi_v$ and
\begin{equation}\label{Psi12}
\Psi_{12}(t)=-\Psi_{22}(t)\int_0^t\Psi_{22}^{-1}(s)\partial^2_{uu}R(0,v^+_+(s)) \Psi_{11}(s)ds.
\end{equation}
Since $\partial^2_{uu} R=O(|z|^{2\kappa-1})$, we expand $\Psi_{11}(t)$ into a sequence of $\rho^\ell=r^{(2\kappa-1)\ell}$
$$
\Psi_{11}(t)=\sum_{\ell=0}^\infty\rho^\ell \Psi_{11,\ell}(t).
$$
The matrices $\{\Psi_{11,\ell}(t)\}$ are obtained inductively. Clearly $\Psi_{11,0}=e^{\Lambda t}$ and for $\ell\ge 1$ one has
$$
\Psi_{11,\ell}(t)=\frac{e^{\Lambda t}}{r^{2\kappa-1}}\int_{0}^{t}e^{-\Lambda s}\partial^2_{vu}R(0,v^+_+(s))\Psi_{11,\ell-1}(s)ds.
$$
As each element of $\partial^2_{vu}R$ decreases to zero not slower than $v^{2\kappa-1}(t)\le cr^{2\kappa-1}e^{-(2\kappa-1)\lambda_1t}$, each element in the integrands is dominated by a exponential function with negative exponent, its coefficient is bounded by $cr^{2\kappa-1}$. Therefore,
\begin{equation*}\label{psi11}
\Psi_{11,\ell}(t')\prec ce^{\Lambda t'}\mathbb{I}, \qquad \forall\ \ell\ge 1.
\end{equation*}
The method is also applied to get an estimate on the fundamental matrix $\Psi_{22}(t)=\sum_{\ell=0}^\infty\rho^\ell\Psi_{22}(t)$ such that $\Psi_{22,0}(t)=e^{-\Lambda t'}$ and
\begin{equation*}\label{psi22}
\Psi_{22,\ell}(t')\prec ce^{-\Lambda t'}\mathbb{I},\qquad \forall\ \ell\ge 1.
\end{equation*}
Hence, the absolute value of each element in the matrix $\Psi_{22}^{-1}(s)\partial^2_{uu}R(0,v^+_+(s))\Psi_{11}(s)$ is bounded by $ce^{(2\lambda_n-(2\kappa-1)\lambda_1)s}$. So, if $(\kappa-1)\lambda_1>\lambda_n$ holds, we obtain from \eqref{Psi12} that some larger constant $c>0$ exists such that
\begin{equation*}\label{psi12}
\Psi_{12}(t')-e^{\Lambda t'}\prec cr^{2\kappa-1}e^{-\Lambda t'}\mathbb{I}.
\end{equation*}

Because $Z'_{E}(t'_E)\to Z'_{0}(t')$ as $E\to 0$, for any small $\epsilon>0$, some $E(\epsilon)>0$ exists such that the following holds for any $E\in(0,E(\epsilon)]$
\begin{equation}\label{ZforsmallE}
Z'_{E}(t'_E)-e^{(\Lambda,-\Lambda)t'_E}\prec cr^{2\kappa-1}e^{(\Lambda,-\Lambda)t'_E}\left[\begin{matrix} \mathbb{I} & \epsilon\mathbb{I}\\
\mathbb{I} & \mathbb{I}\end{matrix}\right]\\
\end{equation}

Let $t^*_E$ be the time so that $\Phi_H^{t^*_E+t'_E}(z^+_{E,r})\in\{u_1=v_1\}$, we apply Lemma \ref{lem7.2} to study $d\Phi_H^{t^*_E}(\Phi_H^{t'_E}(z^+_{E,r}))$. Hence, along the orbit $\Phi_H^{t}(\Phi_H^{t'_E}(z^+_{E,r}))|_{[0,t^*_E]}$, the fundamental matrix $Z^*_E(t)$ satisfies the relation
$$
Z^*_E(t)-e^{(\Lambda,-\Lambda)t}\prec \varepsilon^2e^{(\Lambda,-\Lambda)t}(N_tD_0+D_1)
$$
where all elements in the matrices $D_0,D_1$ are of order 1 and $N_t$ is defined as in \eqref{MandNfort}. So we have
$$
\begin{aligned}
Z'_E(t)&=\left[\begin{matrix} e^{\Lambda t}(I+B_{r,11}) & \epsilon e^{\Lambda t}B_{r,12}\\
e^{-\Lambda t}B_{r,21} & e^{-\Lambda t}(I+B_{r,22})
\end{matrix}\right],\\
Z^*_E(t)&=\left[\begin{matrix} e^{\Lambda t}(I+\varepsilon^2 B_{11}) & \varepsilon^2 e^{\Lambda t}B_{12}\\
\varepsilon^2 e^{-\Lambda t}B_{t} & e^{-\Lambda t}(I+\varepsilon^2 B_{22})
\end{matrix}\right]\\
\end{aligned}
$$
where $B_t\prec e^{(\Lambda-\Lambda')t}\mathbb{I}e^{(\Lambda-\Lambda')t}$, $|B_{r,ij}|\le cr^{2\kappa-1}$, $|B_{ij}|=O(1)$. Let $Z=Z^*_E(t^*_E)Z'_{E,\varepsilon}(t'_E)$ and write
$$
Z=\left[\begin{matrix}Z_{11} & Z_{12}\\ Z_{21} & Z_{22}
\end{matrix}\right]
$$
where each block is a matrix of order $n$, then
$$
\begin{aligned}
Z_{12}&=\epsilon e^{\Lambda t^*_E}(I+\varepsilon^2B_{\varepsilon,11})e^{\Lambda t'_E}B_{r,12}+\varepsilon^2 e^{\Lambda t^*_E}B_{\varepsilon,12}e^{-\Lambda t'_E}(I+B_{r,22}),\\
\end{aligned}
$$
On the other hand, by applying Lemma \ref{lem7.2} to the variational equation along the orbit $\Phi_H^{t}(z^+_{E,r})|_{[0,t^*_E+t'_E]}$ directly, we find
\begin{equation}\label{FundamentalMatrix}
Z^*_E(t^*_E)Z'_{E,\varepsilon}(t'_E)-e^{(\Lambda,-\Lambda)(t^*_E+t'_E)}\prec r^2e^{(\Lambda,-\Lambda)(t^*_E+t'_E)}(N_{t^*_E+t'_E}D_0+D_1)
\end{equation}
The matrix $Z$ represents the tangent map $d\Phi_H^{\tau_E}$ with $\tau_E=t'_E+t^*_E$, which results in the maps $d\Phi_{r,0}$ and $d\Phi_{E,r,0}$.

Let $z_i=d\Phi_{r,0}\frac{\partial}{\partial u_i}$, $z_{n+i}=d\Phi_{r,0}\frac{\partial} {\partial v_i}$ and notice $z'_i=Z e_i$, $z'_{n+i}=Z e_{n+i}$ are the $i$-th and $(n+i)$-th column of $Z$ respectively. By applying Lemma \ref{lem7.4} and in view of \eqref{eq7.20} we obtain \eqref{eq7.21} again where the second inequality is improved by applying the special form of $Z_{12}$
\begin{equation}\label{eq7.26}
|z_{j(i+n)}|\le o(r)(\epsilon(1+\varepsilon^2e^{(\lambda_n-\lambda_1)t'_E})+\varepsilon^2) |E|^{-\frac{\lambda_j}{2\lambda_1}},
\end{equation}
we notice that $t'_E$ remains bounded as $E\to 0$. As we did before, we obtain from the matrix $Z$ the matrix $\Psi$ representing $d\Phi_{r,0}$ and the matrix $\Psi_E$ representing $d\Phi_{E,r,0}$.
$$
\begin{aligned}
\Psi&=[z_1,z_2,\cdots,z_n,z_{n+2},\cdots,z_{2n}];\\
\Psi_E&=[z_{E,2},\cdots,z_{E,n},z_{E,n+2},\cdots,z_{E,2n}]
\end{aligned}
$$
where $z_{E,j}=z_j+\xi_{1,j}z_1$ and $z_{E,j+n}=z_{j+n}+\xi_{1,j+n}z_1$ with $\xi_{1,i}=-\frac{\lambda_iv_i}{\lambda_1r}(1+O(r))$ and $\xi_{1,i+n}=-\frac{\lambda_iu_i}{\lambda_1r}(1+O(r))$. Since the correction terms $\xi_{1,j}z_1$ and $\xi_{1,j+n}z_1$ are relatively small, we still have $z_{E,ii}\ge \frac 12\mu_i|E|^{-\lambda_i/2\lambda_1}$ for $i=2,\cdots,n$, the first inequality of \eqref{eq7.21} and \eqref{eq7.26} also hold for $\{z_{E,ij}: i=2,\cdots,n\}$.

Recall that $\imath\Psi_E$ denote the matrix obtained from $\Psi_E$ by eliminating the $1$-st and the $(n+1)$-th row.

\begin{lem}\label{lem7.9}
Lemma \ref{lem7.5} holds for $\Psi_E$ with extra properties: if $\xi_i=e_{i-1}+b_i$ and $\xi_{i+n}=e_{i+n-2}+b_{i+n}$ denote the eigenvectors for $\sigma_i|E|^{-\lambda_i/2\lambda_1}$ and $\sigma^{-1}_i|E|^{\lambda_i/2\lambda_1}$ respectively and $b_i=(b_{i,\hat u},b_{i,\hat v})$, $b_{n+i}=(b_{n+i,\hat u}, b_{n+i,\hat v})$ then $|b_{n+i,\hat u}|\le \epsilon(1+\varepsilon^2 e^{(\lambda_n-\lambda_1)t'_E})+\varepsilon^2$ and $|b_{i,\hat v}|\le |E|^{\nu_i/2\lambda_1}$ if $0<\nu_i<\lambda_i-(1-\nu')\lambda_1$.
\end{lem}
\begin{proof}
Let $\psi_j$ denote the $j$-th row of $\imath\Psi_E-\sigma_i^{-1}|E|^{\lambda_i/2\lambda_1}I$. If $|b_{n+i,\hat u}|$ is reached at its $j$-th element $b_{n+i,u_{j+1}}$ which is not smaller than $\epsilon(1+\varepsilon^2e^{(\lambda_n-\lambda_1)t'_E})+\varepsilon^2$, we see from the first inequality in \eqref{eq7.21} and \eqref{eq7.26} that the term $(z_{E,jj}-\sigma_i^{-1}|E|^{\lambda_i/2\lambda_1})b_{n+i,u_{j+1}}$ is much larger than all other terms in $\langle\psi_{j},\xi_{n+i}\rangle$, because $|z_{E,jj}|\ge c|E|^{-\lambda_i/2\lambda_1}$. But it is absurd since $\langle\psi_{j},\xi_{n+i}\rangle=0$.

Let $\psi_{n+j}$ denote the $(n+j-2)$-th row of $\imath\Psi_E-\sigma_i|E|^{-\lambda_i/2\lambda_1}I$. If $|b_{i,\hat v}|$ is reached at its $(j-1)$-th element $b_{i,v_j}$ which is not smaller than $|E|^{\nu_i/2\lambda_1}$, we see from the third and the fourth inequalities in \eqref{eq7.21} that the term $|(z_{E,(n+j)(n+j)}-\sigma_i|E|^{-\lambda_i/2\lambda_1})b_{i,v_j}|$ is much bigger than all other terms in $\langle\varphi_j,\xi_i\rangle$ because $|z_{E,(n+j)(n+j)}|\le c|E|^{-(1-\nu')/2}$. It contradicts the fact that $\langle\varphi_j,\xi_i\rangle=0$.
\end{proof}

To study the inner map $\Phi_{E,r,r}$, we consider the map $\Phi_{E,0,r}$: $H^{-1}(E)\cap\{u_1=v_1\}\to H^{-1}(E)\cap\{u_1=r\}$, defined by the flow $\Phi_H^t$. Emanating from $\Phi^{t}_H(z^+_{E,r})|_{t=\tau_E}$, the orbit arrives at the point $z^-_{E,r}\in\{u_1=r\}$ after a time $\tau'_E$. So we have $t_E=\tau_E+\tau'_E$. The inverse of $d\Phi_{E,0,r}$ has the same property as $d\Phi_{E,r,0}$ if we exchange the place $u$ and $v$. Therefore, in virtue of Lemma \ref{lem7.5} and \ref{lem7.9}, we have

\begin{lem}\label{lem7.10}
The map $d\Phi_{E,0,r}$ has $n-1$ pairs of eigenvalues $\{\sigma'_i|E|^{-\frac{\lambda_i}{2\lambda_1}}, \sigma'^{-1}_i|E|^{\frac{\lambda_i} {2\lambda_1}}: 2\le i\le n\}$ associated with the eigenvectors $\{\xi'_i=e_i+b'_i,\xi'_{i+n}=e_{i+n}+b'_{i+n}\}$ respectively, where $|b'_i|,|b'_{i+n}|\le o(r)$ with extra properties $|b'_{i,\hat v}|\le \epsilon(1+\varepsilon^2 e^{(\lambda_n-\lambda_1)t''_E})+\varepsilon^2$, and $|b'_{n+i,\hat u}|\le |E|^{\nu_i/2\lambda_1}$ with $0<\nu_i<\lambda_i-(1-\nu')\lambda_1$.
\end{lem}

By applying Proposition \ref{productofeigenvalue} on $\Phi_{E,r,r}=\Phi_{E,0,r}\Phi_{E,r,0}$, we see that $d\Phi_{E,r,r}$ has $(n-1)$ pairs of eigenvalues $\{\mu_i|E|^{-\lambda_i/\lambda_1}, \mu^{-1}_i|E|^{\lambda_i/\lambda_1}: i=2\cdots n\}$ with the eigenvectors $\hat\eta_i=e_{i}+o(r)$ and $\hat\eta_{i+n}=e_{i+n}+o(r)$ respectively. We next exploit more precise properties of $\hat\eta_i-e_i$ and $\hat\eta_{i+n}-e_{i+n}$. With
$$
\alpha=2\max\{\epsilon(1+\varepsilon^2 e^{(\lambda_n-\lambda_1)t})+\varepsilon^2, |E|^{\nu_i/2\lambda_1}:t\in\{t'_E,t''_E\},i=2,\cdots,n\},
$$
we define the cones
$$
\begin{aligned}
\hat K^-_{\alpha}&=\{(\hat u,\hat v)\in \mathbb{R}^{2n-2}:\alpha|\hat u|\ge |\hat v|\},\\
\hat K^+_{\alpha}&=\{(\hat u,\hat v)\in \mathbb{R}^{2n-2}:\alpha|\hat v|\ge |\hat u|\}.
\end{aligned}
$$
We are going to show that both $d\Phi_{E,r,0}$ and $d\Phi_{E,0,r}$ map the cone $\hat K^-_{\alpha}$ into itself and their inverse maps $\hat K^+_{\alpha}$ into itself either.

Let $E_u=\mathrm{Span}\{\xi_2,\cdots,\xi_n\}$ and $E_v=\mathrm{Span}\{\xi_{2+n}, \cdots,\xi_{2n}\}$. Any $\xi\in\mathbb{R}^{2(n-1)}$ has a decomposition $\xi=\xi_u+\xi_v$ such that $\xi_u\in E_u$ and $\xi_v\in E_v$. Because of Lemma \ref{lem7.9}, we have $\xi_u\in\hat K^-_{\alpha/2}$. Hence, $\xi\in K^-_{\alpha}$ implies $|\xi_v|\le\frac{\alpha}{2}|\xi_u|$. In the decomposition $\xi'=d\Phi_{E,r,0}\xi=\xi'_u+\xi'_v$ with $\xi_u\in\hat K^-_{\alpha/2}$, $\xi'_u\in E_u$ and $\xi'_v\in E_v$, we have $$
|\xi'_v|\le \mu_2^{-1}|E|^{\frac{\lambda_2}{\lambda_1}}|\xi_v|\le\frac{\alpha}{2\mu_2} |E|^{\frac{\lambda_2}{\lambda_1}}|\xi_u|\le\frac{\alpha}{2\mu_2^2} |E|^{2\frac{\lambda_2}{\lambda_1}}|\xi'_u|.
$$
It implies that $\xi'\in\hat K^-_{\alpha/2}$ provided $|E|$ is small, i.e. $\hat K^-_{\alpha}$ is invariant for $d\Phi_{E,r,0}$. With the same argument, we see that $\hat K^-_{\alpha}$ is invariant for $d\Phi_{E,0,r}$. It proves the invariance of $\hat K^-_{\alpha}$ for $d\Phi_{E,r,r}$. By the same reason, we see that the cone $K^+_\alpha$ is invariant for the inverse of $d\Phi_{E,r,r}$.

Obviously, the eigenvectors $\{\hat\eta_2,\cdots,\hat\eta_n\}$ fall into the cone $\hat K^-_\alpha$ and the eigenvectors $\{\hat\eta_{2+n},\cdots,\hat\eta_{2n}\}$ fall into the cone $\hat K^+_\alpha$. Let $\hat\eta_i=(\eta_{i,\hat u},\eta_{i,\hat v})$ and $\hat\eta_{i+n}=(\eta_{i+n,\hat u},\eta_{i+n,\hat v})$, that $\hat\eta_i\in\hat K^-_{\alpha}$ and $\hat\eta_{i+n}\in\hat K^+_{\alpha}$ implies
$$
|\eta_{i,\hat v}|\le \alpha|\eta_{i,\hat u}|,\qquad
|\eta_{i+n,\hat u}|\le \alpha\eta_{i+n,\hat v}|.
$$
We can choose $\alpha\to 0$ as $E\to 0$ because $\varepsilon$ can be set sufficiently small if $|E|$ is small and $\epsilon\to 0$ as $E\to 0$. So, we complete the proof for $d\Phi_{E,r,r}$. The proof for $d\Phi_{-E,-r,r}$ and for $d\Phi_{-E,r,-r}$ is similar.
\end{proof}

What remains to complete the section is the proof of Proposition \ref{productofeigenvalue}, we do it now.

\begin{proof}[Proof of Proposition \ref{productofeigenvalue}]
Let $v_0\in E_0$ and $v'_0\in E'_0$ be unit vector such that $\langle v,v'_0\rangle\ge |v||v'_0|(1-\delta)$, we consider codimension-one affine manifolds $L=E_s\oplus E_\ell +v_0$ and $L'=E'_s\oplus E'_\ell +v'_0$. A map $T$ between $L$ and $L'$ is introduced as follows. Connecting a point $v\in L$ with the origin, we get a line that intersects $L'$ at a point $v'$. The map $T$ is defined such that $T$: $v\to Tv=v'$. Since $E_\imath$ is close to $E'_\imath$ for $\imath=s,0,\ell$, $T$ is an affine map close to identity, some constant $\mu=\mu(\alpha)\ge 1$ exists such that $|T0|\le \mu\delta$ and $\|DT-I\|\le \mu\delta$.

Each point $v\in L$ admits a decomposition $v=v_s+v_\ell+v_0$. Correspondingly the point $v'=Tv$ admits a decomposition $v'=v'_s+v'_\ell+v'_0$ with $\langle v_\imath,v'_\imath\rangle\ge |v_\imath||v'_\imath|(1-O(\delta))$ for $\imath=s,\ell$. The map $\Psi$ induces a map $\Psi_0$: $L_0\to L_0$
\begin{equation}\label{shensuo}
\Psi_0 v=\frac{1}{\sigma_0}\Psi v=\frac{1}{\sigma_0}\Psi(v_s+v_\ell)+v_0,
\end{equation}
The map $\Psi'_0$: $L'\to L'$ is defined similarly. We consider the map $M=T^{-1}\Psi'_0T\Psi_0$, it induces a contraction map on graphs as we are going to study in the following.


For affine map $F:E_s\to E_\ell$, its graph is defined to be the set $\mathcal{G}_F=\{(z_s,F(z_s),v_0):z_s\in E_s\}$. Any affine map $F$ induces another affine map $\mathscr{M}^{-1}F$ such that $M^{-1}\mathcal{G}_F= \mathcal{G}_{\mathscr{M}^{-1}F}$.
We introduce a set of affine maps $\mathfrak{F}_{R,N}$: $F\in\mathfrak{F}_{R,N}$ implies $\|F\|\le N$ and $\|DF\|\le 1$, where  $\|F\|=\max_{|z_s|\le R}|F(z_s)|$.

\begin{lem}\label{lem7.3}
If $1\le N\le 2R$, $\mu\delta\le\frac{1}{12}$, $\min\{\frac{\sigma'_{\ell}}{\sigma'_0}, \frac{\sigma_{\ell}}{\sigma_0}\}\ge 2\max\{2,1+\mu\delta R\}$ and $\min\{\frac{\sigma_0}{\sigma_{s}},\frac{\sigma'_0}{\sigma'_{s}}\}\ge\sqrt{2}$, then there exists a unique $F\in\mathfrak{F}_{R,N}$ such that $\mathscr{M}^{-1}F=F$.
\end{lem}
\begin{proof}
Let $\pi_\imath$: $\mathbb{R}^d\to E_\imath$ denote the projection for $\imath=s,\ell$.
For a vector $\xi\in T_z\mathcal{G}_F$, let $\xi_s=\pi_s\xi$ and $\xi_\ell=\pi_\ell\xi$. If $\|DF\|\le 1$, then $|\xi_\ell|\le |\xi_s|$ holds for $(\xi_s,\xi_\ell)\in T_v\mathcal{G}_F$. Since $\mathscr{M}^{-1}=\Psi_0^{-1}T^{-1}\Psi'^{-1}_0T$, if we write $(\xi'_s,\xi'_\ell)=DT(\xi_s,\xi_\ell)$, $(\bar\xi'_s,\bar\xi'_\ell)=D\Psi_0^{-1}(\xi'_s,\xi'_\ell)$, $(\bar\xi_s,\bar\xi_\ell)=DT^{-1}(\bar\xi'_s,\bar\xi'_\ell)$ and $(\xi^*_s,\xi^*_\ell)=D\Psi^{-1}_0(\bar\xi_s,\bar\xi_\ell)$,
then one obtains step by step
$$
\begin{aligned}
&|\xi'_\ell|\le (1+\mu\delta)|\xi'_s|, \quad && |\bar\xi'_\ell|\le (1+\mu\delta)\frac{\sigma_{s}}{\sigma_{\ell}}|\bar\xi'_s|,\\
&|\bar\xi_\ell|\le (1+\mu\delta)^2\frac{\sigma_{s}}{\sigma_{\ell}}|\bar\xi_s|, \quad && |\xi_\ell^*|\le (1+\mu\delta)^2\frac{\sigma_{s}\sigma'_{s}} {\sigma_{\ell}\sigma'_{\ell}}|\xi_s^*|.
\end{aligned}
$$
It follows from the condition that $(1+\mu\delta)^2\frac{\sigma_{s}\sigma'_{s}} {\sigma_{\ell}\sigma'_{\ell}}<1$. Therefore, $\|DF\|\le 1$ implies $\|D\mathscr{M}^{-1}F\|\le 1$.

For a map $F\in\mathfrak{F}_{R,M}$, let $F'$, $\bar F'$, $\bar F$ and $F^*$ be the maps such that $\mathcal{G}_{F'}=T\mathcal{G}_{F}$, $\mathcal{G}_{\bar F'}=\Psi'^{-1}_0 \mathcal{G}_{F'}$, $\mathcal{G}_{\bar F}=T^{-1}\mathcal{G}_{\bar F'}$ and $\mathcal{G}_{F^*}=\Psi^{-1}_0\mathcal{G}_{\bar F}$. Hence, $\|F'\|\le(1+\mu\delta  R)N+\mu\delta$, $\|\bar F'\|\le\frac{\sigma'_0}{\sigma'_{\ell}} \|F'\|$, $\|\bar F\|\le(1+\mu\delta R)\|\bar F'\|+\mu\delta$ and $\|F^*\|\le\frac{\sigma_0} {\sigma_{\ell}} \|\bar F\|$. Since $N\ge 1$,
\begin{equation}\label{positionofgraph}
\|F^*\|\le\frac{\sigma_0}{\sigma_{\ell}}\Big[\frac{\sigma'_0}{\sigma'_{\ell}} \Big((1+\mu\delta R)N+\mu\delta\Big)(1+\mu\delta)+\mu\delta\Big]<N.
\end{equation}
By the assumptions, we find $2(1+\mu\delta R)<\frac{\sigma'_{\ell}}{\sigma'_0}$ and $2(1+2\mu\delta(1+\mu\delta)) <\frac{\sigma_{\ell}}{\sigma_0}$. In this case, $\mathscr{M}^{-1}$ maps $\mathfrak{F}_{R,M}$ into itself.

For affine maps $F_1,F_2$: $E_s\to E_\ell$ with $\|DF_1\|,\|DF_2\| \le 1$, we have $F^*_1=\mathscr{M}^{-1}F_1$ and $F^*_2=\mathscr{M}^{-1}F_2$. To check $\|\mathscr{M}^{-1}F_1-\mathscr{M}^{-1}F_2\|$, we notice that a vector $\Delta v^*=v^*_2-v^*_1$ is mapped to $\Delta v$ by $M=T^{-1}\Psi'_0T\Psi_0$ through the procedure
$$
\Delta v^*\stackrel{\Psi_i}{\longrightarrow}\Delta\bar v\stackrel{T} {\longrightarrow}\Delta\bar v'\stackrel{\Psi'_0}{\longrightarrow}\Delta z'\stackrel{T^{-1}}{\longrightarrow}\Delta v.
$$
For $v^*_s\in E_s$, let $v^*_j=(v^*_s,F^*_j(v^*_s))+ v_0$ and $\Delta v^*=v^*_2-v^*_1$. If we set $\Delta v^*=(\Delta v^*_s,\Delta v^*_\ell)$, then $\Delta v^*_s=0$. Hence, we have  $|\Delta\bar v_\ell|\ge\frac{\sigma_{\ell}}{\sigma_0} |\Delta v^*_\ell|$ and $|\Delta\bar v_s|=0$. Since $T$ is close to identity, $|\Delta\bar v'_\ell| \ge(1-\mu\delta)\frac{\sigma_{\ell}}{\sigma_0}|\Delta v^*|$ and $|\Delta\bar v'_s|\le \mu\delta|\Delta\bar v'_\ell|$. Applying $\Psi'_0$ to $\Delta\bar v'$ we get
\begin{equation}\label{difference-0}
\begin{aligned}
|\Delta v'_\ell|&\ge\frac{\sigma'_{\ell}}{\sigma'_0}|\Delta\bar v'_\ell| \ge(1-\mu\delta) \frac{\sigma_{\ell}\sigma'_{\ell}}{\sigma_0\sigma'_0}|\Delta v^*|, \\
|\Delta v'_s|&\le\frac{\sigma'_{s}}{\sigma'_0}|\Delta\bar v'_s|\le \mu\delta \frac{\sigma'_{s}}{\sigma'_{0}}|\Delta\bar v'_{\ell}|\le \mu\delta \frac{\sigma'_{s}}{\sigma'_{\ell}}|\Delta v'_{\ell}|.
\end{aligned}
\end{equation}
Applying $T^{-1}$ to $\Delta v'$ and by assuming $(1+\mu\delta)\frac{\sigma'_{s}}{\sigma'_{\ell}}\le 1$ we get
\begin{equation}\label{difference}
\begin{aligned}
|\Delta v_s|&\le(1+\mu\delta)|\Delta v'_s|+\mu\delta|\Delta v'_\ell|\le 2\mu\delta|\Delta v'_\ell|\\
|\Delta v_\ell|&\ge (1-\mu\delta)|\Delta v'_\ell|-\mu\delta|\Delta v'_s|\ge (1-2\mu\delta) |\Delta v'_\ell|.
\end{aligned}
\end{equation}
Let $v_j=(v_{s,j},v_{\ell,j},v_0)=Mv^*_j$. Since $\|DF_j\|\le 1$ for $j=1,2$, we obtain from \eqref{difference} and \eqref{difference-0} that
$$
\begin{aligned}
|F_1(v_{s,1})-F_2(v_{s,1})|&\ge|F_1(v_{s,1})-F_2(v_{s,2})|-|F_2(v_{s,1})-F_2(v_{s,2})|\\
&\ge (1-\mu\delta)(1-4\mu\delta)\frac{\sigma_{\ell}\sigma'_{\ell}}{\sigma_0\sigma'_0} |F^*_1(v_{s,1})-F^*_2(v_{s,1})|.
\end{aligned}
$$
If we choose $v^*_{s,1}\in\{|v_s|\le R\}$ such that $|F^*_1(v^*_{s,1})-F^*_2(v^*_{s,1})| =\|F^*_1-F^*_2\|$ and if $v_{s,1}\in\{|v_s|\le R\}$, we obtain that
\begin{equation}\label{yasuo}
\|F_1-F_2\|\ge|F_1(v_{s,1})-F_2(v_{s,1})|\ge 2\|F^*_1-F^*_2\|,
\end{equation}
because the conditions of the proposition ensure $(1-\mu\delta)(1-4\mu\delta)\frac{\sigma_{\ell}\sigma'_{\ell}}{\sigma_0\sigma'_0}\ge 2$. We derive from Banach's fixed point theorem the existence and uniqueness of the fixed point $F_0\in\mathfrak{F}_{R,N}$ for $\mathscr{M}^{-1}$.

Hence, what remains to prove is $v_{s,1}\in\{|v_s|\le R\}$. Let $(v^*_s,v^*_u,v_0)=M^{-1}(v_s,v_u,v_0)$, one has
$$
\begin{aligned}
|v^*_s|\ge\frac{\sigma_0}{\sigma_{s}}\Big[&(1-\mu\delta)\frac{\sigma'_0}{\sigma'_{s}} \Big((1-\mu\delta)|v_s|-\mu\delta|v_\ell|-\mu\delta\Big)\\
&-\mu\delta\frac{\sigma_0}{\sigma_{\ell}}\Big((1+\mu\delta)|v_s|+\mu\delta|v_\ell| +\mu\delta\Big)-\mu\delta \Big].
\end{aligned}
$$
Therefore, for any $(v_s,v_u,v_0)$ with $|v_s|=R$, $|v_\ell|\le N$ and small $\delta>0$, its image $(v^*_s,v^*_u,v_0)=M^{-1}(v_s,v_\ell,v_0)$ satisfies the condition $|v^*_s|>\frac12\frac{\sigma_0\sigma'_0}{\sigma_{s}\sigma'_{s}}R\ge R$. Since it holds for all $(v_s,v_u,v_0)\in\mathcal{G}_F$ with $|v_s|=R$ that $|v_\ell|\le N$, we see that $\pi_sM^{-1}\mathcal{G}_F\supset\{|v_s|\le R\}$. It implies that $v_{s,1}\in\{|v_s|\le R\}$.
\end{proof}

We also study the set $\mathfrak{G}_{R,N}$ of affine maps $G:E_\ell\to E_s$, which is defined in the same way as $\mathfrak{F}_{R,N}$. The map $M$ induces a map $G\to\mathscr{M}G$. Similar to the proof of Lemma \ref{lem7.3}, we see the existence and uniqueness of the fixed point $G_0=\mathscr{M}G_0\in\mathfrak{G}_{R,N}$. Both graphs intersect at one point $v\in\mathcal{G}_F\cap\mathcal{G}_G$ which is the fixed point of $M$. Recall the definition of $M$, the line passing through $\Psi(v_0+v)$ and the origin intersects the affine manifold $L'$ at a point $\frac 1{\sigma_0}T\Psi(v_0+v)$ which is mapped by $\Psi'$ to a point lying on the line connecting $v_0+v$, namely, $\Psi'\Psi(v_0+v)$ is a point lying on the line passing through $v_0+v$ and the origin, i.e. $v_0+v$ is an eigenvector of $\Psi'\Psi$. Hence, to complete the proof of Proposition \ref{productofeigenvalue}, we only need to localize $v$ and get an estimate on the eigenvalue.

Let $R=2$, we consider a map $F$ with $\|F\|\le 2\mu\delta$. Repeating the procedure to get \eqref{positionofgraph} we have
$$
\|\mathscr{M}^{-1}F\|\le\frac{\sigma_0}{\sigma_{\ell}}\Big[\frac{\sigma'_0} {\sigma'_{\ell}}\Big((1+2\mu\delta)2\mu\delta+\mu\delta\Big)(1+\mu\delta)+\mu\delta\Big] <2\mu\delta,
$$
i.e. $\mathscr{M}^{-1}$ maps $\mathfrak{F}_{2,2\mu\delta}$ into itself. It implies $|v_\ell|\le 2\mu\delta$. Similarly, $\mathscr{M}$ maps $\mathfrak{G}_{2,2\mu\delta}$ into itself either, which implies $|v_s|\le 2\mu\delta$. So we have $|v|\le 2\mu\delta$. Let $v^*_0=\frac{v_0+v}{|v_0+v|}$, then $\langle v_0,v^*_0\rangle\ge 1-2\mu\delta$.

To study the eigenvalue, we use the relation $v'_0+\bar v'=T\Psi_0(v_0+v)= \Psi'^{-1}_0T(v_0+v)$, since both $G_0$ and $F_0$ are invariant for $\mathscr{M}$. From the relation $v'_0+\bar v'=T\Psi_0(v_0+v)$ we see that $|\pi_s\bar v'|\le 2\mu\delta$, from the relation $v'_0+\bar v'=\Psi'^{-1}_0T(v_0+v)$ we see that $|\pi_\ell\bar v'|\le 2\mu\delta$, i.e. $|\bar v'|\le 2\mu\delta$. By the definition of $v'_0+\bar v'$ and $T$, some $\nu\in[-3\mu\delta,3\mu\delta]$ exists such that $\Psi(v_0+v)=(1+\nu)\sigma_0(v'_0+\bar v')$. Because $T^{-1}\Psi'_0(v'_0+\bar v')=v_0+v$, one has $\Psi'(v'_0+\bar v')=(1+\nu')\sigma'_0(v_0+v)$ with some $\nu'\in[-3\mu\delta,3\mu\delta]$. It follows that
$$
\Psi'\Psi(v_0+v)=(1+\nu)(1+\nu')\sigma_0\sigma'_0(v_0+v),
$$
namely, we have $\sigma_0^*=(1+O(\delta))\sigma_0\sigma'_0$.
\end{proof}

\noindent{\bf Remark}. It is crucial in Proposition \ref{productofeigenvalue} that the number $\delta$ is independent of size of the eigenvalues of $\Psi$ and $\Psi'$. In the application, half eigenvalues of $d\Phi_{E,0,-r}$ and of $d\Phi_{E,\pm r,0}$ approach infinity while the other half approach $0$ as $|E|\to 0$.

\section{Applications}\label{sec.8}
The study of nearly integrable Hamiltonian systems was thought by Poincar\'e to be a fundamental problem of dynamics. Soon after Kolmogorov's theorem was established, Arnold discovered the dynamical instability in \cite{A64} and proposed a conjecture about nearly integrable Hamiltonians,
\begin{equation}\label{EqHam}
H(x,y)=h(y)+\epsilon P(x,y),\qquad (x,y)\in\mathbb{T}^d\times\mathbb{R}^d
\end{equation}
now it is named after him as the conjecture of Arnold diffusion

{\bf Conjecture} (\cite{A66}): {\it The ``general case" for a Hamiltonian system \eqref{EqHam} with $d\ge 3$ is represented by the situation that for an arbitrary pair of neighborhood of tori $y=y'$, $y=y''$, in one component of the level set $h(y)=h(y')$ there exists, for sufficiently small $\eps$, an orbit intersecting both neighborhoods.}

In the study of Arnold diffusion, especially after the diffusion in {\it a priori} unstable case has been solved in the works \cite{CY04,DLS,Tr,B08,CY09,Z11}, the main difficulty is to cross double resonance, as foreseen by Arnold in \cite{A66}. The study of the problem was initiated by Mather \cite{M04,M09} and it has been solved in \cite{C17a,C17b,CZ16} in a way by skirting around the double resonant point, where some abstruse theories was involved. It is of great interest to explore a way easier to visualize, to understand.

With Theorem \ref{mainresult}, we are surprised to see that the method for {\it a priori} unstable case still works for the construction of diffusion orbits passing through double resonance, since along the prescribed resonant path there still exists a NHIC with compound type homology class passing double resonance. It is not necessary to switch from the path of compound type homology class to a path of single homology class, as suggested by Mather. To this end, a special case of Theorem \ref{mainresult} for $n=2$ with the type of single homology class was announced in \cite{Mar,KZ} without complete proof.

Along a path of compound type homology class, there exist two pairs of homoclinic orbits $\{z_1^\pm(t),z_2^\pm(t)\}$ associated with positive integers $k_1,k_2$ such that class of the path is $k_1[z_1^+(t)]+k_2[z_2^+(t)]$, and $\langle [z_1^+(t)],[z_2^+(t)]\rangle>-\|[z_1^+(t)]\|\|[z_2^+(t)]\|$. In this case, the condition ({\bf H3}) holds. Therefore, there is a $C^1$-normally hyperbolic cylinder passing through double resonance, along which diffusion orbits are constructed by the method developed in \cite{CY04,CY09}. What is more, it reminds us of a possible way to cross multiple resonance in the systems with arbitrarily many degrees of freedom. We shall discuss it in another paper.

\noindent{\bf Acknowledgement}. The authors are supported by NNSF of China (No.11790272 and No.11631006). They thank J. Xue and J. Zhang for helpful discussions.

\end{document}